\documentclass[11pt,a4paper,twoside,
]{article}

\usepackage[french,english]{babel}
\usepackage{amsfonts}
\usepackage{amsmath,amssymb,amscd}
\usepackage{mathrsfs}  
\usepackage{amsthm}
\usepackage{a4wide}   
\usepackage[T1]{fontenc} 
\usepackage{newlfont} 
\usepackage{color}

\usepackage[
backref=page,bookmarksopen=true]{hyperref} 

\usepackage[all]{xy}  

\usepackage[applemac]{inputenc}       

\def\choixcompteur{subsection}
\newtheorem{theo}[\choixcompteur]{Theorem}

\newtheorem{prop}[\choixcompteur]{Proposition}
\newtheorem{lemm}[\choixcompteur]{Lemma}
\newtheorem{coro}[\choixcompteur]{Corollary}

\theoremstyle{definition}

\newtheorem{defi}[\choixcompteur]{Definition}

\newtheorem{exem}[\choixcompteur]{Example}

\newtheorem{rema}[\choixcompteur]{Remark}
\newtheorem{remas}[\choixcompteur]{Remarks}
\newtheorem{conses}[\choixcompteur]{Consequences}

\newtheorem*{exem*}{Example}
\newtheorem*{exems*}{Examples}
\newtheorem*{exam*}{Example}
\newtheorem*{exams*}{Examples}
\newtheorem*{rema*}{Remark}
\newtheorem*{remas*}{Remarks}
\newtheorem*{NB}{N.B}

\theoremstyle{definition}
\newtheorem*{defi*}{Definition}
\newtheorem*{defiprop*}{Definition-Proposition}

\theoremstyle{plain}
\newtheorem*{prop*}{Proposition}
\newtheorem*{lemm*}{Lemma}
\newtheorem*{coro*}{Corollary}
\newtheorem*{theo*}{Theorem}

 \def\cdr@enoncedef{%
 \newenvironment{enonce*}[2][plain]%
 {\let\cdrenonce\relax \theoremstyle{##1}%
 \newtheorem*{cdrenonce}{##2}%
 \begin{cdrenonce}}%
 {\end{cdrenonce}}   }%

 \AtBeginDocument{%
    \cdr@enoncedef}

\def\cf{{\it cf.\/}\ }
\def\ie{{\it i.e.\/}\ }
\def\eg{{\it e.g.\/}\ }
\def\etc{{\it etc.\/}\ }
\def\lc{{\it l.c.\/}\ }
\def\LC{{\it loc. cit.\/}\ }
\def\rest{\,\rule[-1.5mm]{.1mm}{3mm}_}  

\def\into{\hookrightarrow}
\def\vect{\overrightarrow}

\def\parni{\par\noindent}

\def\ed{ editor}
\def\eds{ editors}
\def\ede{ editor} 

\def\N{{\mathbb N}}    
\def\Z{{\mathbb Z}}
\def\Q{{\mathbb Q}}
\def\R{{\mathbb R}}
\def\C{{\mathbb C}}

\def\A{{\mathbb A}}
\def\M{{\mathbb M}}

\newcommand{\g}[1]{\mathfrak{#1}} 

\def\qa{\alpha}     
\def\qb{\beta}
\def\qd{\delta}
\def\qe{\varepsilon}
\def\qf{\varphi}
\def\qg{\gamma}

 \def\qk{\kappa}
 \def\ql{\lambda}
\def\qm{\mu}
\def\qn{\nu}
\def\qo{\omega}
\def\qp{\pi}

\def\qt{\tau}

\def\qx{\xi}

\def\QD{\Delta}
\def\QF{\Phi}
\def\QG{\Gamma}
\def\QL{\Lambda}
\def\QO{\Omega}

\def\QS{\Sigma}

\def\sha{{\mathcal A}}   

\def\shi{{\mathcal I}}

\def\shm{{\mathcal M}}

\def\sho{{\mathcal O}}
\def\shp{{\mathcal P}}
\def\shq{{\mathcal Q}}

\def\shs{{\mathcal S}}
\def\sht{{\mathcal T}}
\def\shu{{\mathcal U}}
\def\shv{{\mathcal V}}

\def\SHF{{\mathscr F}}

\def\SHI{{\mathscr I}}
\def\SHJ{{\mathscr J}}

\def\SHM{{\mathscr M}}

\begin{document}

\title{Almost split Kac-Moody groups over ultrametric fields}
\author{Guy Rousseau}

\date{July 7, 2015}

\maketitle

%








\begin{abstract}
For a split Kac-Moody group $G$ over an ultrametric field $K$, S. Gaussent and the author defined an ordered affine hovel on which the group acts; it generalizes the Bruhat-Tits building which corresponds to the case when $G$ is reductive.
 This construction was generalized by C. Charignon to the almost split case when $K$ is a local field.
 We explain here these constructions with more details and prove many new properties \eg that the hovel of an almost split Kac-Moody group is an ordered affine hovel, as defined in a previous article.
 \footnote{keywords: Hovel, building, Kac--Moody group, almost split, ultrametric  local field}\footnote{Classification: 20G44, 20E42
(primary), 17B67, 20G25, 22E65, 51E24}
\end{abstract}


\setcounter{tocdepth}{1}    
\tableofcontents


\section*{Introduction}
\label{seIntro}

\par Split Kac-Moody groups over ultrametric local fields were first studied by H. Garland in the case of loop groups \cite{Gd-95}. In \cite{Ru-06} we constructed a "microaffine" building for every split Kac-Moody group over a field $K$ endowed with a non trivial real valuation.
 It is a (non discrete) building with the good usual properties, but it  looks not like a Bruhat-Tits building, rather like the border of this building in its Satake (or polyhedral) compactification.

 \par A more direct generalization of the Bruhat-Tits construction was made by S. Gaussent and the author, in the case where the residue field of $K$ contains $\C$ \cite{GR-08}. This enabled us to deduce interesting consequences in representation theory.
 In \cite{Ru-11} the restriction about the residue field was removed. So, for a split Kac-Moody group $G$ over $K$, one can build an hovel $\SHI$ on which $G$ acts.
  As for the Bruhat-Tits building, $\SHI$ is covered by apartments corresponding to split maximal tori; but it is no longer true that any two points are in a same apartment (this corresponds to the fact that the Cartan decomposition fails in $G$). This is the reason why the word "building" was changed to "hovel".
  Nevertheless this hovel has interesting properties: it is an ordered affine hovel as defined in \cite{Ru-10}. As a consequence the residues in each point of $\SHI$ are twin buildings, there exist on $\SHI$ a preorder invariant by $G$ and, at infinity, we get twin buildings and two microaffine buildings. These are the twin buildings of $G$ introduced by B. R\'emy \cite{Ry-02a} and the microaffine buildings of \cite{Ru-10}.

  \par Cyril Charignon undertook the construction of hovels for almost split Kac-Moody groups \cite{Cn-10}, \cite{Cn-10b}. Actually he considered the disjoint union of the hovel and of some hovels at infinity called fa\c{c}ades. This union is called a bordered hovel, it looks like the Satake compactification of a Bruhat-Tits building; in addition to the main hovel it contains the microaffine buildings.
   He elaborates an abstract theory of bordered hovels associated to a generating root datum, a valuation and a family of parahoric subgroups. He proves an abstract descent theorem and succeeds in using it to build a bordered hovel associated to an almost split Kac-Moody group over a field endowed with a discrete valuation and a perfect residue field. As a corollary the microaffine buildings are also defined in this situation.

   \par In this article we give more details about these constructions and improve many results (see below some details about each section).
    In particular the fixed point theorem in $\R-$buildings proved recently by K. Struyve \cite{Se-11} enables us to prove the existence of bordered hovels in new cases (with a non discrete valuation).
    
    \par The essential new  result we get (from \cite{Cn-10b} and the present article) is the existence of a nice (bordered) hovel for each almost split Kac-Moody group over an ultrametric local field, with a strongly transitive action of the group and a good enclosure map (Theorem \ref{6.9} and Proposition\ref{6.10a}.3). 
    This is the powerful tool which enables S. Gaussent, N. Bardy-Panse and the author to define spherical Hecke algebras or Iwahori-Hecke algebras associated to any almost split Kac-Moody group over an ultrametric local field: \cite{GR-12},  \cite{BGR15}.
    The non split case involves unequal parameters.

    \par In section \ref{s1} we explain the general framework of our study: abstract generating root data, their associated twin buildings and split Kac-Moody groups (as defined by J. Tits \cite{T-87b}).

    \par Section \ref{s1b}  is devoted to B. R\'emy's theory of almost split Kac-Moody groups \cite{Ry-02a}. We improve a few results, \eg on geometric realizations of the associated twin buildings and on imaginary relative root groups.

   \par In section \ref{s2} we define the affine apartments associated to a valuation of an abstract root datum. We explain the interesting subsets or filters of subsets inside them (facets, sectors, chimneys, enclosures,\ldots) and embed them in their bordered apartments. There are several possible choices for these apartments, their imaginary roots or walls and for the fa\c{c}ades at infinity. So this leads to several choices for all these objects and none of them is better in all circumstances.

  \par Section \ref{s3} is devoted to Charignon's abstract construction of the bordered hovel associated to a good family of parahoric subgroups in a valuated root datum \cite{Cn-10}, \cite{Cn-10b}. We select two other conditions he considered for parahoric families and a new third one to define what is a very good family. Then we are able to generalize abstractly the constructions of \cite{GR-08} and prove that the abstract hovel we get is an ordered affine hovel in the sense of \cite{Ru-10} (slightly generalized). This involves an enclosure map ($cl^{\QD^{ti}}_\R$) which gives (too) small enclosures.

   \par In section \ref{s5} we mix these abstract results and the results of \cite{Ru-11} to define the bordered hovel of a split Kac-Moody group over a field endowed with a non trivial real valuation. One of the problems is to extend the results to general apartments, neither essential as in Charignon's or R\'emy's works, nor associated directly to the group as in \cite{Ru-11}. We prove that these bordered hovels are functorial, uniquely defined (in the sense that the very good family of parahorics is unique) and that their residue twin buildings are associated to a generating root datum.

   \par These results are generalized to almost split Kac-Moody groups in section \ref{s6}. We explain the abstract descent theorem of Charignon (generalizing the analogous theorem of F. Bruhat and J. Tits \cite{BtT-72}). To apply it to an almost split Kac-Moody group $G$, we need the same condition as for reductive groups: $G$ is assumed to become quasi split over a finite tamely ramified Galois extension $L/K$ \cite{Ru-77}. There is no need for another condition, even for a non discrete valuation, as is now clear from Struyve's work \cite{Se-11}.
   We explain Charignon's results in this almost split case and generalize them to more general apartments. So we get a bordered hovel and we prove that the hovel inside it (its main fa\c{c}ade) is an ordered affine hovel (in the sense of \cite{Ru-10}) with an enclosure map ($cl^{_K\QD^{r}}_{K}$) which is better than the one in section \ref{s3} (but perhaps still not the best one).
   
   \par The consideration of bordered hovels (of rather general types) leads to many technical complications and many similar definitions (\eg of enclosure maps). This seams unavoidable.
   The final results (for almost split Kac-Moody groups) are simpler, especially when speaking only on hovels (without boundaries).

\section{Root data and split Kac-Moody groups}\label{s1}

\par Most of the following definitions or results (and some other useful ones) may be found in \cite{Ry-02a}, \cite{Ru-06},  \cite{Ru-10} or \cite{Ru-11}.

\subsection{Root generating systems}\label{1.1}

\par
\par {\quad\bf1)} We consider a  Kac-Moody matrix  (or generalized Cartan matrix)
 $\M = (a_{i,j})_{i,j\in I}$ , with rows and columns indexed by a finite set $I$.
Let $Q$ be a free $\Z-$module with basis $(\alpha_i)_{i\in I}$ and $Q^+=\sum_{i\in I}\,\Z_{\geq{}0}.\alpha_i\subset Q$, $Q^-=-Q^+$.
The {\it (vectorial) Weyl group} $W^v$ associated to $\M$ is a Coxeter group with generating system the set $\QS=\{\,s_i\,\mid\,i\in I\}$ of automorphisms of $Q$ defined by $s_i(\alpha_j) = \alpha_j - a_{i,j} \alpha_i$.
The associated system of {\it real roots} is $\QF=\{\,w(\qa_i)\mid w\in W^v,i\in I\,\}$ \cite{K-90};
it is a real root system (with free basis $(\alpha_i)_{i\in I}$) in the sense of \cite{MP-89} or \cite{MP-95}, see also \cite{By-96} and \cite{He-91}.
 If $\alpha \in \Phi$, then $s_{\alpha} =
w.s_i.w^{-1} $ is well defined by $\alpha$  independently of the choice of  $w$ and $i$ such that $\alpha = w(\alpha_i)$.
We say that we are in the {\it classical case} when $W^v$ is finite, then $\M$ is a Cartan matrix and $\QF$ a root system in the sense of \cite{B-Lie}.
 For $J\subset I$,  $\M(J) = (a_{i,j})_{i,j\in J}$ is a Kac-Moody matrix; with obvious notations, $Q(J)$ is a submodule of $Q$ and $\QF^m(J)=\QF\cap Q(J)$ the root system associated to $\M(J)$, its Weyl group is $W^v(J)=\langle s_i\mid i\in J\rangle$.

\medskip
\par {\bf2)} A {\it root generating system } (or RGS) \cite{By-96} will be (for our purpose) a quadruple
${\mathcal S}=(\M,Y,({\overline{\alpha_i}})_{i\in I},({\alpha}{_i^\vee})_{i\in I})$ where $\M$ is a Kac-Moody matrix, $Y$ a free $\mathbb Z-$module of finite rank  $n$,   $({\overline{\alpha_i}})_{i\in I}$ a family (of {\it simple roots}) in its dual $X=Y^*$ and
 $({\alpha}{_i^\vee})_{i\in I}$ a family (of {\it simple coroots}) in $Y$. These data have to satisfy the condition: $\;$
 $a_{i,j} = {\overline{\alpha_j}}({\alpha}{_i^\vee}) $.

 \par The Weyl group $W^v$ acts on $X$ (and dually on $Y$) by $s_i(\chi) = \chi - {\chi}(\alpha{_i^\vee}){\overline{\alpha_i}}$.

 \par We say that  ${\mathcal S}$ is {\it free} (resp. {\it adjoint})  if $({\overline{\alpha_i}})_{i\in I}$ is free in (resp. generates) $X$. For example the minimal adjoint RGS $\shs_{\M m}=\shs^{lad}=(\M,Q^*,({{\alpha_i}})_{i\in I},({\alpha}{_i^\vee})_{i\in I})$ (with an obvious definition of the ${\alpha}{_i^\vee}$) is free and adjoint.

\par There is a group homomorphism $bar: Q\to X$ , $\alpha\mapsto {\overline\alpha}$ such that $bar(\alpha_i)={\overline{\alpha_i}}$; it is $W^v-$equivariant and $\qf^{ad}=bar^*:Y\to Q^*$ is a commutative extension of RGS $\shs\to\shs_{\M m}$ \cite[1.1]{Ru-11}.
 When ${\mathcal S}$ is free, $Q$ is identified with ${\overline Q}=bar(Q)\subset X$.

\par  For $J\subset I$,  ${\mathcal S}(J)=(\M(J),Y,({\overline{\alpha_i}})_{i\in J},({\alpha}{_i^\vee})_{i\in J})$ is also a RGS.

\medskip
\par {\bf3)} The {\it complex Kac-Moody algebra} ${\mathfrak g}_{\mathcal S}$ associated to $\mathcal S$ is generated by the {\it Cartan subalgebra}  ${\mathfrak h}_{\mathcal S}=Y\otimes_\Z\mathbb C$ and elements $(e_i,f_i)_{i\in I}$ with well known relations \cite{K-90}.
This Lie algebra has a gradation by $Q$ : $\g g_{\shs}=\g h_\shs\oplus(\oplus_{\alpha\in\Delta}\;\g g_\alpha)$ where $\Delta\subset Q\setminus\{0\}$ is the {\it root system} of $\g g_\shs$ or of $\M$.

\par We have $\g h_\shs=(\g g_\shs)_0$, $\g g_{\alpha_i}=\C e_i$, $\g g_{-\alpha_i}=\C f_i$ and $\QF\subset\QD$ (as $\QD$ is $W^v-$stable).

\par If $\Delta^+=\Delta\cap Q^+$ (resp. $\Delta^-=-\Delta^+$) is the set of {\it positive} (resp. {\it negative}) roots, then $\Delta=\Delta^-\bigsqcup\Delta^+$.
We set $\QF^\pm=\QF\cap\QD^\pm=-\QF^\mp$. The {\it imaginary} roots are the roots in $\QD\setminus\QF=\QD_{im}$; we set $\QD_{re}=\QF$, $\QD^\pm_{re}=\QF^\pm{}$ and $\QD^\pm_{im}=\QD_{im}\cap\QD^\pm{}$ .

\par  For $J\subset I$,  $\g g_{\shs(J)}$ is a Lie subalgebra of ${\mathfrak g}_{\mathcal S}$ and $\QD^m(J)=\QD\cap Q(J)$ a subroot system of $\QD$.

\par In the classical case, ${\mathfrak g}_{\mathcal S}$ is a reductive finite dimensional Lie algebra and $\QD_{im}=\emptyset$.

\subsection{Vectorial apartments}\label{1.1b}

\par {\quad\bf1)} We consider  a free RGS ${\mathcal S}=(\M,Y,({{\alpha_i}})_{i\in I},({\alpha}{_i^\vee})_{i\in I})$ and the real vector space $V=V^\shs=Y\otimes\R=$ Hom$_\Z(X,\R)$.
 Each element in $X$ or in $Q\subset X$ induces a linear form on $V$. A {\it vectorial wall} in $V$ is the kernel of some $\qa\in\QF$.
The {\it positive} (resp. {\it negative}) {\it fundamental chamber} in $V$ is
$C_+^{v}=\{\,v\in V\mid \qa_i(v)>0,\forall i\in I\,\}$ (resp. $C_-^{v}=-C_+^{v}$).
If $J\subset I$, let $F_+^{v}(J)=\{\,v\in V\mid \qa_i(v)=0,\forall i\in J,\qa_i(v)>0,\forall i\in I\setminus J\,\}$ and $F_-^{v}(J)=-F_+^{v}(J)$, they are cones in $V$.
Then the closed positive fundamental chamber is 
$\overline{C_+^{v}}=\bigsqcup_{J\subset I}\;F_+^{v}(J)$ and symmetrically for $\overline{C_-^{v}}$.

\par The Weyl group $W^v$ acts faithfully on $V$, we identify $W^v$ with its image in $GL(V)$. For $w\in W^v$ and $J\subset I$, $wF_+^{v}(J)$ (resp. $wF_-^{v}(J)$) is called a {\it positive} (resp. {\it negative}) {\it vectorial facet of (vectorial) type $J$}. The fixer or stabilizer of $F_\pm^{v}(J)$ is $W^v(J)$;
 if this group is finite we say that $J$ or $wF_\pm^{v}(J)$ is {\it spherical}.
These positive facets are disjoint and their union $\sht_+$ is a cone: the {\it Tits cone}.
The inclusion in the closure gives an order relation on these facets.
 The {\it star} of a facet $F^v$ is the set $F^{v*}$ of all facets $F^v_1$ such that $F^v\subset\overline{F^v_1}$.

\par The properties of the action of $W^v$ on the set of positive facets allows one to identify this poset (or to be short $\sht_+$) with the Coxeter complex of $(W^v,\QS)$. The interior $\sht_+^\circ$ of $\sht_+$ is the union of its spherical facets, it is also a non empty convex cone. 
 The symmetric results for $\sht_-=-\sht_+$ are also true.

\par We call $\A^{v}=\sht_+\cup\sht_-$ the {\it  vectorial fundamental twin apartment} associated to $\shs$ and set $\vect{\A^{v}}=V$ (vector space generated by $\A^{v}$).
A {\it generic subspace} of $\A^{v}$ is an intersection of $\A^{v}$ with a vector subspace of $\vect{\A^{v}}$ which meets the interior of $\A^{v}$; for example a wall is a generic subspace.
An {\it half-apartment} in $\A^{v}$ is the intersection with $\A^{v}$ of one of the two closed half-spaces of $\vect{\A^{v}}$ limited by a wall.

\par In $V$ the subspace $V_0=F_\pm^{v}(I)=\bigcap_{i\in I}\,$Ker$(\qa_i)$ (trivial facet) is the intersection of all vectorial walls. Acting by translations it stabilizes all facets and the two Tits cones. So  the essentialization of $V$ or $\A^v$ is $V^e=V/V_0$ or $\A^{ve}=\A^v/V_0$.

\par One may generalize these definitions to the case when the chamber $C^v_+$ defined by  $({\overline{\alpha_i}})_{i\in I}$ (for a non free RGS) is non empty in $V=Y\otimes\R$ \cite[p. 113, 114]{By-96} but we shall avoid this.

\medskip
\par {\bf2)} The smallest example for $V$ associated to $\M$ and $\QF\subset Q$ corresponds to $\shs=\shs_{\M m}$.
 Then $V=V^q=V^\M=Q^*\otimes\R=$ Hom$_\Z(Q,\R)$. In the above notations we add an exponent $^q$ to all names.
  We get thus the {\it essential vectorial fundamental twin apartment} $\A^{vq}$. Actually $V^q$ and  $\A^{vq}$ are canonically the essentializations of any $V$ or $\A^v$ in 1) : $V^e=V^q$ and $\A^{ve}=\A^{vq}$

  \medskip
\par {\bf3)} If $\shs$ is a given free RGS, we shall write $V^x=V^\shs$ and add an exponent $^x$ to all names in 1) \eg $V^q=V^{xe}=V^x/V^x_0$ and $\A^q=\A^{xe}=\A^x/V^x_0$.
 We get thus  the {\it normal vectorial fundamental twin apartment} $\A^{vx}$.

 \par If $\shs$ is a given (non necessarily free) RGS, we may consider the free RGS $\shs^l$ of \cite[1.3d]{Ru-11}: ${\mathcal S}^l=(\M,Y^{xl},({{\alpha_i^{xl}}})_{i\in I},({\alpha}{_i^{xl\vee}})_{i\in I})$ with $Y^{xl}=Y\oplus Q^*$, $\qa_i^{xl}=\overline\qa_i+\qa_i\in X^{xl}=X\oplus Q$ and ${\alpha}{_i^{xl\vee}}={\alpha}{_i^{\vee}}\in Y\subset Y^{xl}=Y\oplus Q^*=(X\oplus Q)^*$.
 Then $V^{xl}=$ Hom$_\Z(X\oplus Q,\R)$ and we add an exponent $^{xl}$ to all names in 1).
 We get thus  the {\it extended vectorial fundamental twin apartment} $\A^{vxl}$.



  \medskip
\par {\bf4)} More generally we may consider a quadruple as in \cite[1.1]{Ru-10}:
$(V,W^v,({{\alpha_i}})_{i\in I},({\alpha}{_i^\vee})_{i\in I})$ with $\qa_i$ free in $V^*$ and $a_{i,j} = {{\alpha_j}}({\alpha}{_i^\vee}) $ hence $\QF\subset Q\subset V^*$.
The same things as in 1) (\eg $F_\pm^{v}(J)$, $\A^{v}_\pm{}=\sht_\pm{}$, ..) may be defined in $V$ and we have $V^q=V/V_0$ for $V_0=\bigcap_{i\in I}$ Ker$\qa_i$.
 For example we may take for $V$ a quotient of a $V^\shs$ as in 1) by any subspace $V_{00}$ of $V_0$.
\medskip
\par We get thus many geometric realizations of the Coxeter complex of $(W^v,\QS)$.

 \subsection{The split Kac-Moody group $\g G_\shs$}\label{1.2}

 \par As defined by J. Tits \cite{T-87b}, this group $\g G_\shs$ is a functor from the category of (commutative) rings to the category of groups.

 \par One considers first the torus $\g T_\shs=\g T_Y=\g{Spec}(\Z[X])$ with character group $X(\g T_\shs)=X$ and cocharacter group $Y(\g T_\shs)=Y$. For any ring $R$, $\g T_Y(R)=Y\otimes_\Z R^*=$ Hom$_\Z(X,R^*)$. 
 Then the group $\g G_\shs(R)$ is generated by $\g T_\shs(R)$ and elements $\g x_\qa(r)$ for $\qa\in\QF$ and $r\in R$; for the precise relations see \cite{T-87b}, \cite{Ry-02a} or \cite{Ru-11}.

 \par Actually $\g T_\shs$ is a sub-group-functor of $\g G_\shs$, the {\it standard maximal split subtorus}.
 For $\qa\in\QF$, there is an injective homomorphism $\g x_\qa:\g{Add}\to\g G_\shs$, $r\in R\mapsto \g x_\qa(r)$; the sub-group-functor of
$\g G_\shs$ image of $\g x_\qa$ is written $\g U_\qa$.
  The {\it standard positive} (resp. {\it negative}) {\it maximal unipotent subgroup} is the sub-group-functor $\g U^\pm_\shs$ such that, for any ring $R$, $\g U^+_\shs(R)$ (resp. $\g U^-_\shs(R)$) is generated by  all $\g U_\qa(R)$ for $\qa\in\QF^+$ (resp. $\qa\in\QF^-$); it depends actually only on $\M$, not on $\shs$.
  Then the {\it standard positive} (resp. {\it negative}) {\it Borel subgroup} is the semi-direct product $\g B^+_\shs=\g T_\shs\ltimes \g U^+_\shs$ (resp. $\g B^-_\shs=\g T_\shs\ltimes \g U^-_\shs$).

  \par The construction of $\g G_\shs$ uses a $Q-$graded $\Z-$form $\shu_{\shs\Z}$ of the universal enveloping algebra of $\g g_\shs$, we call it the {\it Tits enveloping algebra} of $\g G_\shs$ over $\Z$.
   It is a filtered $\Z-$bialgebra; the first term of its filtration is $\Z\oplus\g g_{\shs\Z}$, where $\g g_{\shs\Z}$ is a $\Z-$form of the Lie algebra $\g g_\shs$.
   There is a functorial adjoint representation $Ad: \g G_\shs\to\g{Aut}(\shu_{\shs\Z})$, see \cite{Ry-02a} and/or \cite{Ru-11} for details.
    In the classical case $\g G_\shs$ is a reductive group and $\shu_{\shs\Z}$ is often called the Kostant's $\Z-$form. By analogy with this case we define the {\it reductive rank} (resp. {\it semi-simple rank}) of $\g G_\shs$ or $\shs$ as $rrk(\shs)=n=dim(X)$ (resp. $ssrk(\shs)=\,\vert I\vert$\,); there is no a priori inequality between these two ranks.

 \par In the following we shall almost always consider a field $K$ and  restrict the above functors to the category $\shs ep(K)$ of algebraic separable field extensions of $K$ contained in a given separable closure $K_s$.
 The groups associated to $K$ by these functors are then written with roman letters: $G_\shs=\g G_\shs(K)$, $T_\shs=\g T_\shs(K)$, $U_\qa=\g U_\qa(K)$, $x_\qa: K\to U_\qa\subset G_\shs$, etc.
 We set also $\shu_{\shs K}=\shu_{\shs\Z}\otimes_\Z\,K,\cdots$. We shall sometimes forget the subscript $_\shs$.



   \begin{defi}\label{1.3}  \cf \cite{Ru-06} A {\it root datum of type} a (real) root system $\QF$ is a triple $(G,(U_\qa)_{\qa\in\QF},Z)$ where $G$ is a group and $Z$, $U_\qa$ (for $\qa\in\QF$) are subgroups of $G$, satisfying:

   \par (RD1) For all $\qa\in\QF$, $U_\qa$ is non trivial and normalized by $Z$.

   \par (RD2) For each prenilpotent pair of roots $\{\qa,\qb\}$, the commutator group $[U_\qa,U_\qb]$ is included in the group generated by the groups  $U_\qg$ for $\qg=p\qa+q\qb\in\QF$ and $p,q\in\Z_{>0}$.
   \par\qquad (there is a finite number of such roots $\qg$, as $\{\qa,\qb\}$ is supposed prenilpotent).

   \par (RD3) If $\qa\in\QF$ and $2\qa\in\QF$, then $U_{2\qa}\subsetneqq U_\qa$.

   \par (RD4) For all $\qa\in\QF$ and all $u\in U_\qa\setminus\{1\}$, there exist $u',u''\in U_{-\qa}$ such that $m(u):=u'uu''$ conjugates $U_\qb$ into $U_{s_\qa(\qb)}$ for all $\qb\in\QF$. Moreover, for all $u,v\in U_\qa\setminus\{1\}$, $m(u)Z=m(v)Z$.

 \par (RD5) If $U^+$ (resp. $U^-$) is the group generated by the groups $U_\qa$ for $\qa\in\QF^+$ (resp. $\qa\in\QF^-$), then $ZU^+\cap U^-=\{1\}$.

\par The root datum is called {\it generating} if moreover:

\par (GRD) The group $G$ is generated by $Z$ and the groups $U_\qa$ for $\qa\in\QF$.

   \end{defi}

\begin{remas}\label{1.3b} a) This definition is given for a general (real) root system $\QF$. For the system $\QF$ of \ref{1.1} the axiom (RD3) is useless as $\QF$ is reduced. In the classical case (\ie for a finite root system) this is equivalent to the definition of "donn\'ee radicielle de type $\QF$" in \cite {BtT-72}. In general a generating root datum is the same thing as a  "donn\'ee radicielle jumel\'ee enti\`ere" as defined in \cite[6.2.5]{Ry-02a}.

\par b) Actually $Z$ has to be the intersection of the normalizers of the groups $U_\qa$: \cite[1.5.3]{Ry-02a}, see also \cite[7.84]{AB-08}. So one may forget $Z$ in the datum, as in \cite{T-92a} or \cite[10.1.1]{Cn-10b}.

\par c) Even in the classical case, the notion of root datum (of type a root system) is more precise than the notion of RGD-system (of type a Coxeter system) defined in \cite[8.6.1]{AB-08} (which is the same thing as "donn\'ee radicielle jumel\'ee" defined in \cite[1.5.1]{Ry-02a}, see also \cite{T-92a}).
 The "roots" of $(W^v,\QS)$ are in one to one correspondence with the non-divisible roots in $\QF$.
 So, if $(G,(U_\qa)_{\qa\in\QF},Z)$ is a generating root datum, then $(G,(U_\qa)_{\qa\in\QF_{nd}},Z)$ is a RGD-system; the difference is that axiom (RGD1) is less precise than (RD2): it allows $p$ and $q$ to be in $\R_{>0}$.

 \par Root data describe more precisely the algebraic structure of reductive groups or Kac-Moody groups; with RGD systems one can describe more general actions of groups on (twin) buildings.

\end{remas}

\begin{conses}\label{1.3c} Let $(G,(U_\qa)_{\qa\in\QF},Z)$ be a generating root datum, then, by \cite[chap. 1, 2]{Ry-02a} or \cite{AB-08}, one has:

\par 1) The group $B^\pm=ZU^\pm$ is called the {\it standard positive} (resp. {\it negative}) {\it Borel subgroup} or more generally {\it minimal parabolic subgroup}.

\par Let $N$ be the group generated by $Z$ and the $m(u)$ for $\qa\in\QF$ and $u\in U_\qa\setminus\{1\}$. There is a surjective homomorphism $\qn^v:N\to W^v$ (where $W^v$ is the Weyl group of the root system $\QF$) such that $\qn^v(m(u))=s_\qa$ and Ker$(\qn^v)=Z$.

\par Then $B^\pm\cap N=Z$ and $(B^\pm,N)$ is a BN-pair in $G$. In particular we have two Bruhat decompositions: \qquad $G=\bigsqcup_{w\in W^v}\;B^\qe wB^\qe$\qquad (for $\qe=+$ or $-$).

\par Moreover $G=\bigsqcup_{w\in W^v}\;(\prod_{\qb\in\QF^+\cap w\QF^-w^{-1}}\;U_\qb).wZ.U^+$, with uniqueness of the decomposition (refined Bruhat decomposition).
 The same is also true when exchanging $+$ and $-$.

\par 2) More precisely $(B^+,B^-,N)$ is a twin BN-pair; in particular we have a Birkhoff decomposition: \qquad $G=\bigsqcup_{w\in W^v}\;B^+ wB^-$ . Moreover for $u,u'\in U^+$, $v,v'\in U^-$ and $z,z'\in Z$, if $uzv=u'z'v'$ then $u=u'$, $v=v'$ and $z=z'$.

\par 3) Associated to the BN-pair $(B^\qe,N)$, there is a combinatorial building $\SHI_\qe^{vc}$ (viewed as an abstract simplicial complex) on which $G$ acts strongly transitively (with preservation of the types of the facets). The group $B^\qe$ is the stabilizer and fixer of the fundamental chamber $C^{vc}_\qe\subset \SHI_\qe^{vc}$.
The group $N$ stabilizes the fundamental apartment $\A^{vc}_\qe$ (which contains $C^{vc}_\qe$); it is equal to the stabilizer in $G$ of $\A^{vc}_\qe$, as the BN-pair is saturated \ie $Z=\bigcap_{w\in W^v}\;wB^\qe w^{-1}$.

\par The Birkhoff decomposition gives a twinning between the buildings $\SHI_+^{vc}$ and $\SHI_-^{vc}$; we have $Z=B^+\cap B^-$.

\par 4) As the facets of $\sht_\qe=\A^{v}_\qe$ are in one to one, increasing and $N-$equivariant  correspondence with those of the Coxeter complex $\A^{vc}_\qe$, we can glue different apartments together to get a geometric realization $\SHI_\qe^{v}=\SHI^v_\qe(G,\A^v)$ of $\SHI_\qe^{vc}$ (called vectorial or conical) in which the apartments and facets are cones.
 The different peculiar choices of $\A^v$ explained in \ref{1.1b} 2), 3) give vectorial buildings $\SHI^{vq}_\qe$, $\SHI^{vx}_\qe$, $\SHI^{vxl}_\qe$;

 \par For any of these vectorial buildings, the vector space $V_0\subset\vect{\A^v_\qe}$ acts $G-$equivariantly and stabilizes all facets or apartments. The essentialization of this building \ie its quotient $\SHI^{ve}_\qe(G,\A^v)=\SHI^v_\qe(G,\A^v)/V_0$ by $V_0$ is canonically equal to $\SHI_\qe^{vq}=\SHI^v_\qe(G,\A^{vq})$.

\par 5) There is a one to one decreasing correspondence between facets and parabolic subgroups:
 the stabilizer and fixer in $G$ of a facet $F^{v}\subset\sht_\qe$ or of $F^{v}/V_0\subset\sht_\qe^q$ is a parabolic subgroup $P(F^v)$ of $G$ (which is its own normalizer).
 We have a Levi decomposition $P(F^v)=M(F^v)\ltimes U(F^v)$.
 The group $P(F^v)$ (resp. $M(F^v)$) is generated by $Z$ and the groups $U_\qa$ for $\qa\in\QF(F^v)$ (resp. $\qa\in\QF^m(F^v)$) \ie $\qa\in\QF$ and $\qa(F^v)\geq{}0$ (resp. $\qa(F^v)=0$).
 The subgroup $U(F^v)$ is normal in $P(F^v)$ and contains the groups $U_\qa$ for $\qa\in\QF^u(F^v)$ \ie $\qa\in\QF$ and $\qa(F^v)>0$ \cite[6.2]{Ry-02a}. We define $G(F^v)=P(F^v)/U(F^v)\simeq M(F^v)$ and $N(F^v)=N\cap P(F^v)\subset M(F^v)$.

\par If $F^v=F_\qe^{v}(J)$ for $J\subset I$, then $P(F^v)=P^\qe(J)=B^\qe W^v(J)B^\qe$, $\QF^m(F^v)=\QF^m(J)$ and $N(F^v)/Z=W^v(J)$.
The group $G(J)=M(F^v)$ is endowed with the generating root datum $(G(J),(U_\qa)_{\qa\in\QF^m(J)},Z)$.

\end{conses}

\begin{theo}\label{1.4} With the notations of \ref{1.2}, $(G_\shs,(U_\qa)_{\qa\in\QF},T_\shs)$ is a root datum of type $\QF$. Moreover if $\vert K\vert\geq{}4$, $N$ is the normalizer of $T_\shs$ in $G_\shs$.

\begin{proof} This is essentially in \cite{T-87b} and \cite{T-92a}. See \cite[8.4.1]{Ry-02a}
\end{proof}

\end{theo}

\begin{remas}\label{1.4b}
1) $B^\pm_\shs$ (resp. $U^\pm_\shs$) as defined in \ref{1.2} coincide with $B^\pm$  defined in \ref{1.3c}.1 (resp. $U^\pm$ defined in \ref{1.3}).
The group $N$ is $N_\shs=\g N_\shs(K)$, where $\g N_\shs$ is a sub-group-functor of $\g G_\shs$ normalizing $\g T_\shs$.
 Moreover $N$ is the normalizer in $G_\shs$ of $\g T_\shs$, but not always the normalizer of $T_\shs$ (\eg when $\vert K\vert=2$, $T_\shs=\{1\}$).
 The maximal split tori of $\g G_\shs$ are conjugated by $G_\shs$ to $\g T_\shs$ \cite[12.5.3]{Ry-02a}.

 \par The Levi factor of $P^\qe(J)$ is $G(J)=\g G_{\shs(J)}(K)$ where $\g G_{\shs(J)}$ is the split Kac-Moody group associated to the RGS $\shs(J)$ of \ref{1.1}.2 \cite[5.15.2]{Ru-11}.

 \par 2) The combinatorial buildings associated to this root datum are written $\SHI_\qe^{vc}(\g G_\shs,K)$ or $\SHI_\qe^{vc\M}(K)$, as they depend only on the field $K$ and the Kac-Moody matrix $\M$ (not of the SGR $\shs$: \cite[1.10]{Ru-11}).

 \par As $N$ is the stabilizer of the fundamental apartment $\A^{vc}_\qe$ in $\SHI_\qe^{vc\M}(K)$ and the normalizer of $\g T_\shs$, we get a one to one correspondence $\g T\mapsto A^{vc}_\qe(\g T)$ between the maximal split tori in $\g G_\shs$ (or their points over $K$, if $\vert K\vert\geq{}4$) and the apartments of $\SHI_\qe^{vc\M}(K)$.

 \par 3) The geometric realization of $\SHI_\qe^{vc\M}(K)$ introduced in \ref{1.3c}.4 is named $\SHI_\qe^{v}(\g G_\shs,K,\A^v)$.
  If we use $\sht_\qe^q=\A_\qe^{vq}$, we call it  the {\it essential vectorial building}  $\SHI_\qe^{v}(\g G_\shs,K,\A^{vq})=\SHI_\qe^{vq}(\g G_\shs,K)=\SHI_\qe^{v\M}(K)$ of $\g G_\shs$ over $K$ of sign $\qe=\pm$.
 We have also {\it extended vectorial buildings} $\SHI_\qe^{v}(\g G_\shs,K,\A^{vxl})$ $=\SHI_\qe^{vxl}(\g G_\shs,K)=\SHI_\qe^{v\shs^l}(K)$ defined using $\sht_\qe^{xl}=\A^{vxl}_\qe$ instead of $\sht_\qe^{q}$ in  \ref{1.3c}.4.

 \par When $\shs$ is free, we can also use $\sht^x_\qe=\A^{vx}_\qe$ and define the {\it (normal) vectorial buildings} $\SHI_\qe^{v}(\g G_\shs,K,\A^{vx})=\SHI_\qe^{vx}(\g G_\shs,K)=\SHI_\qe^{v\shs}(K)$.


 \par To be short we omit often $K$ and/or $\M$, $\shs$, $\g G_\shs$, $\A^v$ in the above notations.

 \par 4)  $\SHI_\qe^{vc}(\g G_\shs,K)$ is clearly functorial in $K$.  $\SHI_\qe^{v}(\g G_\shs,K,\A^{v})$ is functorial in $K$ and $\shs$ (for commutative extensions).

\end{remas}

\subsection{Completions of $\g G_\shs$}\label{1.4c}

\par There is a {\it positive} (resp. {\it negative}) {\it completion} $\g G^{pma}_\shs$ (resp. $\g G^{nma}_\shs$) of $\g G_\shs$ (defined in \cite{M-88a}, \cite{M-89}) which is used in \cite{Ru-11} to get better commutation relations.
This is an ind-group-scheme which contains $\g G_\shs$ but differs from it by its positive (resp. negative) maximal  pro-unipotent subgroup: $\g U_{\shs}^{+}$ (resp. $\g U_{\shs}^{-}$) is replaced by a greater group scheme $\g U_{\shs}^{ma+}$ (resp. $\g U_{\shs}^{ma-}$) involving the full root system $\QD$ of \ref{1.1}.3.

\par For a ring $R$, an element of $\g U_{\shs}^{ma\pm{}}(R)$ can be written uniquely as an infinite product:

\par\noindent$u= \prod_{\qa\in\QD^\pm{}}\,u_\qa$ with $u_\qa\in\g U_\qa(R)$, for a given order on the roots $\qa=\sum_{i\in I}\,n_i^\qa\qa_i\in\QD^\pm{}$ (\eg an order such that $\vert ht(\qa)\vert=\sum_{i\in I}\,\vert n_i^\qa\vert$ is increasing).
 For $\qa\in\QF$, $u_\qa$ is written $u_\qa=\g x_\qa(r)=[exp]re_\qa$ for a unique $r\in R$ and $e_\qa$ a fixed basis of $\g g_\qa$.
  For $\qa\in\QD_{im}$, $u_\qa$ is written $u_\qa=\prod_{j=1}^{j=n_\qa}\;[exp]r_{\qa,j}.e_{\qa,j}$ for  unique $r_{\qa,j}\in R$ and for $(e_{\qa,j})_{j=1,n_\qa}$ a fixed basis of $\g g_\qa$; but $\g U_\qa(R)$ is not a group: this is only true for $\g U_{(\qa)}(R)=\prod_{\qb\in\Z_{>0}\qa}\, \g U_\qb(R)$.
Moreover the conjugate of  an element $u\in\g U_{\shs}^{ma\pm}(R)$ by $t\in\g T_\shs(R)$ is given by the same formula: for $u'_\qa=tu_\qa t^{-1}\in\g U_\qa(R)$, we just replace $r$ by $\overline\qa(t)r$  or each $r_{\qa,j}$ by $\overline\qa(t)r_{\qa,j}$ \cite[3.2, 3.5]{Ru-11}.
 (Actually we often write $\qa(t)$ for $\overline\qa(t)$.)

 \par The commutation relations between the $u_\qa$ are deduced from the corresponding relations in the Lie algebra (or better in the Tits enveloping algebra $\shu_{\shs\Z}$).
 So we know well the structure of the Borel groups $\g B_{\shs}^{ma\pm{}}=\g T_\shs\ltimes\g U_{\shs}^{ma\pm{}}$.
  For $R$ a field there are Bruhat and Birkhoff decompositions of $\g G^{pma}_\shs(R)$ and $\g G^{nma}_\shs(R)$:
  $\g G^{pma}_\shs(R)=\g U_{\shs}^{ma+}(R).\g N_\shs(R).\g U_{\shs}^{ma+}(R)=\g U_{\shs}^{-}(R).\g N_\shs(R).\g U_{\shs}^{ma+}(R)$ and
   $\g G^{nma}_\shs(R)=\g U_{\shs}^{ma-}(R).\g N_\shs(R).\g U_{\shs}^{ma-}(R)=\g U_{\shs}^{+}(R).\g N_\shs(R).\g U_{\shs}^{ma-}(R)$.

\subsection{Centralizers of tori}\label{1.4d}

Let $\g T'$ be a subtorus of $\g T_\shs$ (over $K_s$). There is a linear map $Y(\g T')\otimes\R\into Y(\g T_\shs)\otimes\R\to V^q$ which sends $\ql\otimes x$ to the map $\qa\mapsto\overline\qa(\ql)x$.  We write $V^q(\g T')$ its image.
 We say that $\g T'$ is {\it generic} (resp.  {\it almost generic}) in $\g T_\shs$ if $V^q(\g T')$ meets the interior of the Tits cone $\sht^q_+$ (resp. if  $V^q(\g T')\cap \sht^q_+$ generates the vector space  $V^q(\g T')$) \cf \ref{1.1b}.1.
 Note that, if $\shs$ is free, $\g T_\shs$ is generic in $\g T_\shs$, as the above map is then onto.

 \begin{prop*} If $\g T'$ is generic,  then, up to conjugacy, $V^q(\g T')$ is generated by $V^q(\g T')\cap \overline{C^{vq}_+}$ (which is convex) or more precisely by $V^q(\g T')\cap F^{vq}_+(J)$ where $F^{vq}_+(J)$ (with $J$ spherical in $I$) is the greatest facet in $ \overline{C^{vq}_+}$ meeting $V^q(\g T')$.
  Then the centralizer $\g Z(\g T')$ of $\g T'$ in $\g G_\shs$ is $\g G_{\shs(J)}$ (a reductive group).

  \par To be short, when $\g T'$ is generic, its  centralizer $\g Z(\g T')$ is the group scheme $\g Z_g(\g T')$ generated by $\g T_\shs$ and the $\g U_\qa$ for $\qa\in\QF$ and   $  \overline\qa\rest{\,\g T'}=1$   (called the generic centralizer of $\g T'$ in $\g G_\shs$).
 \end{prop*}
 \begin{proof} The reduction to $V^q(\g T')$ generated by $V^q(\g T')\cap F^{vq}_+(J)$ is clear as  $V^q(\g T')\cap \sht^q_+$ is convex and  generates $V^q(\g T')$.
  We embed $\g G_\shs$ in the ind-group-scheme $\g G^{pma}_\shs$. Any element $g\in\g G^{pma}_\shs(K_s)$ may be written uniquely as $g=(\prod_{\qa\in\QF^+\cap w\QF^-}\,u_\qa).t.\tilde w.(\prod_{\qa\in\QD^+}\,u_\qa)$, where $t\in\g T_\shs(K_s)$, $w\in W^v$, $\tilde w$ is its representant in a chosen system of representants $\widetilde W^v\subset\g N(\g T_\shs)(K_s)$ and each $u_\qa$ is in $\g U_\qa(K_s)$ \cf \cite[1.2.3]{Ry-02a} and \cite[3.2]{Ru-11}.
   If we conjugate by $s\in\g T'(K_s)$, $t$ is fixed, each $u_\qa$ is sent to $u'_\qa\in\g U_\qa(K_s)$.
   So $g$ commutes with $s$ if and only if $u'_\qa=u_\qa, \forall\qa$ and $s\tilde ws^{-1}=\tilde w$.
   This last condition is $s=\tilde ws\tilde w^{-1}=w(s)$; as it must be true $\forall s \in\g T'(K_s)$, this means that $w\in W^v(J)$.
   Now for $\qa\in\QF$ and $u_\qa=\g x_\qa(r)$, $u'_\qa=\g x_\qa(\qa(s).r)$; hence $u_\qa=u'_\qa,\forall s\in\g T'(K_s)\Rightarrow
   \overline\qa\,\rule[-1.5mm]{.1mm}{3mm}_{\,\g T'}=1 \Rightarrow\qa\,\rule[-1.5mm]{.1mm}{3mm}_{\, V^q(\g T')}=0\Rightarrow\qa\in Q(J)$.
   The same thing is true for $\qa\in\QD^+$ by the formulae in \ref{1.4c}.
   Finally $g\in\g Z(\g T')(K_s)\iff g\in \g G^{pma}_{\shs(J)}(K_s)$.
   But, as $F^{vq}_+(J)$ is in the interior of the Tits cone, $J$ is spherical, $\QD(J)$ is finite and $\g G^{pma}_{\shs(J)}=\g G_{\shs(J)}$ is a reductive group.
 \end{proof}

\begin{remas}\label{1.4e} a) $\g Z(\g T')$ is the schematic centralizer of $\g T'$  or the centralizer of  $\g T'(K_s)$. The centralizer of $\g T'(K)$ may be greater, \eg if $\vert K\vert=2$, $\g T_\shs(K)=\{1\}$.

\par b) If $\g T'$ is almost generic, the above proof tells that $\g Z(\g T')=\g G^{pma}_{\shs(J)}\cap\g G_\shs$.
 But it is not clear that it is the Kac-Moody group $\g G_{\shs(J)}$ \ie that $\g U^{ma+}_{\shs(J)}\cap\g G_\shs=\g U^{ma+}_{\shs(J)}\cap\g U^+_\shs$ is $\g U^+_{\shs(J)}$; \cf \cite[3.17 and {\S{}} 6]{Ru-11}.

 \par c) In the affine case with $\shs$ free, let $\qd$ be the smallest positive imaginary root.
 The torus $\g T'=$Ker$\qd$ is not almost generic, there is no real root $\qa\in\QF$ with $\qa\,\rule[-1.5mm]{.1mm}{3mm}_{\,\g T'}=1$ but $\g Z(\g T')$ is greater than $\g T_\shs$:
  if $\g G_\shs(K_s)=\g G^\circ(K_s[t,t^{-1}])\rtimes K_s^*$ for $\g G^\circ$ a semi-simple group with maximal torus $\g T^\circ$, then $\g T_\shs(K_s)=\g T^\circ(K_s)\times K_s^*$, $\g T'(K_s)=\g T^\circ(K_s)$ and $\g Z(\g T')(K_s)=\g T^\circ(K_s[t,t^{-1}])\rtimes K_s^*$ is the subset of $\g N_\shs(K_s)$ consisting of elements whose image in the affine Weyl group $W^v$ are in the "translation group".

  \par Otherwise said, when $\shs$ is affine non free, $\overline\qd=0$, $\g T'=\g T_\shs$ and $\g Z(\g T_\shs)$ may be greater than $\g T_\shs$: $\g T_\shs$ is not almost generic in $\g T_\shs$.
\end{remas}


\section{Almost split Kac-Moody groups}\label{s1b}

\par The reference for this section is B. R\'emy's monograph \cite{Ry-02a}.

\subsection{Kac-Moody groups}\label{1.5}

\par {\quad\bf1)} A {\it Kac-Moody group} over the field $K$ is a functor $\g G=\g G_K$ from the category $\shs ep(K)$ to the category of groups such that there exist a RGS $\shs$, a field $E\in \shs ep(K)$ and a functorial isomorphism between the restrictions $\g G_E$ and $\g G_{\shs E}$ of $\g G$ and $\g G_\shs$ to $\shs ep(E)=\{F\in\shs ep(K)\mid E\subset F\}$.
We say that $\g G$ is {\it split over} $E$, that  $\g G$ is a $K-${\it form of} $\g G_\shs$ and we fix such a functorial isomorphism to identify $\g G_E$ and $\g G_{\shs E}$.

\par The above condition is the most important but, to compensate the lack of a good notion of algebraicity, we need also a $K-$form $\shu=\shu_K$ of the Tits enveloping algebra $\shu_{\shs K_s}$
and some other technical conditions (PREALG1,2, SGR, ALG1,2) given in \cite[chap. 11]{Ry-02a} and omitted here. We write only the following condition [\lc 12.1.1] which makes more precise the functoriality:
\smallskip
\par (DCS2) For each extension $L$ of $K$ in $\shs ep(K)$ the group $\g G(L)$ maps isomorphically to its canonical image in $\g G(K_s)$ which is the fixed-point-set
$\g G(K_s)^{Gal(K_s/L)}$ of the Galois group.
\medskip
\par We identify all these groups with their images in $\g G(K_s)$. We forget often the subscript $_\shs$ for subgroups of $\g G_{\shs K_s}$ when we think of them as subgroups of $\g G_{K_s}$, \eg $\g B^{\pm}_{\shs K_s}=\g B^{\pm}_{K_s}$. Now the natural action of $\QG=Gal(K_s/K)$ on $\g G_{K_s}$  gives us a twisted action of $\QG$ on $\g G_{\shs K_s}$ such that $\g G_\shs(K_s)^{Gal(K_s/L)}=\g G(L)$ for each $L\in\shs ep(K)$ and $\g G(L)=\g G_\shs(L)$ if $E\subset L$.

\par A subgroup $H$ of $\g G(K_s)$ invariant under this twisted action of $Gal(K_s/L)$ defines a sub-group-functor $\g H_L$ on $\shs ep(L)$; we say that $H$ is $L-defined$ in $\g G(K_s)$ and that $\g H_L$ is a $L-$sub-group-functor of $\g G_L$.

\par {\bf2)} We say that $\g G$ is {\it almost split} if  the twisted action of each $\qg\in \QG$ transforms $\g B^\qe_{\shs K_s}$ into a Borel subgroup in the same conjugacy class under $\g G_\shs(K_s)$.

\par Let $L$ be an infinite field in $\shs ep(E)$,  Galois over $K$, then there is a (twisted) action of $Gal(L/K)$ on the twin buildings $\SHI^{vc}_\qe(L)$ such that the action of $\g G(L)$ on $\SHI^{vc}_\qe(L)$ is $Gal(L/K)-$equivariant; this action permutes the types of the facets [\lc 11.3.2].
One can extend affinely this action on the  geometric realization $\SHI^{vq}_\qe(L)$ [\lc 12.1.2] or on the so-called "metric" realization, where the action is through a bounded group of isomorphisms;
 more precisely any point in this last realization has a finite orbit [\lc 11.3.3, 11.3.4].

 \par As a consequence any Borel subgroup of $\g G(K_s)$  is defined over a finite Galois extension of $K$; taking a greater extension $K'$  such that $Gal(K_s/K')$ preserves the types, we see that the same thing is true for parabolic subgroups.
  A maximal torus of $\g G(K_s)$ is intersection of two opposite Borel subgroups, so it is also defined over a finite Galois extension of $K$.

 \par {\bf3)} If $\g G$ is  almost split, the twisted action of $\QG$ on $\g G_\shs(K_s)$ and $\shu_{\shs K_s}$ is described through a "star action" [\lc 11.2.2, 11.3.2].
 More precisely there is a map $\qg\mapsto g_\qg$ from $\QG$ to $\g G_\shs(K_s)$ and for each $\qg\in\QG$ an automorphism $\qg^*$ of $\g G_\shs(K_s)$ and a $\qg-$linear bijection $\qg^*$ of $\shu_{\shs K_s}$, such that the twisted actions are given by $\widetilde\qg=$Int$(g_\qg)\circ\qg^*$ on $\g G_\shs(K_s)$ and $\widetilde\qg=$Ad$(g_\qg)\circ\qg^*$ on $\shu_{\shs K_s}$; moreover $g_\qg$ and $\qg^*$ are trivial for $\qg\in Gal(K_s/E)$.
  On $\g G_\shs(K_s)$, $\qg^*$ stabilizes $\g T_\shs$ and $\g B^{\pm}_\shs$; on $\shu_{\shs K_s}$, $\qg^*$ stabilizes $\shu_{\shs K_s}^0$ [\lc 11.2.5(i)]. Actually $g_\qg$ is defined up to $\g T_\shs(K_s)$ (but the map $\qg\mapsto g_\qg$ may be chosen with a finite image, by \ref{1.5}.2 above).
   So $\qg^*$ is defined up to $\g T_\shs(K_s)$. The "star action" is perhaps not an action on $\g G_\shs(K_s)$ or $\shu_{\shs K_s}$, but it defines an action on $\shu^0_{\shs K_s}$, $X$, $\QD$, $\QF$, $\QF^+$, $I$ or $W^v$.
   
 \par {\bf4)}  By \cite[3.19.4]{Ru-11} we may actually choose the element $g_\qg$ in $\g G_{\shs^{lad}}(K_s)=:\g G^{xlad}(K_s)$.
  Then we may add the condition that $\qg^*$ stabilizes the \'epinglage $(\g T_\shs,\QF^+,(e_i)_{i\in I})$ \ie $\qg^*(e_i)=e_{\qg^*(i)}$ and $\qg^*(x_{\qa_i}(r))=x_{\qg^*(\qa_i)}(\qg r)$ for $i\in I$ and $r\in K_s$ \cite[11.2.5 iii)]{Ry-02a}.
  This condition determines uniquely $g_\qg$ and $\qg^*$.
  But $\qg^*$ may be extended as an automorphism of $\g G^{xlad}(K_s)$ (\cf \cite[1.8.2]{Ru-11}), so $\qg^*\circ$Int$(g_{\qg'})=$Int$(\qg^*(g_{\qg'}))\circ\qg^*$.
  We deduce from this that $g_{\qg\qg'}=g_{\qg}.\qg^*(g_{\qg'})$ and $(\qg\qg')^*=\qg^*\qg'^*$:
  thus defined, the "star action" is a true action.

\begin{lemm}\label{1.6} Let $\g G$ be an almost split $K-$form of $\g G_\shs$ as above.

\par a) There is  an almost split $K-$form $\g G^{xl}$ of $\g G_{\shs^l}$ which is split over the same field $E$ as $\g G$ and an homomorphism $\g G\to\g G^{xl}$ whose restriction to $\shs ep(E)$ is the known homomorphism $\g G_\shs\to\g G_{\shs^l}$ \cite[1.3d, 1.11]{Ru-11}.

\par b) Let $L$ be an infinite field in $\shs ep(E)$ Galois over $K$, then the action of $Gal(L/K)$ on the building $\SHI^{vc}_\qe(L)$ may be extended linearly to $\SHI^{vxl}_\qe(L)$.
 This action makes $\QG-$equivariant the action of $\g G^{xl}(K_s)$ ($=\g G(K_s)$) over this building and the essentialization map $\eta^v:\SHI^{vxl}_\qe(L)\to \SHI^{vq}_\qe(L)$.

 \par c) When $\shs$ is free, the same thing is true for $\SHI^{vx}_\qe(L)$ and $\g G(K_s)$.

\end{lemm}

\begin{enonce*}[definition]{Remarks}  1) For $\A^v$ as in \ref{1.1b}.4, let us  suppose that the star action of $\QG$ on $I$ (hence on the $\qa_i$, $\qa_i^\vee$) may be extended linearly to $V$.
 Then b) and c) above are also true for $\SHI^v(\g G_\shs,K,\A^v)$.
 
 \par 2) If we define on $\g G^{xlad}(K_s)$ $\qg^*$ as above in \ref{1.5}.4 and $\widetilde\qg=$Int$(g_\qg)\circ\qg^*$, we can prove as below that we get an extension of the twisted action to $\g G^{xlad}(K_s)$.
\end{enonce*}

\begin{proof} a) We have to describe the form $\g G^{xl}$ and a $K-$form $\shu^{xl}$ of the Tits enveloping algebra $\shu_{\shs^lK_s}$ through twisted actions of $\QG$ on $\g G_{\shs^l}(K_s)$ and $\shu_{\shs^lK_s}$ extending those constructed in \ref{1.5}.3 above for $\shs$.
For the RGS $\shs^l$, $Y^l=Y^{xl}=Y\oplus Q^*$, so, by \cite[1.11]{Ru-11}, $\g G_{\shs^lK_s}$ is the semi-direct product of $\g G_{\shs K_s}$ by the torus $\g T_{QK_s}=\g{Spec}(K_s[Q])$.
Clearly $\shu_{\shs^lK_s}$ is also a semi-direct product of $\shu_{K_s}=\shu_{\shs K_s}$ and the "integral enveloping algebra" $\shu_{Q K_s}$ of the torus $\g T_{QK_s}$ (\ie its algebra of distributions at the origin).
 Now the star action of $\QG$ on $Q$ gives a $\QG-$algebraic action on $\g T_{QK_s}$ and a $\QG-$linear action on $\shu_{Q K_s}$.
 \label{N1} This is compatible with the formulae defining semi-direct products and so we construct an automorphism $\qg^*$ of $\g G_{\shs^l}(K_s)$ and a $\qg-$linear bijection $\qg^*$  of $\shu_{\shs^lK_s}$. Now let $\widetilde\qg=$Int$(g_\qg)\circ\qg^*$ or $\widetilde\qg=$Ad$(g_\qg)\circ\qg^*$.

  \par We have to prove that this defines actions of $\QG$. By definition  Int$(g_{\qg\qg'})\circ(\qg\qg')^*=\widetilde{\qg\qg'}$ and $\widetilde\qg\circ\widetilde{\qg'}=$Int$(g_{\qg})\circ{\qg}^*\circ$Int$(g_{\qg'})\circ{\qg'}^*=$Int$(g_{\qg}.\qg^*(g_{\qg'}))\circ{\qg}^*\circ{\qg'}^*$.
  There is  equality of these two expressions on $\g G_\shs(K_s)$, moreover $(\qg\qg')^*=\qg^*\circ\qg'^*$ on $\g T_{\shs K_s}$,
  hence $g_{\qg}.\qg^*(g_{\qg'})=g_{\qg\qg'}.t_{\qg,\qg'}$ with $t_{\qg,\qg'}\in\g T_{\shs}(K_s)$.
   We have to verify that $\widetilde{\qg\qg'}(t)=\widetilde\qg\circ\widetilde{\qg'}(t)$ for $t\in \g T_Q(K_s)$.
   But $(\qg\qg')^*(t)=\qg^*(\qg'^*(t))$ is in $\g T_Q(K_s)$ hence centralized by $t_{\qg,\qg'}$; so the result follows.
    The same proof works also for $\shu_{\shs^lK_s}$.

    \par \label{N19} We define $\shu^{xl}$ as $(\shu_{\shs^lK_s})^{Gal(K_s/K)}$ and, for $L\in\shs ep(K)$, $\g G^{xl}(L)=\g G_{\shs^l}(K_s)^{Gal(K_s/K)}$ (fixed points for the twisted actions). We have now the two ingredients of the Kac-Moody group as defined above. We leave to the reader the verification of the technical conditions of \cite{Ry-02a} (PREALG, SGR, $\cdots$).

    \par b,c) As the star action of $\QG$ is well defined on $X^{xl}=X\oplus Q$ and on $X$, we just have to mimic the proof in the case $\SHI^{vq}_\qe(L)$ (corresponding to $X^q=Q$) [\lc 12.1.2].

\end{proof}

\subsection{Continuity of the actions of the Galois group}\label{1.7}

\par From now on in this section \ref{s1b}, we choose an almost split Kac-Moody group $\g G$ over $K$ with Tits enveloping algebra $\shu$ and keep the above notations. We forget now the (old) actions of $\QG=Gal(K_s/K)$ on $\g G_{\shs K_s}$ or $\shu_{\shs K_s}$ and consider only the star action or the (twisted) action (which is the natural action on $\g G_{K_s}$ or $\shu_{K_s}$).

\par {\bf1)} By \ref{1.5}.2 above the orbits of $\QG$ on the Borel subgroups of $\g G_{K_s}$ are finite. So the $g_\qg\in\g G(K_s)$ (such that $\qg^*=$Int$(g_\qg)^{-1}\circ\qg$ stabilizes  $\g T_{K_s}$ and $\g B^{\pm}_{K_s}$) may be chosen in a finite set.
In particular, if $\qg^*=$Ad$(g_\qg)^{-1}\circ\qg$ on $\shu_{K_s}$, then $\{\qg^*u\mid\qg\in\QG\}$ is finite $\forall u\in \shu_{K_s}$.
So the star action of $\QG$ has finite orbits on $\QF$ \cite[11.2.5(iii)]{Ry-02a} and on $Q$. We know that this star action stabilizes the basis $\{\qa_i\mid i\in I\}$ and acts on $I$ by automorphisms of the Dynkin diagram.

\par {\bf2)} The following extension of condition (ALG2) is implicit in \lc starting \eg from 11.3.2:
\vskip0.2cm\par
\noindent (ALG3) The star action of $\QG$ on $X$ or $Y$ is continuous \ie its orbits are finite.
\medskip

\par As for (ALG2) this is useless if $X=Q$. Without it, only the description of the action of $\QG$ on the center of $\g G(K_s)$ is less precise; in particular there is no problem for the buildings
 $\SHI^{vq}_\qe(L)$. But the proof of [\lc 12.5.1(i)] uses this property.

 \par It would also be reasonable to ask:
 \smallskip
 \par\noindent (ALG3') The evident map $\qf:Y\to Y\otimes K_s\subset\shu^0_{K_s}$ is $\QG^*-$equivariant.
 \medskip
 \par\label{N2} By [\lc 11.2.5] the map $\qf$ restricted to $Q^\vee=\sum_{i\in I}\,\Z\qa_i^\vee$ is $\QG^*-$equivariant. In characteristic $0$ (ALG3) is a consequence of (ALG3').

 \par In the following we add to the conditions of \lc the condition (ALG3) but not (ALG3'). With these assumptions we get the good structure for $\g G$.
  But anybody interested in considering $\shu$ as the good Tits enveloping algebra for $\g G$ should add (ALG3') and, in positive characteristic, perhaps some stronger conditions, see \ref{1.14}.

  \par {\bf3)} By \ref{1.5}.2 $\g T_{K_s}$ and $\g B^\pm_{K_s}$ are defined over a finite Galois extension $L$ of $K$ in $\shs ep(K)$.
  Enlarging a little $L$ we may suppose $\g T_{K_s}$ split over $L$ (\ie $X$ or $Y$ fixed pointwise under $Gal(K_s/L)$) and $Q$ also fixed (\ref{1.7}.1).
   Now we may modify each $e_i$ in $K_se_i=\shu^+_{\qa_iK_s}$, so that $e_i$ (and $f_i$) is fixed under $\QG$.
    By [\lc11.2.5(iii)] this proves that $\qg(\g x_{\pm{}\qa_i}(r))=\g x_{\pm{}\qa_i}(\qg r)$ for $r\in K_s$ and $\qg\in Gal(K_s/L)$.
    So the original action and the new twisted action of $\QG$ on $\g G_\shs(K_s)$ coincide on $\g T(K_s)$ and the groups $\g U_{\pm{}\qa_i}(K_s)$. As these groups generate $\g G_\shs(K_s)$ (see \cite[1.6 KMT7]{Ru-11}), the two actions coincide and $\g G$ is actually split over the finite Galois extension $L$ of $K$.

    \par Now each of the above generators of $\g G(K_s)$ has a finite orbit under $\QG$, so this is also true for every element of $\g G(K_s)$: $\g G(K_s)$ is the union of the subgroups $\g G(L)$ for $L\in \shs ep(K)$ with $L/K$ finite.

    \par {\bf4)} We saw in \ref{1.5}.2 that the orbits of $\QG$ on $\SHI^{vc}_\qe(K_s)$ are finite.
    The stabilizer in $\QG$ of a facet of $\SHI^{vc}_\qe(K_s)$ acts on the corresponding facet of $\SHI^{vq}_\qe(K_s)$, $\SHI^{vxl}_\qe(K_s)$ or $\SHI^{vx}_\qe(K_s)$ through a finite group (see 2) above and the definition of these actions).
    So the actions of $\QG$ other these buildings have finite orbits.

    \par If a Galois group $Gal(K_s/M)$ stabilizes a facet of one of these geometric buildings, then it has a fixed point in this facet (as this facet is a convex cone and the action is affine).

\subsection{$K-$objects in the buildings} \label{1.8}

\par {\quad\bf1)} Let $E\in\shs ep(K)$ be infinite, Galois over $K$ and such that $\g G$ is split over $E$.
 By \cite[10.1.4 and 13.2.4]{Ry-02a} the buildings over $E$ are the fixed point sets in the buildings over $K_s$ of the Galois group $Gal(K_s/E)$. So we set $\QG=Gal(E/K)$ and we shall work over $E$ (\cf \lc 12.1.1(1)).

 \par Let $\SHI^{v}=\SHI^{v}_+\cup\SHI^{v}_-$ be the union $\SHI^{v}(\g G_\shs,E,\A^v)=\SHI^{v}_+(\g G_\shs,E,\A^v)\cup\SHI^{v}_-(\g G_\shs,E,\A^v)$ as in remark \ref{1.6} (\eg $\A^v=\A^{vq}$, $\A^{vxl}$ or if $\shs$ is free $\A^{vx}$).
 The essentialization $\SHI^{ve}$ of $\SHI^{v}$ is always $\SHI^{vq}(E)=\SHI^{v}(\g G_\shs,E,\A^{vq})$ which is the building investigated in \lc so one may use this reference.

 \par {\bf2)} \textbf{Definitions.} A {\it $K-$facet} (resp. {\it spherical $K-$facet}) in $\SHI^v$ is the fixed point set under $\QG$ of a facet (resp. spherical facet) of $\SHI^v$ stable under $\QG$ (by \ref{1.7}.4 the $K-$facet is non empty).

 \par A {\it $K-$chamber} in $\SHI^v$ is a spherical $K-$facet with maximal closure.

 \par A {\it $K-$apartment} in $\SHI^v$ is a generic subspace (of an apartment) of $\SHI^v$ which is (pointwise) fixed under $\QG$ and maximal for these properties.

 \par A (real) {\it $K-$wall} in a $K-$apartment $_KA^v$ is the intersection with $_KA^v$ of a wall in an apartment of $\SHI^v$ containing $_KA^v$, provided that this intersection contains a spherical $K-$facet. This $K-$wall divides $_KA^v$ into two (closed) {\it $K-$half-apartments}.

 \par {\bf3)} \textbf{Properties.} By definition $K-$facets (resp. spherical $K-$facets, $K-$chambers) correspond bijectively to $K-$defined parabolics (resp. $K-$defined spherical parabolics, minimal $K-$defined parabolics). The union of the $K-$facets is $(\SHI^v)^\QG$; their set is written $_K\SHI^{vc}$.

 \par By [\lc 12.2.4 and 12.3.1] two $K-$facets are always in a same $K-$apartment and there exists an integer $d=d(\SHI^{v\QG})\geq{}1$ such that each $K-$chamber or each $K-$apartment is of dimension $d$.
  One should notice that the different choices for $\SHI^v$ may give  different integers $d$.
   The group $G=\g G(K)$ acts transitively on the pairs $(_KC,{_KA^v})$ of a $K-$chamber $_KC$ of given sign in a $K-$apartment $_KA^v$ (see also [\lc 12.4.1]).

   \par {\bf4)}  \textbf{Standardizations.} Any $K-$apartment $_KA^v$ in $\SHI^v$ is contained in a Galois stable apartment $A^v$ of $\SHI^v$ (perhaps after enlarging a little $E$) [\lc 12.3.2(1)].
  \label{N20}  We may choose moreover opposite chambers $_KC^v_+$, $_KC^v_-$ in $_KA^v$ and (non necessarily $\QG-$stable) opposite chambers $C^v_+$, $C^v_-$ in $A^v$ with $_KC^v_\pm\subset\overline{C^v_\pm}$.
    We say that $(_KA^v,{_KC^v_+},{_KC^v_-})$ and $(A^v,C^v_+,C^v_-)$ are {\it compatible standardizations} of $\SHI^{v\QG}$ and $\SHI^v$.

    \par The apartment $A^v$ determines a $K-$defined maximal torus $\g T_{K_s}$ (such that $A^v=A^v(\g T_{K_s})$). After enlarging a little $E$ we may suppose $\g T_{K_s}$ split over $E$, and conjugated under $\g G(E)$ to the fundamental torus $\g T_{\shs E}$ [\lc 10.4.2].
     So we may (and will) suppose $\g T_{K_s}=\g T_{\shs K_s}$ and $C^v_\pm$ associated to the Borel subgroups $\g B^\pm_{\shs E}$.
     Then the star action of $\QG$ is defined by $\qg^*=$Int$(g_\qg)^{-1}\circ\qg$ with $g_\qg\in\g G(E)$ normalizing $\g T_\shs$ and fixing pointwise $_KC^v_+$, $_KC^v_-$ and $_KA^v$.

     \par Let $I_0=\{i\in I\mid\qa_i(_KA^v)=\{0\}\;\}$ and $A^{vI_0}=\{x\in A^v\mid \qa_i(x)=0,\forall i\in I_0\}$.
     Then $I_0$ is spherical (as $_KA^v$ meets spherical facets) and stable under $\QG^*$, the (normal twisted) action and the star action of $\QG$ coincide on $A^{vI_0}$ and $_KA^v=(A^{vI_0})^{\QG^*}$.
     The vector space generated by $_KA^v$ in $\vect{A^v}$ is $\vect{_KA^v}=\{v\in \vect{A^v}\mid \qa_i(x)=0,\forall i\in I_0\}^{\QG^*}$ [\lc 12.6.1].

 \subsection{Maximal split tori and relative roots}\label{1.9}

 \par We choose standardizations and identifications as in \ref{1.8}.4 above.

 \par {\bf1)} The maximal split subtorus $\g S$ of $\g T$ depends only on  the $K-$apartment $_KA^v$  and is actually a maximal split torus in $\g G$.  The maximal split tori are conjugated under $G=\g G(K)$ \cite[12.5.2, 12.5.3]{Ry-02a}. The dimension of a maximal split torus is the {\it reductive relative rank over} $K$ of $\g G$, written $rrk_K(\g G)$.

 \par As a consequence of [\lc cor. 12.5.3] and lemma \ref{1.10}(ii) below, there is a bijection $\g S\mapsto {_KA^v}(\g S)$ between maximal $K-$split tori in $\g G$ (or their points over $K$, if $\vert K\vert\geq{}4$) and the $K-$apartments in $\SHI^v$.

 \par {\bf2)} Let $_KX$ (resp. $_KY$) be the group of characters (resp. cocharacters) of $\g S$.
 For each $\qa\in Q$, let $_K\overline\qa\in{_KX}$ be the restriction to $\g S$ of $\overline\qa\in X$ and $_K\qa$ be the restriction of $\qa$ to $\vect{_KA^v}$.
  We define $_KQ$ as the image of $Q$ by this restriction map $\qa\mapsto{_K\qa}$; the set of {\it relative $K-$roots} is $_K\QD=\{_K\qa\mid\qa\in\QD,\,_K\qa\not=0\}$;
  the set of {\it real relative $K-$roots} is $_K\QF=\,_K\QD^{re}=\{_K\qa\in\,_K\QD\mid\,_KA^v\cap$Ker$\qa$ is a (real) $K$-wall$\}$.
  Let $_KQ_{re}$ be the submodule of $_KQ$ generated by the real relative roots.

  \par With the notations in \ref{1.8}.4, $_K\qa_i$ is a root if and only if $i\notin I_0$; for $i,j\notin I_0$, $_K\qa_i=\,_K\qa_j$ if and only if $i$ and $j$ are in the same $\QG^*-$orbit;  $_K\qa_i$ is a real root if $I_0\cup \QG^*i$ is spherical and different from $I_0$.
   Hence a basis of $_K\QD$ (or $_KQ$) is given by $\{_K\qa_i\mid i\in \,_KI\}$ where $_KI=(I\setminus I_0)/\QG^*$ and a basis of $_K\QF$ (or $_KQ_{re}$) is given by $\{_K\qa_i\mid i\in \,_KI_{re}\}$ where $_KI_{re}=\{i\in I\setminus I_0\mid I_0\cup\QG^*i\, spherical\}/\QG^*$:
    the basis of $_K\QD$ may contain some imaginary relative roots.
    The set $\QF_0=\{\qa\in\QF\mid\,_K\qa=0\}$ is actually $\QF\cap(\oplus_{i\in I_0}\,\Z\qa_i)$.
    We say that $\vert_KI_{re}\vert=ssrk_K(\g G)$ is the {\it semi-simple relative rank over} $K$ of $\g G$.

    \par For $i\in I_0$, $\qa_i$ is trivial on $\g S$ (see 3) below); so the two actions of $\QG$ (star or not) on $\g T$ coincide on $\g S$ and $_KY=\{y\in Y\mid\qa_i(y)=0,\forall i\in I_0\}^{\QG^*}$.
  \label{N21}   Hence $\forall\qa\in Q$, $_K\overline \qa$ is the canonical image $\overline{_K\qa}$ of $_K\qa$ in $_KX$.
      It is now clear that $\{\overline{_K\qa}\mid\,_K\qa\in\,_K\QD\}$ is the set of roots of $\g S$ for the adjoint representation on the Lie algebra $\g g_K\subset\shu_K$.

      \par When $\shs$ is free and $\SHI^v=\SHI^{vx}(E)$ is the normal geometric realization, then
      $dim(_KA^v)=dim(\g S)$ is the reductive relative rank. \label{N3} Hence, when $\shs$ is free, the reductive relative rank is at least $1$: an almost split Kac-Moody group (with $\shs$ free) cannot be anisotropic.

   \par   When $\SHI^v$ is the essential building $\SHI^{vq}(E)$, then dim$(_KA^v)=\vert_KI\vert$ may be greater than $\vert_KI_{re}\vert=ssrk_K(\g G)$, so $_KA^v$ may be inessential.

      \par {\bf3)} {\bf Relative Weyl group.} [\lc12.4.1, 12.4.2]

      \par Let $_KN$ (resp. $_KZ$) be the stabilizer (resp. fixer) of $_KA^v$ in $G$; by \ref{1.9}.1 $_KN$ is the normalizer of $\g S$ in $G$ and $_KZ$ centralizes $\g S$ (by definition of $\g S$ [\lc 12.5.2]).
      Actually $_KZ$ is generated by $T$ and the $U_\qa$ for $\qa\in\QF_0$ [\lc 6.4.1], hence $\qa(\g S)=1$ $\forall\qa\in\QF_0$.
      The quotient group $_KW^v=\,_KN/_KZ$ is the {\it relative Weyl group of} $\g G$ (associated to $\g S$ or $_KA^v$).
      It acts simply transitively on the $K-$chambers of fixed sign  in $_KA^v$ and (as $_KN\subset N.{_KZ}$) is induced by the action of the subgroup of $W^v=W^v(A^v)$ stabilizing $_KA^v$.

      \par To each real relative root $_K\qa\in\,_K\QF$ is associated an element $s_{_K\qa}\in\,_KW^v$ of order 2 which fixes the wall Ker$(_K\qa)$. The pair $(_KW^v,\{s_{_K\qa_i}\mid i\in\,_KI_{re}\})$ is a Coxeter system.

      \par When $\shs$ is free, the map $Y(\g T_\shs)\otimes\R\to\vect{A^{vq}}=(Q\otimes\R)^*$ is onto.
  But $_K\!A^{vq}$ is generic in $A^{vq}$ (\ref{1.8}.2) and, by \ref{1.8}.4 and \ref{1.9}.2, the same equations define $\vect{_K\!A^{vq}}$ in $\vect{A^{vq}}$ or $Y(\g S)\otimes\R$ in $Y(\g T_\shs)\otimes\R$.
   So $\g S$ is generic in $\g T_\shs$, $\g Z(\g S)=\g Z_g(\g S)$ and $_KZ=\g Z(\g S)(K)$ (\ref{1.4d}).

   \par For $\shs$ general $_KZ=\g Z_g(\g S)(K)$ may be smaller than $\g Z(\g S)(K)$, \cf \ref{1.4e}c.
   The reductive group $\g Z_g(\g S)$ is the {\it anisotropic kernel} \cite[12.3.2]{Ry-02a} associated to $_KA^v$ \ie to $\g S$ (by 1) above).

 \par {\bf4)}  The set $_K\QD$ (resp. $_K\QF=\,_K\QD^{re}$) is a system of roots (resp. of real roots) in the sense of \cite{By-96}  \cf \cite[12.6.2]{Ry-02a},  \cite{By-96} or  \cite{B3R-95}.
      If $_K\QD^\pm{}=\pm{}(_K\QD\cap(\oplus_{i\in_KI}\,\Z_{\geq{}0}._K\qa_i))$ (resp. $_K\QF^\pm{}=\pm{}(_K\QF\cap(\oplus_{i\in_KI_{re}}\,\Z_{\geq{}0}._K\qa_i))$ then $_K\QD=\,_K\QD^+\sqcup\,_K\QD^-$ and $_K\QF=\,_K\QF^+\sqcup\,_K\QF^-$.
      The system $_K\QD$ or $_K\QF$ is stable under $_KW^v$ and any real relative root is of the form $w._K\qa_i$ or $2w._K\qa_i$ with $w\in{_KW^v}$ and $i\in{_KI_{re}}$ (as the system may be unreduced).

      \par It is not too hard to find a RGS $(_K\M,\,_KY,(_K\overline{\qa_i})_{i\in_KI_{re}},(_K\qa_i^\vee)_{i\in_KI_{re}})$ (in the sense of \ref{1.1}) with Weyl group $_KW^v$.
      But it is not sufficient to describe $_K\QD$ (or even $_K\QF$); one has to use a more complicated notion of RGS, see  \cite{By-96}, \cite{B3R-95} or \cite[12.6.2]{Ry-02a}.
     On the contrary the reduced system $_K\QF_{red}$ is a system of real roots in the sense of \cite{MP-89},  \cite{MP-95} and even of \cite{K-90} as its basis is free.

      \subsection{Relative root groups}\label{1.10}

      \par For $_K\qa\in{_K\QF}$, we consider the finite set $(_K\qa)=\{\qb\in\QF\mid{_K\qb}\in\N.{_K\qa}\}$ and the unipotent group $\g U_{(_K\qa)K_s}$ generated by the $\g U_{\qb K_s}$ for $\qb\in(_K\qa)$, it is defined over $K$. We set $V_{_K\qa}=\g U_{(_K\qa)}(K)$.

      \par The positive integral multiples of $_K\qa$ in $_K\QD$ are $_K\qa$ and (eventually) $2{_K\qa}$ ($\in{_K\QF}$).
       If $2{_K\qa}\notin{_K\QD}$ we set $\g U_{(2{_K\qa})K_s}=\{1\}$ and $V_{2{_K\qa}}=\{1\}$.
       So $\g U_{(2{_K\qa})K_s}$ (resp.  $V_{2{_K\qa}}$) is always a normal subgroup of $\g U_{(_K\qa)K_s}$ (resp. $V_{_K\qa}$) \cite[12.5.4]{Ry-02a}.


       \begin{lemm*} Let $_K\qa\in{_K\QF}$ be a real relative root.

       \par (i) $V_{_K\qa}/V_{2{_K\qa}}$ is isomorphic to a vector space over $K$ on which $s\in\g S(K)$ acts by multiplication by $_K\overline\qa(s)\in K^*$. Its dimension is $\vert(_K\qa)\vert-\vert(2{_K\qa})\vert>0$.

       \par (ii) The centralizer $Z_{V_{_K\qa}}(\g S(K))$ of $\g S(K)$ in $V_{_K\qa}$ is trivial if $\vert K\vert\geq{}4$.
       \end{lemm*}

       \begin{proof} (suggested in [\lc 12.5.3]) By [\lc 12.5.4] there exists a reductive $K-$group $\g H$ of relative semi-simple rank $1$ containing $\g S$ and $\g U_{(_K\qa)}$. Then (i) is classical \cf \eg  \cite[th. 3.17]{BlT-65}.
       Now (for $\g H$) there exists a coroot $_K\qa^\vee\in$ Hom$(\g{Mult},\g S)$ such that $2{_K\qa}({_K\qa^\vee})=2$ or ${_K\qa}({_K\qa^\vee})=2$ (if  $2{_K\qa}$ is not a root) (one may use the $K-$split reductive subgroup of $\g H$ constructed in \cite[7.2]{BlT-65}). So (ii) follows.
       \end{proof}

       \begin{theo}\label{1.11} (\cite[12.6.3 and 12.4.4]{Ry-02a}) Let $\g G$ be an almost split Kac-Moody group over $K$, then,

       \par a) The triple $(\g G(K),(V_{_K\qa})_{_K\qa\in{_K\QF}},{_KZ})$ is a generating root datum of type $_K\QF$.

       \par b) The fixed point set $_K\!\SHI^v=(\SHI^v)^\QG$ is a good geometrical representation of the combinatorial twin building $^K\!\!\SHI^{vc}=\SHI^{vc}(\g G,K)$ associated to this root datum:
       there are $\g G(K)-$equivariant bijections, between the $K-$apartments and the apartments of  $^K\!\!\SHI^{vc}$, and between the $K-$chambers and the chambers of  $^K\!\!\SHI^{vc}$; this last bijection is compatible with adjacency and opposition.
       \end{theo}

       \begin{enonce*}[definition]{N.B}  1) When $\g G$ is already split over $K$, we see easily, using galleries, that ${^K\!\!\SHI}^{vc}=(\SHI^{vc})^\QG$ and $_K\SHI^{v}=\SHI^{v}(\g G,K,\A^v)$.

       \par 2) The group $_KB^+={_KZ}U^+$ defined in \ref{1.3c}.1 for this root datum is a minimal $K-$parabolic of $\g G$.
       It is a Borel subgroup if and only if there exist Borel subgroups defined over $K$ (\ie $\g G$ is {\it quasi split} over $K$); this is equivalent to $I_0=\emptyset$ \ie to $\g Z_g(\g S)$ being a torus.

       \par 3) The objects defined in \ref{1.3} to \ref{1.3c} for the above root datum will bear a left or right index $_K$, sometimes a left exponent $^K$.
\end{enonce*}

       \subsection{Comparison with a Weyl geometric realization}\label{1.12}

       \par {\quad\bf1)} With the notations in \ref{1.8}.4, \ref{1.9}, we may describe the positive $K-$facets:

       $C^v_+=\{x\in\vect{A^v}\mid\qa_i(x)>0,\forall i\in I\}$

       $_KC^v_+=\{x\in\vect{_KA^v}\mid{_K\qa_i}(x)>0,\forall i\in {_KI}\}$ \qquad(relative interior of $\overline{C^v_+}\cap\vect{_KA^v}$)

       $_KA^v_+=\cup_{w\in{_KW^v}}\,w.\overline{_KC^v_+}$

       \par\noindent The $K-$facets in $\overline{_KC^v_+}$ correspond bijectively to subsets $_KJ$ of $_KI$ by setting:

       $_KF^v_+({_KJ})=\{x\in\vect{_KA^v}\mid{_K\qa_i}(x)>0,\forall i\in {_KI}\setminus{_KJ}\,and\,{_K\qa_i}(x)=0,\forall i\in {_KJ}\}\subset{_KA^v_+}\cap\overline{C^v_+}$.

 \par\noindent so the definition of the $K-$facets uses the whole $_K\QD$ (not only $_K\QF$).

       \par \label{N22} Moreover $_KF^v_+({_KJ})$ is spherical if and only if $_K\underline J=I_0\cup\{i\in I \mid \QG^*i\in{_KJ}\}$ is spherical, which is equivalent to $_KJ\subset{_KI_{re}}$ and $_KJ$ spherical in ${_KI_{re}}$ (as defined by the root system $_K\QF$) \cf \cite[p 163, p 175]{By-96}.

       \par {\bf2)} A Weyl geometric realization $_K\!\SHI^v_+=\SHI^v_+(\g G,K,{^K\!A^v})$ of the combinatorial building $^K\!\SHI_+^{vc}$ can be constructed using, for fundamental apartment and facets, subcones of the vector space $\vect{_KA^v}$ defined using $_K\QF$ (\ie $_KW^v$).
       The corresponding Weyl facets in the closure of the positive fundamental chamber are defined as:

\par\noindent       $^KF^v_+({_KJ})=\{x\in\vect{_KA^v}\mid{_K\qa_i}(x)>0,\forall i\in {_KI_{re}}\setminus{_KJ}\,and\,{_K\qa_i}(x)=0,\forall i\in {_KJ}\}$ \; for  $_KJ\subset{_KI_{re}}$

\par\noindent and the positive fundamental Weyl$-K-$apartment is $^K\!A^v_+=\bigcup_{w\in{_KW^v},{_KJ}\subset{_KI_{re}}}\;w.{^KF^v_+({_KJ})}$.

\par The building $^K\!\SHI^v$ is the disjoint union of the Weyl$-K-$facets associated to the parabolics of $(\g G(K),(V_{_K\qa})_{_K\qa\in{_K\QF}},{_KZ})$; it contains $_K\SHI^v$ by the following lemma. Its minimal facet is $^KF^v_+({_KI_{re}})=(\vect{_KA^v})_0$.

\begin{enonce*}[plain]{\quad3) Lemma} Let $_KJ\subset{_KI_{re}}$.

\par a) The intersection $^KF^v_+({_KJ})\cap{_KA^v_+}$ is the disjoint union of the $K-$facets $_KF^v_+({_KJ'})$ for $_KJ'\supset{_KJ}$ and $_KJ'\cap{_KI_{re}}={_KJ}$.

\par Among these $K-$facets the maximal one (for the inclusion of the closures) is $_KF^v_+({_KJ})$, which is open in $^KF^v_+({_KJ})$; moreover $^KF^v_+({_KJ})={_KF^v_+({_KJ})}+(\vect{_KA^v})_0$. The minimal one corresponds to $_KJ'={_KJ}\cup({_KI}\setminus{_KI_{re}})$.

\par b) The Weyl$-K-$facet $^KF^v_+({_KJ})$ is spherical if and only if $_KF^v_+({_KJ})$ is spherical and then this $K-$facet is the only spherical $K-$facet in $^KF^v_+({_KJ})\cap{_KA^v_+}$.

\par c) The Weyl$-K-$facet $^KF^v_+({_KJ})$ and all $K-$facets $_KF^v$ in $^KF^v_+({_KJ})\cap{_KA^v_+}$ have the same fixer $P_K^+({_KJ})={P_K}({_KF^v})$ in $\g G(K)$.
 Hence each $K-$facet of $_K\SHI^v_+$ is associated to a unique Weyl$-K-$facet in $^K\!\SHI^v_+$;
\end{enonce*}
\begin{proof} Let $w.\overline{_KC^v_+}$ (with $w\in{_KW^v}$) be a closed $K-$chamber meeting $^KF^v_+({_KJ})$, then $w.\overline{^KC^v_+}$ meets $^KF^v_+({_KJ})$; so $w\in{_KW^v}({_KJ})$ which fixes (pointwise) $^KF^v_+({_KJ})$.
Hence $w.\overline{_KC^v_+}\cap{^KF^v_+({_KJ})}\subset\overline{_KC^v_+}$ and ${_KA^v_+}\cap{^KF^v_+({_KJ})}\subset\overline{_KC^v_+}$. Now a) and b) \label{N22} are clear.

\par The fixer in $\g G(K)$ of $_KF^v_+({_KJ'})$ contains the fixer $P$ of $_KC^v_+$, hence it is a parabolic subgroup of the positive $BN-$pair associated to the root datum in $\g G(K)$ \ie of the form $P.{_KW^v({_KJ''})}.P$ for some $_KJ''\subset{_KI_{re}}$.
 It is easy to check that $_KJ''$ has to be $_KJ$ and c) follows.
\end{proof}

       \par {\bf4)} So the Weyl$-K-$facets of $^K\!\SHI_+^{vc}$ correspond to some $K-$facets of $(\SHI^v_+)^\QG$ and there is a good correspondence between spherical Weyl$-K-$facets and spherical $K-$facets.
       But, if $ {_KI_{re}}\not= {_KI}$, some non-spherical $K-$facets correspond to nothing in $^K\!\SHI_+^{vc}$.
       So $(\SHI^v)^\QG$ is only a geometric representation of $^K\!\SHI_+^{vc}$ in the sense of the theorem, it is not really a geometric realization of it.
       Note also that, if $ {_KI_{re}}\not= {_KI}$, the Weyl geometric realization $^K\!\SHI_+^{v}$ of $^K\!\SHI_+^{vc}$, constructed in 2) above, is not essential, even if $\SHI^v=\SHI^{vq}$ is.

       \par The above results (and those in \ref{1.13}) are well illustrated by example 13.4 in \cite{Ry-02a}.

 \par   {\bf5)} {\bf Remarks.} a) In this example we see also that $_K\QF$ may be a classical (finite) root system, even if $\QF$ is infinite.
       It may also happen that $_K\QF$ is empty (\ie $ssrk_K(\g G)=0$); this is always the case when $\QF$ is infinite and $\vert{_KI}\vert=1$ (see examples for $K=\R$ in the tables of \cite{B3R-95}).
       Then $\g G(K)={_K Z}$ and $(\SHI^v)^\QG$ is reduced to one $K-$apartment and two $K-$chambers (one of each sign).

       \par b) On the contrary, if $\QF$ is infinite, $_K\QD$ is always infinite and $\vert{_KI}\vert\geq{}1$.

       \par c) Actually some vectorial facets (\eg the minimal one $V_0$) are positive and negative.
 So to associate a maximal $K-$facet to a Weyl$-K-$facet, we may have to make a choice of a sign, at least if $_KI_{re}\not={_KI}$.

\begin{enonce*}[plain]{\quad6) Lemma} let $\g G$ be an almost split Kac-Moody group defined over $K'$, with $K'\subset K\subset E$, $E/K'$, $K/K'$ Galois and $(\g G_K,E)$ as above (\cf \ref{1.8}.4).

\par a) Let ${_{K'}A^v}\subset{_{K}A^v}\subset A^v$ be respectively a $K'-$apartment in $_{K'}\SHI^v$, a $K-$apartment in $_{K}\SHI^v$ and an apartment in $\SHI^v$ (stable under $Gal(E/K')$ or not).
 Then the $K-$facets or $K'-$facets are described in $A^v$ as in 1) above with help of $_KI$ or $_{K'}I$.

 \par b) The action of $Gal(E/K')$ on $\SHI^v$ induces an action of $Gal(K/K')$ on $_K\SHI^v$ which may be extended (linearly and uniquely) to $^K\!\SHI^v$.
\end{enonce*}
\begin{proof} There is a star-action of $Gal(E/K)$ (resp. $Gal(E/K')$) on $A^v$ (and its vector space $\vect{A^v}$) and a subset $I_0^K$ (resp. $I_0^{K'}$) of $I$ which describes entirely $_KA^v$ (resp. $_{K'}A^v$); this is independent of the choice of $A^v$, as different choices are conjugated \cite[prop. 6.2.3 (i)]{Ry-02a}.
 They describe also the $K-$facets (resp. $K'-$facets), so a) follows.
  The action of $Gal(K/K')$ on $_K\SHI^v=\g G(K).{_KA^v}$ is described through its action on $_K\SHI^{vc}$ and its star action on $_KA^v$ which may be extended (linearly and uniquely) to $^KA^v$. So b) is a consequence of 3)c above.
\end{proof}

\subsection{Imaginary relative root groups}\label{1.13}

\par{\quad\bf1)} Let's consider $_K\qa\in{_K\QD}^{im}$. The sets $({\Z_{>0}}.{_K\qa})\cap({_K\QD})$ and $(_K\qa)=\{\qb\in\QD\mid{_K\qb}\in {\Z_{>0}}.{_K\qa}\}$ are infinite \cite[3.3.2]{B3R-95}.

\par We saw in \ref{1.4c} that  $\g G_E$ is embedded in some ind-group-scheme.
If $_K\qa$ is positive (resp. negative) we can define in the pro-unipotent group-scheme $\g U_E^{ma+}$ (resp. $\g U_E^{ma-}$) a pro-unipotent subgroup-scheme $\g U_{(_K\qa)E}^{ma}$ such that the elements of $U_{(_K\qa)}^{ma}=\g U_{(_K\qa)E}^{ma}(E)$ are written uniquely as infinite products:
  $u= \prod_{\qb\in(_K\qa)}\,\prod_{j=1}^{j=n_\qb}\;[exp]\ql_{\qb,j}.e_{\qb,j}$ where $(e_{\qb,j})_{j=1,n_\qb}$ is a basis of $\g g_\qb$ ($n_\qb=1$ for $\qb$ real) and $\ql_{\qb,j}\in E$.
Moreover the conjugate of such an element $u\in U_{(_K\qa)}^{ma}$ by $s\in\g S(E)$ is $\prod_{\qb\in(_K\qa)}\,\prod_{j=1}^{j=n_\qb}\;[exp]\qb(s).\ql_{\qb,j}.e_{\qb,j}$.

\par We define the root group corresponding to $_K\qa$ as $V_{_K\qa}=U_{(_K\qa)}^{ma}\cap\g G(K)$.

\begin{enonce*}[plain]{\quad2) Lemma} The group $G=\g G(K)$ has an extra large (abstract) center: it contains $S_Z=\{s\in\g S(K)\mid{_K{\qa_i}}(s)=1\,,\,\forall i\in{_KI_{re}}\}$.
\end{enonce*}
\begin{proof} As $\g S(K)$ is in the center of $_KZ$  and $G$ is generated by $_KZ$ and the groups $V_{_K\qa}$ for $_K\qa\in{_K\QF}$, this result is a consequence of lemma \ref{1.10} (i).
\end{proof}

\par\noindent{\quad\bf3) Remarks and definition} a) The schematic center of $G$, \ie the centralizer in $G$ of $\g G(K_s)$, is $\{s\in\g T(K)\mid{{\qa_i}}(s)=1\,,\,\forall i\in{I}\}$ \cite[9.6.2]{Ry-02a}. Hence its intersection with $\g S(K)$ is smaller than $S_Z$ in general.

\par b) If $_K\qa\in{_K\QD}$, we can write uniquely $_K\qa=\pm{}(\sum_{i\in{_KI}}\,n_i.{_K\qa_i})$ with $n_i\in\Z_{\geq{}0}$.
 We shall say that $_K\qa$ is {\it almost real} and write $_K\qa\in{_K\QD^r}$ if and only if $n_i=0\; \forall i\in {_KI}\setminus{_KI_{re}}$. Hence ${_K\QF}={_K\QD}_{re}\subset{_K\QD^r}\subset{_K\QD}$.
 This set ${_K\QD^r}$ is a system of roots in the sense of \cite[2.4.1]{By-96}.

\par c) By the following lemma the non trivial root groups $V_{_K\qa}$ correspond to roots $_K\qa\in{_K\QD^r}$.
 So it is natural to abandon the $K-$facets (defined using $_K\QD$) and to use the Weyl$-K-$facets of \ref{1.12}.2 (defined using $_K\QF$ or $_K\QD^r$).

 \par We may define $^K\QD=\{{_K\qa}\in{_K\QD}\mid V_{_K\qa}\not=\{1\}\,\}$, so $_K\QF\subset{^K\QD}\subset{_K\QD^r}$ (by the following).

\begin{enonce*}[plain]{\quad4) Lemma} If $_K\qa\in{_K\QD}\setminus{_K\QD^r}$ (hence $_K\qa\in{_K\QD_{im}}$), then $V_{_K\qa}=\{1\}$.
\end{enonce*}

\begin{proof} Suppose $\shs$ free, $K$ infinite and $_K\qa\in{_K\QD}\setminus{_K\QD^r}$, then $\forall n\in\Z_{>0}$ we have $(n.{_K\qa})(S_Z)\not=\{1\}$.
But the conjugation by $s\in S_Z$ of an element of $G$ (resp. $U_{(_K\qa)}^{ma}$) is trivial (resp. given by the formulae in 1) above). Hence $V_{_K\qa}=\{1\}$.

\par When $\shs$ is not free we obtain the same result by using $\g G^{xl}$ \cf \ref{1.6}a.
When $K$ is finite, the (schematic) centralizer $\g Z$ of $\g S$ is a $K-$quasi-split reductive group with $\g S$ as maximal $K-$split torus; so $\g Z$ is a torus.
     \label{N17}   Now $\g Z$ splits over a Galois extension of degree $D$. If $L$ is an infinite union of extensions of degree prime to $D$, $\g S$ is still maximal $K-$split over $L$ and the wanted result is true over $L$. The result over $K$ is then clear.
\end{proof}

\subsection{Associated almost split maximal Kac-Moody groups}\label{1.14}

\par In \ref{1.4c} or \cite{Ru-11}, we associated to $\g G_\shs$ a positive (resp. negative) completion $\g G_\shs^{pma}$ (resp. $\g G_\shs^{nma}$).
 We want to associate to $\g G$ a positive (resp. negative) completion $\g G^{pma}$ (resp. $\g G^{nma}$) considered as a functor from $\shs ep(K)$ to the category of groups.
 For this we have to describe a (twisted) action of $\QG=Gal(K_s/K)$ on $\g G_\shs^{pma}(K_s)$ (resp. $\g G_\shs^{nma}(K_s)$) extending the known one on $\g G_\shs(K_s)$ and to define 
 $\g G^{pma}(L)=\g G_\shs^{pma}(K_s)^{Gal(K_s/L)}$ (resp. $\g G^{nma}(L)=\g G_\shs^{nma}(K_s)^{Gal(K_s/L)})$ for every $L\in\shs ep(K)$.
 
 \par We consider a maximal $K-$split torus $\g S$ in $\g G$, contained in a $K-$defined maximal torus $\g T$. We choose a Weyl$-K-$chamber $^KC^v$ in $^KA^v={^KA^v}(\g S)$ and define $F_1^v$ as the vectorial (spherical) facet in $A^v=A^v(\g T)$ such that $_KF_1^v=F^v_1\cap{_KA^v}$ is open in $^KC^v$ (\cf \ref{1.12}).
 We choose a vectorial chamber $C^v$ in $A^v$ whose closure contains $F^v_1$.
 We identify $\g G_{K_s}$ and $\g G_{\shs K_s}$ in such a way that $\g T_{K_s}=\g T_{\shs K_s}$ and $C^v$ is the fundamental positive chamber $C^v_+$; this defines a (twisted) action of $\QG$ on  $\g G_{K_s}=\g G_{\shs K_s}$.
 As in \ref{1.5}.4 we choose, for any $\qg\in\QG$, an element $g_\qg\in\g G^{xlad}(K_s)$ such that $Int(g_\qg^{-1})\circ\qg$ stabilizes $A^v$ and $C^v$, more precisely $\g T_{K_s}$ and $\g B^+_{\shs K_s}$.
 It is clear, in the above situation, that $g_\qg\in P^{xlad}(F_1^v)\cap\g Z_g^{xlad}(\g S)(K_s)=M^{xlad}(F_1^v)$; moreover $g_\qg=1$ for $\qg$ in some finite index subgroup of $\QG$.
 Now the star action of $\QG$ is defined by $\qg^*=Int(g_\qg^{-1})\circ\qg$ on $\g G(K_s)$ and $\qg^*=Ad(g_\qg^{-1})\circ\qg$ on $\shu_{\shs K_s}$.
 This star action is a true action.
 
 \par By hypothesis $\qg^*(\g T_{K_s})=\g T_{K_s}$, $\qg^*(\g B^+_{K_s})=\g B^+_{K_s}$ and $\qg^*(e_i)=e_{\qg^*i}$ for every $\qg\in\QG$ and $i\in I$.
  By \cite[11.2.5]{Ry-02a}, we have $\qg^*(\qa^\vee_i)=\qa^\vee_{\qg^*i}$, $\qg^*(f_i)=f_{\qg^*i}$ and $\qg^*(x_{±\qa_{i}}(k))=x_{±\qa_{\qg^*i}}(\qg k)$ for $k\in K_s$.
  We deduce that $\qg^*(\widetilde s_{\qa_{i}})=\widetilde s_{\qg^*\qa_{i}}$ and $\qg^*(x_\qa(k))=x_{\qg^*\qa}(\qg k)$, for all $\qa\in\QF$ and $k\in K_s$, \cf \cite[sec. 1]{Ru-11} in particular (KMT7).
  Together with the clear star action on $\g T(K_s)$ this gives a complete description of the star action on $\g G(K_s)$.
  
  \par For $\shu_{\shs K_s}$ the description of the star action is less clear.
  We have $\qg^*(e_\qa)=e_{\qg^*\qa}$ and $\qg^*(\qa^\vee)=(\qg^*\qa)^\vee$ for all $\qa\in\QF$.
  The condition (ALG3') of \ref{1.7} (if it is assumed) tells us that $\qg^*$ on $Y\otimes K_s\subset\shu^0_{K_s}$ is given by $\qg^*$ on $Y$ \ie on $\g T_\shs$.
  If $\qa(Y)$ is non zero in $K$ (\eg if $char(K)≠0$) the formula $\qg^*(x_\qa(k))=x_{\qg^*\qa}(\qg k)$ and the $\QG^*-$equivariance of $Ad$ tells us that $\qg^*(e_{\qa}^{(n)})=e_{\qg^*\qa}^{(n)}$ for all $\qa\in\QF$ and $n\in\N$.
  More precisely in characteristic $0$ the condition (ALG3') tells us that $\qg^*$ on $\shu_{\shs K_s}$ is as we want: the $\qg^{**}$ of \cite[13.2.3]{Ry-02a} entirely defined by the star action on $I$ and $X$ or $Y$.
  In positive characteristic, particularly in characteristic $2$, the equality of $\qg^{*}$ and $\qg^{**}$ is far from obvious, \eg for imaginary roots.
  
  \par We have two solutions to this problem. First add an axiom (ALG3'') (involving (ALG3')) telling that $\qg^{*}=\qg^{**}$.
  We choose the second solution: we change the star action on $\shu_{\shs K_s}$, we take $\qg^{**}$ instead of  $\qg^{*}$; the description we gave of $\qg^{*}$ on $\g G(K_s)$ tells us that $Ad$ is still $\QG^*-$equivariant.
  Then we define $\widetilde\qg=Ad(g_\qg)\circ\qg^{**}$; this gives an action as $g_{\qg\qg'}=g_\qg.\qg^*(g_{\qg'})$ (\ref{1.5}.4) and $Ad$ is $\widetilde\QG-$equivariant.
  Moreover the orbits of $\widetilde\QG$ on $\shu_{\shs K_s}$ are finite, hence we have defined a new $K-$form $\shu'_K$ of $\shu_{\shs K_s}$.
  With the precise definition of Kac-Moody groups given in \cite{Ry-02a}, we have changed the Kac-Moody group, but the functor $\g G$ is still the same.
  
  \par The definition of an almost split maximal Kac-Moody group is now clear: the $\QG^{**}-$action on $\shu_{\shs K_s}$ induces clearly an action on $\g G^{pma}(K_s)=\g G^{}(K_s).\g U^{ma+}(K_s)$ or on $\g G^{nma}(K_s)=\g G^{}(K_s).\g U^{ma-}(K_s)$ which coincides on $\g G(K_s)$ with the known $\QG^*-$action (see the definition of these groups in \cite{Ru-11}).
  Then we define the action of $\QG$ by $\widetilde\qg=Int(g_\qg)\circ\qg^{**}$.

\section{Valuations and affine (bordered) apartments}\label{s2}

\begin{defi}\label{2.1} A {\it valuation} of a root datum $(G,(U_\qa)_{\qa\in\QF},Z)$ of type a real root system $\QF$ is a family $(\qf_\qa)_{\qa\in\QF}$ of maps $\qf_\qa:U_\qa\to\R\cup\{+\infty\}$ satisfying the following axioms:

\par(V0) $\forall\qa\in\QF$, $\vert\qf_\qa(U_\qa)\vert\geq{}3$

\par(V1) $\forall\qa\in\QF$, $\forall\ql\in\R\cup\{+\infty\}$, $U_{\qa,\ql}=\qf_\qa^{-1}([\ql,+\infty])$ is a subgroup of $U_\qa$ and $U_{\qa,\infty}=\{1\}$

\par(V2.1) $\forall\qa,\qb\in\QF$, $\forall u\in U_\qa\setminus\{1\}$, $\forall v\in U_\qb\setminus\{1\}$,
\par\qquad\qquad\qquad $\qf_{r_\qa(\qb)}(m(u)vm(u)^{-1})=\qf_\qb(v)-\qb(\qa^\vee)\qf_\qa(u)$

\par(V2.2)  $\forall\qa\in\QF$, $\forall t\in Z$, the map $v\mapsto\qf_\qa(v)-\qf_\qa(tvt^{-1})$ is constant on $U_\qa\setminus\{1\}$

\par(V3) For each prenilpotent pair of roots $\{\qa,\qb\}$ and all $\ql,\qm\in\R$, the commutator group $[U_{\qa,\ql},U_{\qb,\qm}]$ is contained in the group generated by the groups $U_{p\qa+q\qb,p\ql+q\qm}$ for $p,q\in\Z_{>0}$ and $p\qa+q\qb\in\QF$

\par(V4) If $\qa\in\QF$ and $2\qa\in\QF$, then $\qf_{2\qa}$ is the restriction of $2\qf_\qa$ to $U_{2\qa}$.
\end{defi}

\begin{remas*} 1) This definition appears in \cite[10.2.1]{Cn-10b}. A weaker definition is given in \cite[2.2]{Ru-06}; there, axiom (V2.1) is replaced by axioms named (V2a) and (V5).
In the classical case, both definitions are equivalent to the original one of \cite[6.2.1]{BtT-72}, \cf \cite[10.2.3.2]{Cn-10b}. Actually (V2.1) is then the proposition 6.2.7 of \cite{BtT-72}.
 This definition may be extended to RGD-systems for a family $(\qf_\qa)_{\qa\in\QF_{nd}}$: in (V3) just allow $p$ and $q$  to be in $\R_{>0}$; in (V2.1) if $r_\qa(\qb)=\ql\qg$ with $\ql>0$ and $\qg\in\QF_{nd}$, replace $\qf_{r_\qa(\qb)}$ by $\ql\qf_\qg$.

\par 2)We define $\QL_\qa=\qf_\qa(U_\qa\setminus\{1\})\subset\R$. From (V2.1) with $\qa=\qb$, $u=v$ we get $\QL_\qa=-\QL_{-\qa}$.
 For $u,u',u''$ as in \ref{1.3} (RD4), we have $\qf_{-\qa}(u')=\qf_{-\qa}(u'')=-\qf_{\qa}(u)$, \cite[11.1.11]{Cn-10b}.
 For $\ql\in\R$, we set $U_{\qa,\ql+}=\qf_\qa^{-1}(]\ql,+\infty])$

\par 3) Let $Q=\Z\QF$ be the $\Z-$module generated by $\QF$ and $V^q=(Q\otimes\R)^*$. Then using this (strong) definition one can build an action of the group $N$ (defined in \ref{1.3c}.1) over $V^q$ (this seems impossible with the weaker definition of \cite{Ru-06}):
\end{remas*}

\begin{prop}\label{2.1b}  \cf \cite[prop. 11.1.9, 11.1.10]{Cn-10b} There exists a unique action $\qn^q$ of $N$ over $V^q$ by affine transformations such that:

- $\forall t\in Z$, $\qn^q(t)$ is the translation by the vector $\vect{v_t}$ such that $\qa(\vect{v_t})=\qf_\qa(u)-\qf_\qa(tut^{-1})$, $\forall\qa\in\QF$ and $\forall u\in U_\qa\setminus\{1\}$,

- $\forall n\in N$, $\qn^q(n)$ is an affine automorphism with associated linear map $\vect{\qn^q(n)}=\qn^v(n)$.

\end{prop}

\subsection{Valuation for a split Kac-Moody group}\label{2.2}

\par Let $\g G=\g G_\shs$ be a split Kac-Moody group over $K$, as in \ref{1.2}.
We suppose the base field $K$ endowed with a non trivial real valuation $\qo=\qo_K:K\to\R\cup\{+\infty\}$. Its {\it ring of integers} (resp. {\it maximal ideal}, {\it residue field}) is $\mathcal{O}=\mathcal{O}_K=\qo^{-1}([0,+\infty])$ (resp. $\g m=\g m_K=\qo^{-1}(]0,+\infty])$, $\qk=\sho/\g m$) and $\QL=\QL_K=\qo(K^*)$ is its {\it value group}. An important particular case (the {\it discrete case}) is when $\QL$ is discrete in $\R$.

\par Let $u=x_\qa(r)\in U_\qa$ with $\qa\in\QF$ and $r\in K$, we set $\qf_\qa(u)=\qo(r)\in\R\cup\{+\infty\}$.

\begin{prop*} The family $(\qf_\qa)_{\qa\in\QF}$ is a valuation of the root datum $(G,(U_\qa)_{\qa\in\QF},T)$.
\end{prop*}
\begin{proof} Clear except for (V2.1) proved in  \cite[10.2.3.1]{Cn-10b}.
\end{proof}
\begin{rema*} We have $\QL_\qa=\QL$, $\forall\qa\in\QF$.
\end{rema*}

\subsection{Affine apartments}\label{2.3}

\par We consider an abstract valuated root datum as in \ref{2.1}.

\par {\bf1)} Let $V=\vect\A$ be a real vector space with $\QF\subset Q\subset V^*$ and $(\qa_i^\vee)_{i\in I}\subset V$ as in \ref{1.1b}.4; we consider in $V$ all objects defined in \ref{1.1b}.

\par For $\ql\in\R$ and $\qa\in Q\setminus\{0\}$, we define the affine hyperplane $M(\qa,\ql)=\{x\in V\mid \qa(x)+\ql=0\}$ of direction Ker$\qa$, the closed half-space $D(\qa,\ql)=\{x\in V\mid \qa(x)+\ql\geq{}0\}$ and its interior $D^\circ(\qa,\ql)=\{x\in V\mid \qa(x)+\ql>0\}$.
For $\qa\in\QF$ the reflection $s_M=s_{\qa,\ql}$ with respect to $M=M(\qa,\ql)$ is the affine reflection with associated linear map $\vect{s_M}=s_\qa$ and with fixed point set $M$.

\par We suppose $V$ endowed with an action $\qn$ of $N$ such that, $\forall n\in N$, $\qn(n)$ is an affine automorphism with associated linear map $\vect{\qn(n)}=\qn^v(n)$.
 We ask moreover that, for $\qa\in\QF$ and $u\in U_\qa\setminus\{1\}$, $\qn(m(u))$ is the reflection $s_{\qa,\qf_\qa(u)}$. We write $Z_0=$Ker$\qn\subset Z$.

\par Then $t\in Z=$ Ker$(\qn^v)$ acts on $V$ by a  translation of vector $\vect{v_t}$.
The action $\qn$ commutes with the translations by $V_0$ and the induced action on the essential quotient $V^q=V/V_0$ is $\qn^q$ as defined in proposition \ref{2.1b}:
as $m(tut^{-1})=tm(u)t^{-1}$, we have clearly $\qa(\vect{v_t})=\qf_\qa(u)-\qf_\qa(tut^{-1})$. \label{N5}

\par As a consequence, $\forall n\in N$, $\forall\qa\in\QF$ and $\forall u\in U_\qa\setminus\{1\}$ we have $\qn(n).D(\qa,\qf_\qa(u))=D(\qn^v(n).\qa,\qf_{\qn^v(n).\qa}(nun^{-1}))$ and the same thing for the walls \cite[11.1.10]{Cn-10b}.

\par For $v\in V$, we may define a new valuation $\qf'$ ({\it equipollent} to $\qf$) by $\qf'_\qa(u)=\qf_\qa(u)+\qa(v)$ for $\qa\in\QF$. This corresponds to choosing for $V$ a new origin $0_{\qf'}=v$.

\begin{enonce*}[definition]{\quad2) Definitions} a) A {\it wall} (resp. an {\it half-apartment}) in $V$ is an hyperplane (resp. a closed half-space) of the form $M(\qa,\qf_\qa(u))=V^{m(u)}$ (fixed point set) (resp. $D(\qa,\qf_\qa(u))$ with $\qa\in\QF$ and $u\in U_\qa\setminus\{1\}$.  The action $\qn$ of $N$ permutes the walls and half-apartments.

 \par More generally, for $(V,W^v,({{\alpha_i}})_{i\in I},({\alpha}{_i^\vee})_{i\in I})$ as in \ref{1.1b}.4 and a family $\QL=(\QL_\qa)_{\qa\in\QF}$ of infinite subsets of $\R$, we say that $M(\qa,\ql)$ (resp. $D(\qa,\ql)$) is a {\it wall} (resp. {\it half-apartment}) if, and only if, $\ql\in\QL_\qa$. 
 We use this definition only when, for any $\qa\in\QF$, $\ql\in\QL_\qa$, the reflection $s_{\qa,\ql}$ stabilizes this set $\SHM$ of walls.

\par b) The {\it affine apartment} $\A$ 
is $V$ considered as an affine space and endowed with its family $\SHM$ of walls and the corresponding reflections.
 It is called {\it semi-discrete} if, $\forall\qa\in\QF$, the set of walls of direction Ker$\qa$ is locally finite, \ie if $\QL_\qa$ is discrete in $\R$.
  Its essentialization is $\A^e=\A/V_0$ endowed with the image of the family $\shm$.

\par A preorder is defined on  $\A$ (or $\A^e$) by \quad $x\leq{}y\iff y-x\in\sht_+$.

\par There is also a more restrictive preorder: \quad$x\stackrel{o}{<} y\;\Leftrightarrow\; y-x\in\mathcal T^o_+$.

\par c) An {\it automorphism} of $\A$ is an affine bijection $\qf:\A\to\A$ stabilizing the family $\SHM$ of walls and conjugating the corresponding reflections.
We ask also that its associated linear map $\vect\qf$ stabilizes $\QF$ (this is automatic in the semi-discrete case with $\QF$ reduced and $\QL_\qa$ independent of $\qa$) and the union $\sht_+\cup\sht_-$ of the Tits cones (this is automatic in the classical case).
 Then $\vect\qf$ normalizes the vectorial Weyl group $W^v$ and transforms vectorial facets into vectorial facets.

 \par d) We say that an automorphism $\qf$ is {\it positive} (or of first kind) (resp. {\it vectorial-type-preserving, vectorially Weyl}) if $\vect{\qf}(\sht_\pm{})=\sht_\pm{}$ (resp. $\vect\qf$ preserves the types of the vectorial facets, $\vect\qf\in W^v$).

 \par e) The {\it (affine) Weyl group} $W^a=W^a(\A)$ of $\A$ is the subgroup of Aut$(\A)$ generated by the reflections $s_M$ for $M\in\SHM$. Its elements are called  {\it Weyl-automorphisms} of $\A$.

\par f) An {\it apartment of type} $\A$ is a set $A$ endowed with a set $Isom_{W^a}(\A,A)$ of bijections $f:\A\to A$ (called {\it Weyl isomorphisms}) such that if $f_0\in Isom_{W^a}(\A,A)$, then $f\in Isom_{W^a}(\A,A)$ if, and only if, there exists $w\in W^a$ such that $f=f_0\circ w$.

\par g) An {\it isomorphism} between two apartments $A$ and $A'$ is a bijection $\qf:A\to A'$ such that for some $f_0\in Isom_{W^a}(\A,A)$ and $f'_0\in Isom_{W^a}(\A,A')$ (the choices have no importance) the map $(f'_0)^{-1}\circ\qf\circ f_0$ is an automorphism of $\A$.
 We say that this $\qf$ is {\it positive}, {\it vectorial-type-preserving}, {\it vectorially Weyl} or a {\it Weyl isomorphism} if $(f'_0)^{-1}\circ\qf\circ f_0$ is positive, vectorial-type-preserving, vectorially Weyl or a Weyl automorphism (compare with \cite[1.13]{Ru-10}); actually it is sufficient to verify this property by restriction to a non empty open convex subset of $A$.

  \end{enonce*}

    \begin{enonce*}[definition]{\quad3) Remarks}
a) By definition $N$ and $W^a$ act on $\A$ by vectorially Weyl automorphisms. The Weyl group $W^a$ is a normal subgroup of $\qn(N)$. They are not always equal.

    \par b) If $\qf$ is an automorphism of $\A$, $\vect\qf$ stabilizes $\QF$ and $\sht_+\cup\sht_-$; so, up to $W^v$, $\pm{}\vect\qf$ stabilizes the basis $(\qa_i)_{i\in I}$ and $\QS=\{s_i\mid i\in I\}$ or $(\qa_i^\vee)_{i\in I}$.
    The induced action on $I$ preserves the Kac-Moody matrix $\M$ and its Dynkin diagram (it is a diagram automorphism); it is trivial if $\qf$ is vectorial-type-preserving.
   A vectorially Weyl automorphism is positive and vectorial-type preserving; the converse is true when $\A$ is essential.

    \par c) We define $G^\emptyset$ (resp. $N^\emptyset$) as the subgroup of $G$ (resp. $N$) generated by $Z_0$ and the groups $U_\qa$ (resp. by $Z_0$ and the $m(u)$, $u\in U_\qa\setminus\{1\}$) for $\qa\in\QF$.
    It is normal in $G$ (resp. $N$) and $G=G^\emptyset.Z$ (resp. $N=N^\emptyset.Z$).
    By definition $\qn(N^\emptyset)=W^a$ and even $N^\emptyset=\qn^{-1}(W^a)$ is the group of Weyl automorphisms in $N$. We set $Z^\emptyset=N^\emptyset\cap Z$ which is normal in $Z$.

    \par By \cite[6.1.2 (12)]{BtT-72} $(G^\emptyset,(U_\qa)_{\qa\in\QF},Z^\emptyset)$ is a generating root datum of type $\QF$.
    The associated group "N" (as in \ref{1.3c}.1) is $N^\emptyset$.
    Comparing the refined Bruhat decompositions (\ref{1.3c}.1) of $G^\emptyset$ and $G$, we obtain $G^\emptyset\cap N=N^\emptyset$. Compare with \cite[6.2]{Ru-10}.

    \end{enonce*}

  \par {\bf 4) Imaginary roots}
  We consider moreover a set $\QD_{im}$  in $V^*$ of {\it imaginary roots} with $\QD_{im}\cap(\cup_{\qa\in\QF}\,\R\qa)=\emptyset$ and $\QD_{im}$ $W^v-$stable; we write $\QD_{re}=\QF$ and $\QD=\QF\cup\QD_{im}$.
 The best example for $\QD$ is  a root system as in \cite{By-96} with $\QF$ as system of real roots   (it can be \eg the root system generated by $\QF$ as in \ref{1.1}.3 or, if $\QF={_K\QF}$, the system $_K\QD$ as in \ref{1.9}).
  The {\it totally imaginary} choice $\QD^{ti}$ for $\QD$ corresponds to $\QD^{ti}_{im}=V^*\setminus(\cup_{\qa\in\QF}\,\R\qa)$.

\par  We say that $\QD$ is {\it tamely imaginary} \cite[1.1]{Ru-10} (resp.  {\it relatively imaginary}) if $\QD_{im}=\QD^+_{im}\cup \QD^-_{im}$ with $W^v-$stable sets $\QD^{\pm}_{im}=\pm{}(\QD\cap(\oplus_{i\in I}\,\R^+\qa_i))$  (resp. $\QD^{\pm}_{im}=\pm{}(\QD\cap(\oplus_{i\in I_\QD}\,\R^+\qa_i))$, where $I_\QD\supset I$ is finite and $(\qa_i)_{i\in I_\QD}$ is free).
   Remark that $_K\QD$ (as defined in \ref{1.9}.2) is always relatively imaginary and is tamely imaginary if and only if it is equal to $_K\QD^r$: \ref{1.13}.3b.

  \par For all $\qa\in\QD_{im}$, we consider an infinite subset $\QL_\qa=-\QL_{-\qa}$ of $\R$.
  We define the system $\mathcal M^i$ of {\it imaginary walls} as the set of affine hyperplanes $M(\qa,\ql)$ for $\qa\in\QD_{im}$ and $\ql\in\QL_\qa$ (actually the real walls are given by the same formula for $\qa\in\QF$).
  We ask that these walls are permuted by $\qn(N)\supset W^a$, in particular  $\QL_{w\qa}=\QL_\qa$, $\forall w\in W^v$.

 \par For $\qa\in\QD$ and $k\in\R$, we sometimes say that $M(\qa,k)$ (resp. $D(\qa,k)$) is a {\it true} or {\it ghost} wall (resp. half-apartment), according to the fact that $k\in\QL_\qa$ or $k\not\in\QL_\qa$.

  \begin{enonce*}[definition]{\quad5) Remarks}
a) Actually these imaginary roots or walls will be used only to define enclosures, hence facets and chimneys (\ref{2.4}). 
 So making a difference between true or ghost imaginary walls is often useless, \eg in the case of \ref{2.3b}, see \ref{2.4}.1.
It would be possible to modify the vectorial facets (hence the sectors, facets, chimneys,...) with $\QD_{im}$ (as in \ref{1.12}) in the relatively imaginary case (this changes nothing in the tamely imaginary case). But it seems useless for us: see \ref{1.13}.4 and section \ref{s6}.

   \par b) Let $\qf$ be an automorphism of $\A$. Then $\vect\qf$  stabilizes $\QD_{im}^+$ and $\QD_{im}^-$ or exchanges them if $\qf$ is a Weyl automorphism (by definition) or if
$\QD$ is generated by $\QF$ \cite[4.2.15, 4.2.20 and 2.4.1]{By-96}.
 We say in general that $\qf$ is {\it imaginary-compatible} if $\vect\qf(\QD^+_{im})=\QD^\pm_{im}$ and $\qf$ permutes the imaginary walls (automatic \eg if $\QL_\qa=\R$, $\forall\qa\in\QD_{im}$).


    \end{enonce*}

    \subsection{Affine apartments for a split Kac-Moody group}\label{2.3b}

    \par We consider the group and valuation as in \ref{2.2}.

\par {\bf1)} We can build easily examples of pairs $(V,\qn)$ as in \ref{2.3}.1.
We choose a commutative extension of RGS $\qf:\shs\to\shs'=(\M,Y',(\qa'_i)_{i\in I},(\qa_i^{'\vee})_{i\in I}$ with $\shs'$ free and we set $V=Y'\otimes\R$.

\par There is an action $\qn_T$ of $T$ over $V$  by translations: for $t\in T$, $\qn_T(t)$ is the translation of vector $\qn_T(t)$ such that $\chi(\qn_T(t))=-\qo(\overline{\chi}(t))$ for $\chi\in X'$
 and $\overline\chi=\qf^*(\chi)\in X$.
 In other words $\qn_T$ is the map $-(\qf\otimes\qo)$ from $T=Y\otimes_\Z K^*$ to $V=Y'\otimes_\Z\R$.
  This action is $W^v-$equivariant.

  \par By  \cite[2.9]{Ru-06} there exists an affine action $\qn$ of $N$ over $V$ whose restriction to $T$ is $\qn_T$ and satisfying the properties asked in \ref{2.3}.1.
   Actually $N/$Ker$\qn_T$ is a semi-direct product by $T/$Ker$\qn_T$ of a group isomorphic to $W^v$ and generated by the images of $m(x_{\qa_i}(\pm{}1))$ for $i\in I$. This last group will fix the origin of $V$.


\par For $\shs'$ we may choose $\shs_{\M m}$, $\shs^l$ or (if $\shs$ is free) $\shs$ itself.
 We get thus $V=V^q,V^{xl}$ or $V^x$ and corresponding affine apartments $\A=\A^q,\A^{xl}$ or $\A^x$

  \begin{enonce*}[definition]{\quad2) Remarks}  Suppose $(V,\qn)$ as in 1) above.

  \par a) The kernel $Z_0=$Ker$\qn_T=$ Ker$\qn$ of $\qn$ contains the group $\g T(\sho)=Y\otimes\sho^*\simeq(\sho^*)^n$ of points of $\g T$ over $\sho$.
   It is actually equal to it except when  the image of the map $\qf^*:X'\to X$, $\chi\mapsto\overline\chi$ has infinite index \ie when $\qf$ is not injective.

  \par b)  We have $\qn(N)=W^v\ltimes(\overline{Y}\otimes_\Z\QL)$ and $W^a=W^v\ltimes(\overline{Q^\vee}\otimes_\Z\QL)$, where $\overline{Y}$ (resp. $\overline{Q^\vee}$) is the image by $\qf$ of $Y$ (resp. $Q^\vee=\sum_{i\in I}\,\Z\qa_i^\vee\subset Y$) in $V$.
   So there is equality in the simply connected case (in a strong sense: $Y=Q^\vee$) and only in this case when $\qf$ is injective (\eg $V=V^{xl}$ or $V=V^x$) and $\qo$ discrete.

   \par{\bf 3) General affine apartments}  a)  We consider now any pair $(V,\qn)$ as in \ref{2.3}.1.
   But we add the condition (useful in section \ref{s5}) that the kernel $Z_0=$ Ker$\qn$ contains $\g T(\sho)$.
   We speak then of a {\it suitable apartment} for $(\g G_\shs,\g T_\shs)$; apartments defined in 1) are suitable.

   \par Then $\qn\,\rule[-1.5mm]{.1mm}{3mm}_{\, T}$ induces a $\Z-$linear map $\overline\qn:Y\otimes\QL\to V$ and this map sends $\qa_i^\vee\otimes\ql$ to $-\ql\qa_i^\vee$: $\qa_i^\vee\otimes\ql$ is the class modulo $\g T(\sho)$ of $\qa_i^\vee(r)\in\g T(K)$ with $\qo(r)=\ql$.
   But $\qa_i^\vee(r)=m(x_{-\qa_i}(1))^{-1}.m(x_{-\qa_i}(r))$ by \cite[1.5, 1.6]{Ru-11}, \label{N13} so by the hypothesis in \ref{2.3}.1, $\qn(\qa_i^\vee(r))=s_{-\qa_i,0}\circ s_{-\qa_i,\qo(r)}$ which is the translation of vector $-\ql\qa_i^\vee$.
  In particular the $\Z-$linear relations  between the $\qa_i^\vee$  in $Y$ are also satisfied in $V$.

   \par By \ref{2.3}.1 and \ref{2.2}, we have also $\qa(\overline\qn(y\otimes\ql))=-\qa(y).\ql$.

   \par b) We choose $\QD_{im}$ as in \ref{1.1}.3 \ie generated by $\QF$ \cite[5.4]{K-90}, \cite[2.4.1]{By-96}. We have $\QL_\qa=\QL$, $\forall\qa\in\QF$ and we set
$\QL_\qa=\QL$, $\forall\qa\in\QD_{im}$. The system $\mathcal M$ of walls is discrete (resp. semi-discrete) if and only if we are in the classical discrete case (resp. if the valuation is discrete).

\par   If $\qf$ is an automorphism of $\A$ and $\qa\in\QD$, $\ql\in\QL$, then $x\in\qf(M(\qa,\ql))\iff0=\qa(\qf^{-1}(x))+\ql=\qa(\qf^{-1}(0))+\vect\qf(\qa)(x)+\ql$.
So $\qf(M(\qa,\ql))$ is a (real or imaginary) wall (of direction Ker$\vect\qf(\qa)$) if and only if $\qa(\qf^{-1}(0))\in\QL$.
By hypothesis this is true for $\qa\in\QF$, so this is also true for $\qa\in\QD_{im}\subset Q$; hence $\qf$ permutes the imaginary walls. Therefore any automorphism of $\A$ is imaginary-compatible.


    \end{enonce*}

    \subsection{Enclosures, facets, sectors and chimneys}\label{2.4}

    \par We come back to the general abstract case of \ref{2.3}; the following notions depend only on $\A$ (with $\cal M$) and $\mathcal M^i$.

    \par We consider filters in $\A$ as in \cite{GR-08} or \cite{Ru-08}, \cite{Ru-10}, \cite{Ru-11}. The reference for the following is  \cite{Ru-11} or \cite{Ru-10}. The {\it support} of a filter in $\A$ is the smallest affine subspace in $\A$ containing it.
    We identify a subset in $\A$ to the filter whose elements are the subsets of $\A$ containing this subset. We use definitions for filters (inclusion, union, closure, (pointwise) fixation or stabilization by a group) which coincide with the usual ones for sets when these filters are associated to subsets.

    \par {\bf1)} If $F$ is a filter in $\A$, we define several types of {\it enclosures} for $F$ (corresponding to different choices for a greater family of real or imaginary walls) \cf \cite[4.2.5]{Ru-11}:
    if $\QF\subset\shp\subset \QD$ and, $\forall \qa\in\shp$, $\QL_\qa\subset\QL'_\qa\subset\R$, then  $cl^\shp_{\QL'}(F)$ is the filter made of the subsets of $\A$ containing an element of $F$ of the form $\cap_{\qa\in\shp}\,D(\qa,\ql_\qa)$ with, for each $\qa\in\shp$, $\ql_\qa\in\QL'_\qa\cup\{+\infty\}$;
    in particular each $D(\qa,\ql_\qa)$ contains  the filter $F$ \ie is an element of this filter.
    When $\QL'_\qa=\QL_\qa$ (resp. $\QL'_\qa=\R$) $\forall\qa$, we write $cl^\shp:=cl^\shp_\QL$ (resp. $cl^\shp_\R:=cl^\shp_{\QL'}$); when $\QL'_\qa=\QL_\qa$, $\forall\qa\in\QF$ and $\QL'_\qa=\R$,  $\forall\qa\in\QD^{im}$ we write $cl^\shp_{ma}:=cl^\shp_{\QL'}$.
    We define $\QL_\qa''=\{k\in\R\mid cl^\shp_{\QL'}(D(\qa,k))=D(\qa,k)\}$ for $\qa\in\QD$, then $\QL_\qa''=\QL_\qa'\cup\frac{1}{2}\QL'_{2\qa}$ for $\qa\in\QF$  and $cl^\shp_{\QL'}=cl^\shp_{\QL''}$.
    In the case of \ref{2.3b}.3b with $\QL'=\QL$, $cl^\shp_{ma}=cl^\shp_{\QL''}=cl^\shp_{\QL'}$ \cite[1.6.2]{Ru-10}.

  \par   We define $cl^\#(F)$ (resp. $cl^\#_\R(F)$) as the filter made of the subsets of $\A$ containing an element of $F$ of the form $\cap_{j=1}^k\;D(\qb_j,\ql_j)$ for $\qb_j\in\QF$ and $\ql_j\in\QL_{\qb_j}\cup\{+\infty\}$ (resp. $\ql_j\in\R\cup\{+\infty\}$);
    $cl^\#$ is the enclosure map used by Charignon \cite[sec. 11.1.3]{Cn-10b}.

    \par In \cite{GR-08} (resp. \cite{Ru-10} or \cite{Ru-11}) one uses $cl^\QD_{}$ (resp. $cl^\QD_{ma}$, $cl^\QD_{\R}$, $cl^\QF_{}$, $cl^\QF_{\R}$ or $cl^\QD_{}$, $cl^\QF_{}$, $cl^\#_{}$) under the names $cl$ (resp. $cl$, $cl_\R$, $cl^{si}_{}$, $cl^{si}_{\R}$ or $cl$, $cl^{si}_{}$, $cl^{\#}_{}$).

  \par   One has: $cl^\#(F)\supset cl^{\QF}(F)\supset cl^{\QD}(F)\supset cl^{\QD}_{ma}(F)\supset cl^{\QD}_\R(F)\supset cl^{\QD^{ti}}_\R(F)=\overline{conv}(F)$ (closed convex hull), $cl^{\QF}(F)\supset cl^{\QF}_\R(F)\supset cl^\QD_\R(F)$ and some other clear inclusions.

    \par The maps $cl^{\QD^{ti}}_\R=\overline{conv}$, $cl^{\QF}$,$cl^{\QF}_\R$,  $cl^\#$, $cl^\#_\R$ (resp. $cl^{\QD}$, $cl^{\QD^{ti}}$, $cl^{\QD}_{ma}$, $cl^{\QD}_\R$) are equivariant with respect to automorphisms (resp. imaginary-compatible automorphisms) of $\A$.

        \par In the following, we choose  one of these enclosure maps ($cl^\shp_{\QL'}$, $cl^\#$ or $cl^\#_\R$) which we call $cl$. We say that $F$ is {\it enclosed} or  {\it $cl-$enclosed} if $F=cl(F)$.
        
        \par Actually $\SHM$ is entirely determined by $cl$ and $\QF\subset\vect\A^*$ (provided that $\QL_\qa'=\QL_\qa$ for $\qa\in\QF$): a (true) wall in $\SHM$ is an affine hyperplane which is $cl-$enclosed and with direction Ker$\qa$ for $\qa\in\QF$.
   The enclosure map does not always determine $\SHM^i$, but this set of imaginary walls is only used to determine $cl$.

    \par {\bf2)} A {\it local-facet} is associated to a point $x$ in $\A$ and a vectorial facet $F^v$ in $\vect\A$; it is the filter $F^l(x,F^v)=germ_x(x+F^v)$ intersection of $x+F^v$ with the filter of neighbourhoods of $x$ in $\A$.

    \par The {\it facet} or {\it $cl-$facet} associated to $F^l(x,F^v)$ and the enclosure map $cl=cl^\shp_{\QL'}$ (resp. $cl=cl^\#$ or $cl=cl^\#_\R$) is the filter $F(x,F^v)=F^\shp_{\QL'}(x,F^v)$ (resp. $F^\#(x,F^v)$ or $F^\#_\R(x,F^v)$) made of the subsets containing an intersection (resp. a finite intersection) of half spaces $D(\qa,\ql_\qa)$ or $D^\circ(\qa,\ql_\qa)$
     (at most one $\ql_\qa\in\QL''_\qa$ for each $\qa\in\shp$) (resp. with $\qa\in\QF$ and $\ql_\qa\in\QL_\qa$ or $\ql\in\R$)
     such that this intersection contains $F^l(x,F^v)$ \ie a neighbourhood of $x$ in $x+F^v$.

        \par The {\it closed-facet}  $\overline{F}(x,F^v)$ is the closure of $F(x,F^v)$, also $\overline{F}(x,F^v)=cl(F(x,F^v))=cl(F^l(x,F^v))$.
         Note that $F^l=F^{\QD^{ti}}_\R\subset F^{\QD}_\R\subset F^{\QF}_\R=F^{\#}_\R=F^l+V_0$ and 
        $\overline F^l=\overline F^{\QD^{ti}}_\R\subset \overline F^{\QD}_\R\subset \overline F^{\QF}_\R=\overline F^{\#}_\R=\overline F^l+V_0$,
        where $V_0$ is as defined in \ref{1.1b}.4.

        \par These  facets are called {\it spherical} (resp. {\it positive}, {\it negative}) if $F^v$ is. When $F^v$ is a vectorial chamber, these facets are {\it chambers} hence spherical.

        \par {\bf3)} A {\it sector} (resp. {\it sector-face}) is a $V-$translate $\g q=x+C^v$ (resp. $\g f=x+F^v$) of a vectorial chamber $C^v$ (resp. vectorial facet $F^v$). A {\it shortening} of a sector or sector-face $\g f=x+F^v$ is a sector or sector-face $\g f'=x'+F^v$ included in $\g f$.
        The {\it germ} of a sector $\g q=x+C^v$ (resp. sector-face $\g f=x+F^v$) is the filter $\g Q=germ_\infty(\g q)$ (resp. $\g F=germ_\infty(\g f)$) made of the subsets containing shortenings of $\g q$ (resp. $\g f$). The {\it direction} of $\g f=x+F^v$ or of its germ is $F^v$, its {\it sign} is the sign of $F^v$. When $F^v$ is spherical, we say that $\g f$ and $\g F$ are spherical or {\it splayed} ("\'evas\'e" in \cite{Ru-10} or \cite{Ru-11}).
     The {\it vertex} $x$ of $\g f=x+F^v$ is well defined by $\g f$ when $\A$ is essential.

        \par {\bf4)} A {\it chimney} or  {\it $cl-$chimney} is associated to a facet $F=F(x,F^v_0)$ (its {\it base}) and a vectorial facet $F^v$; it is the filter $\g r(F,F^v):=cl(F+F^v)=cl(F^l(x,F_0^v)+F^v)=:\g r(F^l,F^v)$ (containing $cl(F)+\overline{F^v}=\overline{F}+\overline{F^v}$).
        If $cl=cl^\shp_{\QL'}$, we write $\g r^\shp_{\QL'}(F,F^v)=\g r(F,F^v)$.

        \par A {\it shortening} of $\g r(F,F^v)$ (with $F=F(x,F_0^v)$)  is defined by $\xi\in\overline{F^v}$, it is the chimney $\g r(F(x+\xi,F_0^v),F^v)$.
        The germ of this chimney is the filter $\g R(F,F^v)$ made of the subsets containing a shortening of $\g r(F,F^v)$.
        The {\it direction} of $\g r(F,F^v)$ or $\g R(F,F^v)$ is $F^v$, its {\it sign} is the sign of $F^v$, it is said {\it splayed} if $F^v$ is spherical and {\it solid} (resp. {\it full}) if the direction of its support has a finite fixer in $W^v$ (resp. if its support is $\A$).

        \par\label{N4} For example the enclosure $cl(\g f)$ of a sector-face $\g f=x+F^v$ is a chimney of direction $F^v$; its germ is splayed 
         if and only if $\g f$ is spherical, it is full if (but not only if) it is a sector. A facet is a chimney and a chimney germ with direction the minimal vectorial facet $V_0=F^v_\pm{}(I)$; it is splayed or solid if and only if it is spherical, it is full if and only if it is a chamber.

        \begin{NB} In \cite{Ru-77} a chimney is a specific set among the sets of the chimney as defined above, the chimney germs are the same. In \cite{CL-11} P.E. Caprace and J. L\'ecureux introduce (generalized) sectors in any combinatorial building; in the classical discrete case for $\A$, these sectors are the enclosures of a facet and a chimney germ.
        \end{NB}

        \subsection{Bordered apartments}\label {2.5}

        \par Following \cite{Cn-10b}, we shall add to $\A$ some other apartments at infinity, see also \cite{Ru-10}.

        \par {\bf1)} {\bf Fa\c{c}ades} : For $F^v$ a vectorial facet in $V$, we consider the sets $\QD^m(F^v)=\{\qa\in\QD\mid\qa(F^v)=0\,\}$ and $\QF^m(F^v)=\QF\cap\QD^m(F^v)$ of roots. They are clearly systems of roots: if $F^v=F^v(J)$, then $\QD^m(F^v)=\QD^m(J)$ and $\QF^m(F^v)=\QF^m(J)$.

        \par We define $\A^{ne}_{F^v}$ as the affine space $V$ endowed with the set $\SHM(F^v)=\{M(\qa,\ql)\mid\qa\in\QF(F^v),\ql\in\QL\}$ of walls, the corresponding reflections and $\SHM^i(F^v)$ defined similarly using $\QD(F^v)$. Its points are written $(x,F^v)$ with $x\in V$.
        The essentialization $\A^{e}_{F^v}$ of $\A^{ne}_{F^v}$ is the quotient of $V$ by the vector space $\vect{F^v}$ generated by $F^v$ (with the corresponding walls and reflections); the class of $x\in V$ in $\A^{e}_{F^v}$ is written $[x+F^v]$.

        \par When $F_1^v\in F^{v*}$ \ie $F^v\subset\overline{F_1^v}$, we have $\vect{F^v}\subset\vect{F_1^v}$; so there is a projection $pr_{F_1^v}$ of $\A^{e}_{F^v}$ onto $\A^{e}_{F_1^v}$ :
        $pr_{F_1^v}([x+F^v])=[x+F_1^v]$. We also write $pr_{F_1^v}$ the evident map from $\A^{ne}_{F^v}$ onto $\A^{ne}_{F_1^v}$ or $\A^{e}_{F_1^v}$.

        \par Following \cite{Cn-10b}, we say that $\A^{e}_{F^v}$ (resp. $\A^{ne}_{F^v}$) is the (essential) {\it fa\c{c}ade} (resp. {\it non essential fa\c{c}ade}) of $\A$ in the {\it direction} $F^v$. A fa\c{c}ade is called {\it spherical} (resp. {\it positive} or {\it negative}) if its direction is spherical (resp. positive, or negative). The same things as in \ref{2.4} may be defined in each fa\c{c}ade.

        \par {\bf2)} {\bf Bordered apartments} : Let $\overline{\overline{\A}}$ (resp. $\overline{\A}^e$) be the disjoint union of all $\A^{ne}_{F^v}$ (resp. $\A^{e}_{F^v}$) for $F^v$ a vectorial facet in $V$ and let $\overline{\A}^i$ be the disjoint union of  $\A$ and all $\A^{e}_{F^v}$ for $F^v$ a non trivial vectorial facet in $V$.
        Then $\overline{\overline{\A}}$ (resp. $\overline{\A}^e$, $\overline{\A}^i$) is the {\it strong} (resp. {\it essential}, {\it injective}) {\it bordered apartment} associated to $\A$; its {\it main fa\c{c}ade} is $\A_0=\A$ (resp. $\A^e$, $\A$) of direction the trivial vectorial facet $V_0=F^v_\pm{}(I)$.

 \par       In the following we set $\overline{\A}=\overline{\overline{\A}}$  (resp. $\overline{\A}^e$, $\overline{\A}^i$), $\A^{}_{F^v}=\A^{ne}_{F^v}$ (resp. $\A^{e}_{F^v}$, $\A^{ne}_{F^v}$), \etc

        \par For $x\in\overline\A$, we write $F^v(x)$ the direction of the fa\c{c}ade containing $x$.
        For $\qe=\pm{}$, $\overline\A^\qe$ (resp. $\overline\A_{sph}$) is the union of the fa\c{c}ades of sign $\pm{}$ (resp. the spherical fa\c{c}ades) in $\overline\A$ and $\overline{\A}^{\qe}_{sph}=\overline{\A}^{\qe}_{}\cap\overline{\A}^{}_{sph}$.

        \par To each wall $M(\qa,\ql)$ or half-apartment $D(\qa,\ql)$ is associated a wall $\overline{M}(\qa,\ql)$ or half-apartment $\overline{D}(\qa,\ql)$ of $\overline{\A}$:  $\forall F^v$, $\overline{M}(\qa,\ql)\cap\A^{}_{F^v}$ (resp. $\overline{D}(\qa,\ql)\cap\A^{}_{F^v}$) is the projection of $M(\qa,\ql)$ (resp. $D(\qa,\ql)$) on $\A^{}_{F^v}$ if $\qa(F^v)=0$, the empty set if $\qa(F^v)<0$ and the empty set (resp. $\A^{}_{F^v}$) if $\qa(F^v)>0$.
       With these definitions we may define enclosures $cl(\QO)$  in $\overline{\A}$.

    \par    The essentialization of  $\overline{\A}$ is $\overline{\A}^e$, which is the bordered apartment defined in \cite{Cn-10b}. We shall focus on $\overline{\A}^i$, as $\overline{\A}^e$ is $\overline{\A}^i$ if we choose $V=V^q$.

        \par The set $\overline{\A}^{i\qe}_{sph}$  is the microaffine apartment of sign $\qe$ as in \cite{Ru-06} (in its Satake realization). The corresponding object $\overline{\overline{\A}}^\qe_{sph}$ is closer to the apartments of \cite[2.3]{Ru-06}.

        \par {\bf3)} {\bf Links with sector-face germs and chimney germs} :
        There is a one to one correspondence between the points of $\overline{\A}^e$ and the sector-face germs in $\A$.
        To $\g F=germ_\infty(x+F^v)$ corresponds the class $[x+F^v]$ of $x$ modulo $\vect{F^v}$ in $\A_{F^v}^e$, also written $[\g F]$.
        When $\A=\A^q$ is essential, the points in $\overline{\overline\A}$ correspond bijectively to the sector-faces in $\A$.

        \par By definition $\A_{F^v}$ itself is an affine apartment with walls defined using $\QF^m(F^v)$. The closed-facets in $\A_{F^v}$ correspond bijectively with the chimney germs of $\A$ of direction $F^v$. To $\g R=\g R(F,F^v)$ corresponds the closed facet $[\g R]$ which is the filter made of the subsets in $\A_{F^v}$ containing $\{\,[\g F]\mid\g F\subset\QS\}$ for some subset $\QS$ of $\A$ containing $\g R$.

        \par Actually $\g R$ is splayed (resp. solid, full) if and only if $[\g R]$ is in a spherical fa\c{c}ade (resp. is spherical in its fa\c{c}ade, is a chamber in its fa\c{c}ade).

         \par {\bf4)} {\bf Topology} : On $\overline{\A}^i$ (or $\overline{\A}^e$) one can define a topology inducing the affine topology on each fa\c{c}ade and such that $\overline{\A}$  (or $\overline{\A}^e$) is the closure of $\A_0=\A$ (or $\A^e$) which is open in $\overline{\A}$  (or $\overline{\A}^e$) \cite[11.1.1]{Cn-10b}:

         \par For a non trivial vectorial facet $F^v$, $x\in\A$ and $U$ an open subset of $\A$ containing $x$, we set $\shv(U,F^v)=(U+F^v)\cup\{\,[\g F]\mid\g F\subset U+F^v\,\}$.
         When $x,U$ vary but $F^v$ and $germ_\infty(x+F^v)$ are fixed, we get a fundamental system of neighborhoods of $[x+F^v]$ in $\overline{\A}^i$ (or $\A^e$).

         \par For this topology the closure $\overline\A_{F^v}$ of a fa\c{c}ade $\A_{F^v}$ (with $F^v$ non trivial) is the union of the fa\c{c}ades $\A_{F_1^v}$ for $F_1^v\in F^{v*}$ \ie $F^v\subset\overline{F_1^v}$; we take this for definition of $\overline\A_{F^v}$ when $\overline\A=\overline{\overline\A}$.
          In the classical case, $\overline{\A}^i$ is a compactification of $\A$ called the Satake or polyhedral compactification, see \eg \cite{Cn-08}.

          \par {\bf5)} {\bf Automorphisms} : Any automorphism $\qf$ of $\A$ may be extended to an automorphism $\overline\qf$ of $\overline{\A}$.
          For $\overline{\A}^i$ or $\overline{\A}^e$ the image of $[x+F^v]\in\A_{F^v}$ is $[\qf(x)+\vect{\qf}(F^v)]\in\A_{\vect{\qf}(F^v)}$ and $\overline{\qf}$ is continuous. For $\overline{\overline{\A}}$, ${\overline{\qf}}(x,F^v)=(\qf(x),\vect{\qf}(F^v))$. Automorphisms permute the fa\c{c}ades.

          \par In particular, the action $\qn$ of $N$ on $\A$ may be extended as an action on  $\overline{\A}$ which is also written $\qn$.

\section{Hovels and bordered hovels}\label{s3}

\par The main reference for this section is \cite{Cn-10b}, but we recall all needed definitions.

\subsection{Wanted}\label{3.1}

Let $G$ be a group, $N$ a subgroup and $\qn$ an action of $N$ over some space $A$. We want a space $\SHI$ containing $A$ as a subset and an action of $G$ on $\SHI$ such that $G.A=\SHI$,  $A$ is stable under $N$ and the induced action is $\qn$.

\par Following F. Bruhat and J. Tits \cite[7.4.1]{BtT-72}, a good way to get it, is to define $\SHI$ as a quotient:

\par\noindent$\SHI=G\times A/\sim_\shq$ with $(g,x)\sim_\shq(h,y)\iff\exists n\in N$ such that $y=\qn(n).x$ and $g^{-1}hn\in Q(x)$

  \par\qquad where $\shq=(Q(x))_{x\in A}$ is a family of subgroups of $G$.

  \par\noindent The action of $G$ on $\SHI$ is induced by the left multiplication on $G$, we have a map $i:A\to\SHI$ : for $x\in A$, $i(x)$ is the class of $(1,x)$.

  \par We are interested in the groups $G$ and $N$ as in \ref{2.1} (in particular associated to a valued root datum) and an action $\qn$ as in \ref{2.3}.1 or \ref{2.5}.5. 
  As in \cite{Cn-10}, \cite{Cn-10b} we skip a possible generalization to RGD-systems.
  We shall now precise the conditions on the family $(Q(x))_{x\in A}$, following \cite[11.2.1 and 11.3]{Cn-10b}.

  \subsection{Families of parahoric subgroups}\label{3.2}

  \par {\quad\bf1)} Let $A=\overline\A$ be $\overline{\overline{\A}}$, $\overline{\A}^e$ or $\overline{\A}^i$ and $\qn$ the corresponding action of $N$.
  For a family $\shq=(Q(x))_{x\in A}$, we then write $\overline\SHI=G\times\overline\A/\!\sim_\shq$, it is the {\it bordered hovel} associated to the situation.
  The {\it (bordered) apartments} of $\overline\SHI$ are the sets $g.i(\overline\A)$ for                         $g\in G$.

  \par For a subset or a filter $\QO$ in $\overline\A$ (resp. $\A$), $\qa\in\QF$ and $\Psi\subset\QF$, we define   $\overline D(\qa,\QO)=\overline D(\qa,\sup(-\qa(\QO)))$ (resp. $D(\qa,\QO)=\overline D(\qa,\QO)\cap\A$),  $U_\qa(\QO)=\{u\in U_\qa\mid\QO\subset\overline D(\qa,\qf_\qa(u))\}$ (hence $U_\qa(y)=U_{\qa,-\qa(y)}$), $N(\QO)=\{n\in N\mid n\; \mathrm{fixes}\; \QO\}$, $G(\Psi,\QO)=\langle U_\qa(\QO)\mid\qa\in\Psi\rangle$ and $G(\QO)=G(\QF,\QO)$, written $U_\QO$ in \cite[{\S{}} 3.2]{GR-08}  or \cite[4.6a]{Ru-11}.
  As in these references we write $U_\QO^{++}:=G(\QF^+,\QO)\subset U_\QO^+:=U^+\cap G(\QO)$ and the same things with $-$.
  It often happens that $U_\QO^{++}\not=U_\QO^+$ \cite[4.12.3a]{Ru-11}, see also \ref{5.11}.4 below.

\par  It is clear that $N(\QO)\subset N(F^v)$ normalizes $G(\QO)$ and $U(F^v)$ and that $G(\QO)\subset G(\QF(F^v),\QO)\subset P(F^v)$ normalizes  $U(F^v)$, if $\QO\cap\A_{F^v}\not=\emptyset$. We always have $G(\Psi,\QO)=G(\Psi,cl^\#(\QO))$.
  When $\QO=\emptyset$, we have $G^\emptyset=Z_0.G(\emptyset)$.
  For $\QO\not=\emptyset$, the group $N_\QO^{min}=N\cap G(\QO)$ is normal in $N(\QO)$.
  Its image $W_\QO^{min}$ by $\qn$ is in $W^a$ and generated by the reflections with respect to the (true) walls of $\A$ containing $\QO$.
  This group $W_\QO^{min}$ is isomorphic to its image $W^v_\QO$ in $W^v$ \cite[{\S{}} 3.2]{GR-08}.

  \begin{enonce*}[definition]{\quad2) Definition} A family $\mathcal{Q}=(Q(x))_{x\in\overline{\A}}$ of subgroups of $G$ is a {\it family of parahoric subgroups} if it satisfies the following axioms:

  \par\noindent (For the convenience of the reader, we give to axioms the names in \cite{Cn-10b} and shorter names.)

  \par\label{N16} (P1) (para 0.1) For all $x\in\overline{\A}$, $U(F^v(x))\subset Q(x)\subset P(F^v(x))$

    \par (P2) (para 0.2) For all $x\in\overline{\A}$, $N(x)\subset Q(x)$

    \par (P3) (para 0.3) For all $x\in\overline{\A}$, for all $\qa\in\QF$, for all $\ql\in\R$, if $x\in\overline{D}(\qa,\ql)$, then $U_{\qa,\ql}\subset Q(x)$

    \par (P4) (para 0.4) For all $x\in\overline{\A}$, for all $n\in N$, $nQ(x)n^{-1}=Q(\qn(n).x)$.
  \end{enonce*}

  \par If $\QO$ is a subset of $\overline{\A}$, we define $Q(\QO)=\bigcap_{x\in\QO}\,Q(x)$. If $\QO$ is a filter in $\overline{\A}$, we define $Q(\QO)=\bigcup_{\QO'\in\QO}\,Q(\QO')$. In both cases $Q(\QO)$ is a subgroup of $G$.

    \begin{enonce*}[definition]{\quad3) Easy consequences} a) Axiom (P4) tells that $\sim_\shq$ is an equivalence relation [\lc, 11.3.2].
    By axiom  (P3), $U_{\qa,\ql}$ fixes (pointwise) $\overline{D}(\qa,\ql)$.
    Axiom (P2) tells that the map $i:\overline\A\to\overline\SHI$ is $N-$equivariant, but it is not clearly one to one, \cf \ref{3.3}.2 below.

    \par b) [\lc, 11.3.8] The fixer of $i(x)\in i(\overline\A)$ in $G$  is $G_x=Q(x)$.
    More generally for a subset or filter $\QO$ in $g.i(\overline\A)\subset\overline\SHI$, we define $G_\QO=Q(\QO)$ as the fixer $g.Q(g^{-1}.\QO).g^{-1}$ of $\QO$.

    \par For $x\in\overline\A$ and $g\in G$, if $g.i(x)\in i(\overline\A)$, then there exists $n\in N$ with $g.i(x)=n.i(x)$.
     For a subset or filter $\QO$ in $\overline\A$, the set $G(\QO\subset\overline\A)=\{g\in G\mid G.i(\QO)\subset i(\overline\A)\}$ is equal to $\bigcap_{x\in\QO}\,NQ(x)$ (if $\QO$ is a set) or $\bigcup_{\QO'\in\QO}\,G(\QO'\subset\overline\A)$ (if $\QO$ is a filter).

     \par For all $x\in\overline\SHI$, $Q(x)$ is transitive on the apartments containing $x$.

    \par c) If $F^v$ is a vectorial facet of $\A^v$, axiom (P1) tells that the map $P(F^v)\times\A_{F^v}\to\overline\SHI$ induces a map $G(F^v)\times\A_{F^v}/\sim_{F^v}\to\overline\SHI$, where $\sim_{F^v}$ is the equivalence relation defined using $\shq\,\rule[-1.5mm]{.1mm}{3mm}_{\,\A_{F^v}}$ and $N(F^v)$.
    This map is one to one, as $y=\qn(n).x$ with $x,y\in\A_{F^v}$ and $n\in N$ implies $n\in N(F^v)=N\cap P(F^v)$.

   \par The image of this map is the {\it fa\c{c}ade} $\SHI_{F^v}$ of $\overline\SHI$ in the direction $F^v$.
    In particular the {\it main fa\c{c}ade} of $\overline\SHI$ is the {\it hovel} $\SHI=G\times\A/\sim$ where $\sim$ is  defined using $\shq\,\rule[-1.5mm]{.1mm}{3mm}_{\,\A}$ and $N$.
     Actually each fa\c{c}ade $\SHI_{F^v}$ is an hovel, the main fa\c{c}ade of $\overline\SHI_{F^v}=\bigcup_{F_1^v\in F^{v*}}\,\SHI_{F^v_1}$ associated to $\overline\A_{F^v}$, $\shq\,\rule[-1.5mm]{.1mm}{3mm}_{\,\overline\A_{F^v}}$ and the valuated root datum $(G(F^v),(U_\qa)_{\qa\in\QF^m(F^v)},Z,(\qf_\qa)_{\qa\in\QF^m(F^v)})$.

    \par d) By (P1) and (P3), if $\QO\subset\A_{F^v}$ is non empty, then $G(\QF^m(F^v),\QO)\subset G(\QO)\subset U(F^v)\rtimes G(\QF^m(F^v),\QO)\subset Q(\QO)=U(F^v)\rtimes(M(F^v)\cap Q(\QO))$.

    \par\label{N16} e) If $\shq$ is a family of parahorics and $x\in\overline{\A}$, then $Q(x)\supset P(x):=\langle N(x), G(x),U(F^v(x))\rangle=N(x).G(x).U(F^v(x))$.
    So it is clear that $\mathcal{P}=(P(x))_{x\in\overline{\A}}$ is the {\it minimal family of parahorics}. In the classical (= spherical) case it is the right family; this is the reason for axiom (P6) below. But it is not clear in general that $\shp$ satisfies axiom (P5) below.
    Note that, even for $x\in\A$, $P(x)$ is seldom equal to $P_x$, as defined in \cite[5.14]{Ru-11} or \cite[3.12]{GR-08}.

  \end{enonce*}

    \begin{enonce*}[definition]{\quad4) Definition} A {\it good family of parahorics} is a family $\shq$ of parahorics satisfying moreover:

    \par(P5) (para inj) For all $x\in\overline{\A}$, $N(x)=Q(x)\cap N$

     \par(P6) (para sph) For all $x\in\overline{\A}$,  if $F^v(x)$ is spherical, $Q(x)=P(x)$
     \ie $\shq\,\rule[-1.5mm]{.1mm}{3mm}_{\,\overline\A_{sph}}=\shp\,\rule[-1.5mm]{.1mm}{3mm}_{\,\overline\A_{sph}}$.

\par (P7) (para 2.2)(sph)  If $F^v$ is a spherical facet, $F_1^v\subset\overline{F^v}$ and $x\in\A_{F_1^v}$
\par \qquad\qquad then \qquad $NQ(x)\cap NP(F^v)=NQ(\{x,pr_{F^v}(x)\})$.

  \end{enonce*}

    \begin{enonce*}[definition]{\quad5) Definition} A {\it very good family of parahorics} is a good family $\shq$ of parahorics satisfying moreover:

     \par(P8) (para dec) For all $x\in\overline{\A}$, for all chamber $C^v\in F^v(x)^*$,
 \par\qquad\qquad\qquad\label{N23}    $Q(x)=(Q(x)\cap U(C^v)).(Q(x)\cap U(-C^v)).N(x)$

 \par (P9) (para 2.1+)(sph)  If $F^v$ is a spherical facet, $F_1^v\subset\overline{F^v}$ and $x\in\A_{F_1^v}$
\par \qquad\qquad then \qquad $Q(x)\cap P(F^v)=Q(\overline{x+F^v})$

\par\noindent where $x+F^v=pr_{F_1^v}(x_1+F^v)$ for any $x_1\in\A$ with $pr_{F_1^v}(x_1)=x$ and $\overline{x+F^v}$ is the union of the sets $pr_{F_2^v}(x+\overline{F^v}+\vect{F_2^v})$ for $F_1^v\subset\overline{F_2^v}\subset\overline{F^v}$ (it is the closure of $x+F^v$ when $\overline\A=\overline\A^e$ or $\overline\A^i$).

   \par (P10) If $x<y$ or $y<x$ in $\A_{F^v}$, then $Q(]x,y])\subset Q(x)$ \ie $Q(]x,y])=Q([x,y])$
    \par\noindent where $]x,y]=[x,y]\setminus\{x\}$ is an half-open-segment.

  \end{enonce*}

      \begin{enonce*}[definition]{\quad6) Remarks} a) (P8) is an important tool for calculations. (P7)  and (P9) give links between $\shq$ and $\shq\,\rule[-1.5mm]{.1mm}{3mm}_{\,\overline\A_{sph}}$ which is well known by (P6).

      \par b) By [\lc, 11.9.2] a consequence of (P9) is the following condition:

       \par (P9-) (para 2.1+-)(sph)  If $F^v$ is a vectorial facet and $g\in U(F^v)$, there exists $x\in\A$ such that $g\in Q(\overline{x+C^v})$ for all chamber $C^v\in F^{v*}$.

       \par c) For $x$, $F^v=F^v(x)$ and $C^v$ as in (P8), suppose $x\notin\A$ \ie $F^v$ non trivial.
       Then we have $Q(x)\cap U(C^v)=(Q(x)\cap M(F^v)\cap U(C^v))\ltimes U(F^v)$.
       Now, by the uniqueness in Birkhoff decomposition (\ref{1.3c}.2), $P(F^v)\cap U(-C^v)=M(F^v)\cap U(-C^v)=M(-F^v)\cap U(-C^v)$ which is a "maximal unipotent" subgroup (opposite $M(F^v)\cap U(C^v)$) in $M(F^v)=M(-F^v)$; hence $Q(x)\cap U(-C^v)=Q(x)\cap M(-F^v)\cap U(-C^v)$.
       Now, if $F^v$ is spherical, we can give another explanation: $M(F^v)\cap U(-C^v)=M(F^v)\cap U(C_-^v)$ where $C_-^v$ is the chamber opposite to $C^v$ in $F^{v*}$; so
       $Q(x)\cap U(-C^v)=Q(x)\cap M(F^v)\cap U(C^v_-)$.

       \par d) Except (P9) all axioms impose relations between a single fa\c{c}ade $\A_{F^v}$ and the spherical fa\c{c}ades $\A_{F^v_1}$ for $F_1^v\in F^{v*}$.
       We may fix $F_0^v$  and take $x\in \A_{F^v_0}\cup\overline\A_{sph}$ in the axioms, then we get the same results.
       So, starting with \ref{3.4}, we shall use actually a family $\shq$ of groups $Q(x)$ for $x\in \A_{}\cup\overline\A_{sph}$ with the corresponding axioms.

    \par e) A priori a good family of parahorics has no property of continuity. This is the reason of the (weak) axiom (P10). But without it everything in this section is still true (except when the contrary is explicitly told). This axiom (P10) is satisfied by the minimal family $\shp$.

    \par f) If $\shq$ is a very good family for $(G,(U_\qa)_{\qa\in\QF},Z,(\qf_\qa)_{\qa\in\QF})$ then we define $Q^\emptyset(x)=Q(x)\cap G^\emptyset$.
    We have $Q(x)=Q^\emptyset(x).N(x)$ (by (P8)) and $\shq^\emptyset$ is a very good family of parahorics for $(G^\emptyset,(U_\qa)_{\qa\in\QF},Z^\emptyset,(\qf_\qa)_{\qa\in\QF})$ (defined in \ref{2.3}.3c).
  \label{N24}  The two bordered hovels associated to $\overline\A$ and $(G,\shq)$ or $(G^\emptyset,\shq^\emptyset)$ are canonically isomorphic.

      \end{enonce*}

      \subsection{Bordered hovels associated to good families}\label{3.3}

      \par We explain now some of the abstract results of \cite{Cn-10b} (or \cite{Cn-10}).
     So let $\shq$ be a good family of parahorics (if it exists) and $\overline\SHI$ be the associated bordered hovel.

    \par {\bf1)} By Bruhat-Tits theory and (P6) $\shq$ is well known on the spherical fa\c{c}ades [\lc, 11.2.3]:
    \eg the results of (P8) and (P9) are true when $F^v(x)$ is spherical, for $\QO$ in a spherical fa\c{c}ade $\A_{F^v}$, $Q(\QO)=U(F^v)\rtimes(N(\QO).G(\QF^m(F^v),\QO))=N(\QO).Q(cl^\#(\QO))$ and
 $G(\QO\subset\overline\A)=N.Q(\QO)$.
  \par  Actually for $F^v$ spherical $\SHI_{F^v}$ is the Bruhat-Tits building of the classical valuated root datum $(G(F^v),(U_\qa)_{\qa\in\QF^m(F^v)},Z,(\qf_\qa)_{\qa\in\QF^m(F^v)})$ (with the facets associated to $cl$).

    \par {\bf2)} The minimal family $\shp$ satisfies also (P5) and (P6). Axiom (P5) tells us that $i:\overline\A\to\overline\SHI$ is one to one [\lc, 11.3.4]; we identify $\overline\A$ and $i(\overline\A)$. The stabilizer of $\overline\A$ in $G$ is $N$ [\lc, 11.3.5].

   \par {\bf3)  Iwasawa decomposition} [\lc, 11.4.2]: For a chamber $C^v$ or a facet $F^v$ in $\A^v$ and a facet $F\subset\overline\A$, we have $G=U(C^v).N.G(F)=U(F^v).N.Q(cl^\#(F))=U(F^v).N.Q(F)$.
   
 \par Actually we may replace here $F$ by any narrow filter (see \cite[3.2, 3.6]{GR-08}), \eg any segment germ $[x,y)\subset\A$, preordered or not (\cf \ref{3.6}.1).

   \par {\bf4) Bruhat-Birkhoff-Iwasawa decomposition} [\lc, 11.5] : Let $F_1\subset\A_{F_1^v}$ and $F_2\subset\A_{F_2^v}$ be two facets with $F_1^v$ or $F_2^v$ spherical.  \label{N6} Then
   $G=U(F_1^v).G(\QF^m(F_1^v),F_1).N.G(\QF^m(F_2^v),F_2).U(F_2^v)=Q(cl^\#(F_1)).N.Q(cl^\#(F_2))=Q(F_1).N.Q(F_2)$.
    If $F_1$ and $F_2$ are facets in $\overline\SHI$ and $F_1$ or $F_2$ is in a spherical fa\c{c}ade, then there is an apartment of $\overline\SHI$ containing $F_1$ and $F_2$ (even $cl^\#(F_1)$ and $cl^\#(F_2)$).

    \par {\bf5) Projection} : Let $F^v$ be a spherical facet and $F_1^v\subset\overline{F^v}$.
    Then, by (P7), the projection $pr_{F^v}$ of $\A_{F_1^v}$ onto $\A_{F^v}$ extends to a well defined map $pr_{F^v} : \SHI_{F_1^v}\to\SHI_{F^v}$ between the corresponding fa\c{c}ades.
    For each $g\in G$, $g.pr_{F^v}(x)=pr_{gF^v}(gx)$ [\lc, prop. 11.7.3].

    \par {\bf6)} If $\overline\A=\overline\A^e$, then (P8) is satisfied by any good family of parahorics [\lc, 11.7.5].

    \par {\bf7)} Let $\shq$ be a good family of parahorics for $\overline\A=\overline\A^e$, satisfying moreover (P9) or (P9-). Suppose $\QO\subset\overline\A$ is in $\overline\A^\qe$ and intersects non trivially $\overline\A^\qe_{sph}$ or intersects non trivially $\overline\A^+_{sph}$ and $\overline\A^-_{sph}$.
    Then $G(\QO\subset\overline\A)=N.Q(\QO)$, hence $Q(\QO)$ is transitive on the apartments containing $\QO$.
     If $\QO$  intersects non trivially $\overline\A^+_{sph}$ and $\overline\A^-_{sph}$, then $Q(\QO)=N(\QO).Q(cl^\#(\QO))$ : $cl^\#(\QO)$(and also $cl(\QO)$,...) is well defined in $\overline\SHI$ independently of the apartment containing $\QO$. [\lc, section 11.9.2]

     \par {\bf8)} One can find in \LC many other implications between the various axioms. Actually Charignon introduces also useful notions of functoriality \ie the possibility of embedding the valuated root datum in greater ones, with arbitrarily large subsets $\QL_\qa$ of $\R$ and various good compatibilities. We shall not explain this, as it is more natural in the framework of split Kac-Moody groups over valuated fields on which we shall  concentrate in the next section.

   \begin{prop}\label{3.4} Let $\shq$ be a  good family of parahorics satisfying (P9) and $\QO$ be a non empty subset or filter in $\A$.

  \par a)  Let $F^v\subset\overline{C^v}\subset\A^v$ be a spherical vectorial facet in the closure of a chamber and $\qa_1,\cdots,\qa_n\in\QF$ be the non divisible roots such that $\qa(C^v)>0$ and $\qa(F^v)=0$.
 Then:
   \par\qquad $Q(\QO)\cap P(C^v)=(Q(\QO)\cap U(C^v)) \rtimes Z_0$
  \par\qquad $Q(\QO)\cap U(C^v)=(Q(\QO)\cap U(F^v)) \rtimes (Q(\QO)\cap U(C^v)\cap M(F^v))$
   \par\noindent and $Q(\QO)\cap U(C^v)\cap M(F^v)=U_{\qa_1}(\QO).\cdots.U_{\qa_n}(\QO)$ with uniqueness of the decomposition.

   \par b) Let $C^v\subset\A^v$ be a chamber. Then the set $Q^{dec}(\QO,C^v)=(Q(\QO)\cap U(C^v)).(Q(\QO)\cap U(-C^v)).N(\QO)$ depends only of the sign of the chamber $C^v$.
   \end{prop}

      \begin{enonce*}[definition]{N.B} 1) So we define $Q^{dec}(\QO,\qe)=Q^{dec}(\QO,C^v)$ if $C^v$ is of sign $\qe$.

      \par 2) By (P9) $Q(\QO)\cap U(F^v)\subset Q(\QO+\overline{C^v})$ for all $C^v\in F^{v*}$.
\end{enonce*}

           \begin{proof} a) 
           We have $P(F^v)=U(C^v)\rtimes Z$, $U(C^v)=U(F^v)\rtimes(U(C^v)\cap M(F^v))$ and, by Bruhat-Tits theory (\ref{3.3}.1), $U(C^v)\cap M(F^v)=U_{\qa_1}.\cdots.U_{\qa_n}$ (unique).
            Using these uniqueness results, we have just to prove a) for $\QO=\{x\}$ and $x\in\A$.
            We write $x'=pr_{F^v}(x)$.

      \par By (P9) and \ref{3.3}.1 $Q(x)\cap P(F^v)=Q(\overline{x+F^v})\subset Q(x')=U(F^v)\rtimes(Q(x')\cap M(F^v))$ and $Q(x')\cap M(F^v)=U_{\qa_1}(x).\cdots.U_{\qa_n}(x).U_{-\qa_1}(x).\cdots.U_{-\qa_n}(x).N(x')$ \cite[7.1.8]{BtT-72}.
      So $Q(x')\cap M(F^v)\cap U(C^v)=U_{\qa_1}(x).\cdots.U_{\qa_n}(x)$ and $Q(x')\cap M(F^v)\cap P(C^v)=U_{\qa_1}(x).\cdots.U_{\qa_n}(x).Z_0$ (by uniqueness in the Birkhoff decomposition \ref{1.3c}.2.). And, as each $U_{\qa_i}(x)$ is in $Q(x)$, we get what we wanted.

      \par b) Any two chambers of sign $\qe$ are connected by a gallery of chambers of sign $\qe$.
      So one has only to show that $Q^{dec}(\QO,C^v)=Q^{dec}(\QO,r_\qa(C^v))$ when $\qa\in\QF$ is simple with respect to $\QF^+(C^v)$.
      We consider $F^v=\overline{C^v}\cap$Ker$\qa$ and apply a). But $U_{\qa}(\QO).U_{-\qa}(\QO).N(\QO+$Ker$\qa)=U_{-\qa}(\QO).U_{\qa}(\QO).N(\QO+$Ker$\qa)$ by \cite[6.4.7]{BtT-72}; so the same proof as in \cite[3.4a]{GR-08} applies.
           \end{proof}

  \subsection{Good fixers}\label{3.5}

     \par{\quad\bf1)} We consider now a very good family of parahorics $\shq=(Q(x))_{x\in \A_{}\cup\overline\A_{sph}}$ and we want to define the same notions as in  \cite[def. 4.1]{GR-08}, using the axioms and proposition \ref{3.4}; this is suggested in the beginning of \cite[sec. 5]{Ru-11}.

  \begin{enonce*}[definition]{\quad2) Definition} Consider the following conditions for a subset or filter $\QO$ in $\A$:

  \par\qquad (GF$\qe$)\quad$Q(\QO)=Q^{dec}(\QO,\qe)$\qquad\quad for $\qe=+$ or $-$
   \par\qquad (TF)\quad$G(\QO\subset\overline\A)=NQ(\QO)$ \qquad (where $G(\QO\subset\overline\A)$ is defined in \ref{3.2}.3b)

   \par\noindent We say that $\QO$ has a {\it good fixer} if it satisfies these three conditions.
        \par\noindent We say that $\QO$ has an {\it half-good fixer} if it satisfies (TF) and (GF+) or (GF-).
      \par\noindent We say that $\QO$ has a {\it transitive fixer} if it satisfies (TF).

 \end{enonce*}

       \begin{enonce*}[definition]{\quad3) Consequences} We get the following results by mimicking the proofs in \cite[sec. 4.1]{GR-08}.
       The ingredients are proposition \ref{3.4} and the facts that $Q(\QO)\cap U(\pm{}C^v)=Q(\QO\pm{}\overline{C^v})\cap U(\pm{}C^v)$, $\bigcap_\QO\,Q(\QO)\cap U(\pm{}C^v)=Q(\cup_\QO\,\QO)\cap U(\pm{}C^v)$ for a family $\QO$ of filters, \etc

       \par a) By (P8) a point has a good fixer. 
       The group $N$ permutes the filters with good fixers and the corresponding fixers.

      \par  If $\QO$ has a transitive fixer, then $Q(\QO)$ acts transitively on the apartments containing $\QO$.
      Hence the "shape" of $\QO$ doesn't depend of the apartment containing it.
      As a consequence of the many examples below of filters with (half) good fixers, we may define in $\SHI$ (independently of the apartment containing it) what is a preordered segment, preordered segment-germ, generic ray, closed (local) facet, spherical sector face, solid chimney \etc

        \par b) In the classical case every filter has a good fixer.

        \par c) Let $\SHF$ be a family of filters with good (or half-good) fixers such that the family $\QO$ of the sets belonging to one of these filters is a filter. Then $\QO$ has a good (or half-good) fixer $Q(\QO)=\bigcup_{F\in\SHF}\,Q(F)$.

        \par d) Suppose the filter $\QO$ is the union of an increasing sequence $(F_i)_{i\in\N}$ of filters with good (or half-good) fixers and that, for some $i$, the support of $F_i$ has a finite fixer in $\qn(N)$, then $\QO$ has a good (or half-good) fixer $Q(\QO)=\bigcap_{i\in\N}\,Q(F_i)$.

     \par e) Let $\QO$ and $\QO'$ be two filters in $\A$ and $C_1^v,\cdots,C_n^v$ be positive vectorial chambers.
     If $\QO'$ satisfies (GF+) and (TF) and $\QO\subset\bigcup_{i=1}^n\,(\QO'+\overline{C_i^v})$, then $\QO\cup\QO'$ satisfies (GF+) and (TF) with $Q(\QO\cup\QO')=Q(\QO)\cap Q(\QO')$.
      If moreover $\QO$ (resp. $\QO'$) satisfies (GF-) and $\QO'\subset\bigcup_{i=1}^n\,(\QO-\overline{C_i^v})$ (resp. $\QO\subset\bigcup_{i=1}^n\,(\QO'-\overline{C_i^v})$), then $\QO\cup\QO'$ has a good fixer.

           \end{enonce*}

            \begin{enonce*}[definition]{\quad4) Remarks} a) Let $\QO$ in $\A$ be a filter with good (or half-good) fixer and $F^v$ be a spherical vectorial facet.
            We write $\Theta=\bigcup_{C^v\in F^{v*}}\,\overline{C^v}$ and
   $\QO'=(\QO+\Theta)\cap(\QO-\Theta)\cap(\bigcap_{\qa\in\QF^m(F^v)}\, D(\qa,\QO))$ (which is in $cl^{\QF}_\R(\QO)$), then, by \ref{3.4}a and \ref{3.5}.3e, any $\QO''$ with $\QO\subset\QO''\subset\QO'$ has a good (or half-good) fixer; moreover $Q(\QO)=Q(\QO'')N(\QO)$.
   In particular any apartment $A$ of $\SHI$ containing $\QO$ contains $\QO'$ and is conjugated to $\A$ by $Q(\QO')$.

   \par b) By \ref{3.2}.6f, for every filter $\QO$, we have $Q^\emptyset(\QO)=Q(\QO)\cap G^\emptyset$, $Q(\QO)\cap U(C^v)=Q^\emptyset(\QO)\cap U(C^v)$, and $N^\emptyset(\QO)=N(\QO)\cap G^\emptyset$.
   Hence if $\QO$ has a (half) good fixer for $\shq$, $N.Q(\QO)=N.Q^\emptyset(\QO)$, $N.Q(\QO)\cap G^\emptyset=N^\emptyset.Q^\emptyset(\QO)$ and $\QO$ has a (half) good fixer for $\shq^\emptyset$.

 \end{enonce*}

   \subsection{Examples of filters with good fixers}\label{3.6}

     \par {\quad\bf1)} If $x\leq{}y$ or $y\leq{}x$ in $\A$ ({\it preordered} situation), then $\{x,y\}$ and the segment $[x,y]$ have good fixers $Q(\{x,y\})=Q([x,y])$
 (apply \ref{3.5}.3e); in particular any apartment $A$ of $\SHI$ containing $\{x,y\}$ contains $[x,y]$ and is conjugated to $\A$ by $Q([x,y])$.
 If moreover $x\not=y$ the segment germ $[x,y)=germ_x([x,y])$ has a good fixer (      \ref{3.5}.3c).

 \par {\bf2)} If $x,y\in\A$ and $\xi=y-x\not=0$ is in a spherical vectorial facet $F^v$ ({\it generic} situation: $x\stackrel{o}{<} y$ or $y\stackrel{o}{<} x$),  then the half-open segment $]x,y]=[x,y]\setminus\{x\}$, the line $(x,y)$ and the ray $\qd=x+[0,+\infty[.\xi$ of origin $x$ containing $y$ (or the open ray $\qd^\circ=\qd\setminus\{x\}$) have good fixers (\ref{3.5}.3d). Using now \ref{3.5}.3c the germs $]x,y)=germ_x(]x,y])$ and $germ_\infty(\qd)$ have good fixers.

 \par {\bf3)} A closed local facet $\overline{F^l(x+F^v)}$ has a good fixer:
 choose $\xi\in F^v$ and $\ql>0$ then the intersection $\QO_{\ql,\xi}$ of $(x+\overline{F^v})\cap(x+\ql\xi-\overline{F^v})$ with a ball of radius $\Vert\ql\xi\Vert$ and center $x$ (for any norm) has a good fixer (\ref{3.5}.3e with $\QO'=[x,x+\ql\xi]$) and $\overline{F^l(x+F^v)}$ is as described in (\ref{3.5}.3c) using the family $\QO_{\ql,\xi}$ (when $\ql$ varies).

 \par If the local facet is spherical, then it has a good fixer.
 We just have to use above $(x+\qe\xi+\overline{F^v})\cap(x+\ql\xi-\overline{F^v})$ for $0<\qe<\ql$ and \ref{3.5}.3c,d,e.

  \par {\bf4)} By similar arguments we see that a spherical sector face or its closure or its germ has a good fixer.
   The apartment $\A$ has a good fixer $Q(\A)=Z_0$, so the stabilizer of $\A$ is $N$.
   An half-apartment $D(\qa,k)$ has a good fixer $Z_0.U_{\qa,k}$ \cf \cite[5.7.7]{Ru-11}.

  \par {\bf5)} We suppose now in this paragraph that the family $\shq$ satisfies axiom (P10). If $y<x$ or $x<y$ in $\A$, then the half-open segment $]x,y]$ (resp. the open-segment-germ $]x,y)$) has a good fixer $Q(]x,y])=Q([x,y])$ (resp. $Q(]x,y))=Q([x,y))\,$), even if $F^v(y-x)$ is not spherical.
  By arguments as in 3) above (using $x+F^v=]x,x+\ql\xi]+\overline{F^v}$ instead of $x+\overline{F^v}$) we deduce that any local facet $F^l=F^l(x,F^v)$ has a good fixer and $Q(F^l)=Q(\overline{F^l})$.

        \begin{prop}\label{3.7} Let $\shq$ be a very good family of parahorics, $\qx\not=0$ a vector in a spherical vectorial facet $F^v$ and $x\in\A$.
        We consider the ray $\qd=x+[0,+\infty[\qx$, then $Q(\qd)\subset Q(germ_\infty(\qd))\subset P(F^v)$.
           \end{prop}

           \begin{enonce*}[definition]{N.B} 1) This is a kind of reciprocity for axiom (P9).
           We have $Q(x)\cap P(F^v)=Q(\overline{x+F^v})=Q(x+\overline{F^v})$ with $x+\overline{F^v}$ in $\A$.
       \par 2)      We see thus directly that $Q(\A)$ fixes $\overline\A_{sph}$.
\end{enonce*}
           \begin{proof} It is sufficient to prove that $Q(\qd)\subset  P(F^v)$.
           Let $C^v$ be a chamber in $F^{v*}$, then by \ref{3.6}.2, $Q(\qd)=(Q(\qd)\cap U(C^v)).(Q(\qd)\cap U(-C^v)).N(\qd)$.
   As $N(\qd)$ and $U(C^v)$ are in $P(F^v)$, we have only to prove $Q(\qd)\cap U(-C^v)\subset P(F^v)$. By (P9) and \ref{3.3}.1, for $\ql\geq{}0$, $Q(x+\ql\xi)\cap U(-C^v)\subset Q(x+\ql\xi)\cap P(-F^v)\subset Q(pr_{-F^v}(x))=N(pr_{-F^v}(x)).G(\QF^m(F^v),pr_{-F^v}(x)).U(-F^v)$.
    So $Q(\qd)\cap U(-C^v)=(G(\QF^m(F^v),pr_{-F^v}(x))\cap U(-C^v)).(U(-F^v)\cap Q(\qd))$ as $G(\QF^m(F^v),pr_{-F^v}(x))$ fixes pointwise $x+\R\xi\supset\qd$.

  \par  Now $U(-F^v)\cap Q(\qd)=\bigcap_{C^v\in F^{v*},\ql\geq{}0}\,U(-C^v)\cap Q(x+\ql\xi)\subset Q(\bigcup_{C^v\in F^{v*},\ql\geq{}0}\,(x+\ql\xi-\overline{C^v}))$.
     But this last union is actually $\A$, so $U(-F^v)\cap Q(\qd)=U(-F^v)\cap Q(\A)=U(-F^v)\cap Z_0=\{1\}$ and $U(-C^v)\cap Q(\qd)\subset G(\QF^m(F^v),pr_{-F^v}(x))\subset P(F^v)$.
           \end{proof}

           \begin{coro}\label{3.8} Let $F\subset \A_{F^v}$ be a facet in a fa\c{c}ade and $\g R\subset\A$ be the corresponding chimney germ (\cf \ref{2.5}.3). \label{N14} Then $U(F^v).G(\QF^m(F^v),F).N(\g R)\subset Q(F)\cap Q(\g R)$.
    If $F^v$ is spherical $Q(F)=U(F^v).G(\QF^m(F^v),F).N(F)\supset Q(\g R)=U(F^v).G(\QF^m(F^v),F).N(\g R)$.
  \end{coro}

  \begin{NB} We sometimes say that $Q(\g R)$ is the {\it strong fixer} of $F$.
  \end{NB}

  \begin{proof} For $x\in\A$, it is clear that $\g R$ is in the union of all $\overline{x+C^v}$ for $C^v\in F^{v*}$. So the first result is due to (P9-).
   For $F^v$ spherical $Q(F)$ is given in \ref{3.3}.1. By proposition \ref{3.7} $Q(\g R)\subset P(F^v)$ and $Q(\g R)\subset Q(F)$ by (P9), hence the result.
  \end{proof}

\subsection{Properties specific to $cl_\R$}\label{3.9}

We are interested here in the cases $cl=cl_\R$ \ie $cl=cl_\R^\QF$, $cl_\R^\#$, $cl_\R^{\QD}$ or $cl_\R^{\QD^{ti}}$.

    \par{\bf1)} For a filter $\QO$ in $\A$, $\QO\subset\QO+V_0\subset(\QO+\overline{C^v})\cap(\QO-\overline{C^v})$ for all vectorial chamber $C^v$.
    So, by \ref{3.5}.4, if $\QO$ has a good or half-good fixer, it is also true for any $\QO'$ with
    $\QO\subset\QO'\subset\QO+V_0$. Moreover $Q(\QO)=Q(\QO')$: $N(\QO)=N(\QO+V_0)$ as $W^v$ fixes $V_0$.

    \par{\bf2)} For a local facet $F^l$, we saw that 
    $\overline F^l=\overline F^{\QD^{ti}}_\R\subset \overline F^{\QD}_\R\subset \overline F^{\QF}_\R=\overline F^{\#}_\R=\overline F^l+V_0$,
    hence the closed $cl_\R-$facet associated to $F^l$  has a good fixer by 1) above.

    \par We saw also that the $cl_\R-$facet associated to $F^l$ is between $F^l$ and $F^l+V_0$. If the family $\shq$ satisfies (P10), then this $cl_\R-$facet has a good fixer (\ref{3.6}.5 and 1) above);
     by \ref{3.5}.4 any apartment containing $F^l$ contains $cl_\R(F^l)$ and is conjugated to $\A$ by $Q(F^l)=Q(\overline{F^l})=Q(\overline{F^l}+V_0)$.

    \par{\bf3)} Let $\QO$ be a point, preordered segment,  preordered segment-germ, generic ray,  generic ray-germ or generic  line (resp.  preordered half-open segment, preordered open-segment-germ or  generic open ray if $\shq$ satisfies (P10)) as in \ref{3.6}  and let $\QO\subset\QO''\subset cl_\R(\QO)$.
     Then $Q(\QO)=Q(\QO'')$ by \ref{3.5}.4, as $cl_\R(\QO)\subset (\QO+\overline{F^v})\cap(\QO-\overline{F^v})$ (resp.  $cl_\R(\QO)\subset (\overline\QO+\overline{F^v})\cap(\overline\QO-\overline{F^v})$) for some facet $F^v$ pointwise fixed by $\qn^v(N(\QO))$.
     Hence any apartment containing $\QO$ contains $\QO''$ and is conjugated to $\A$ by $Q(\QO'')=Q(\QO)$.

     \par We may choose $cl_\R=cl_\R^\QF$.
     So, for $\QO$ a  preordered segment-germ, generic ray or generic ray-germ,
      we may choose above $\QO''$ equal to its $cl_\R^\QF-$enclosure \ie the corresponding closed-local-facet, spherical sector-face-closure or spherical sector-face-germ.
      If $\shq$ satisfies (P10) and $\QO$ is a preordered open segment-germ  (resp. a generic open ray) the same result is true with $\QO''$ the corresponding local facet (resp. corresponding spherical sector-face).

    \par{\bf4)} Let $\overline{F^l}=\overline{F}^l(x,F^v)$ be a closed local facet in $\A$ and $F^v_1$ a vectorial facet.
    Then $\g r=\overline{F^l}+\overline{F^v_1}$ is \label{N10} closed convex \ie $cl^{\QD^{ti}}_\R-$enclosed; hence it is the $cl^{\QD^{ti}}_\R-$chimney $ \g r_\R^{\QD^{ti}}(F^l,F_1^v)$; note that this is not always true for $cl_\R^\QF$, $cl_\R^\#$ or  $cl_\R^{\QD}$. \label{N7}

    \par Suppose $\g r$ solid \ie the fixer in $\qn(N)$ of its support finite. \label{N25} Then $\g r$ and its germ $\g R$ have good fixers: we apply \ref{3.5}.3e to $\overline{F^l}$ and $\overline{F^l}+\ql\xi$ (with $\ql>0$, $\xi \in F_1^v$), then \ref{3.5}.4 and  \ref{3.5}.3d to see that $\g r$ has a good fixer; now the result for $\g R$ is a consequence of  \ref{3.5}.3c.

       \begin{enonce*}[definition]{\quad5) Remark} Suppose $F^v$ and $F^v_1$ as above in 4) and of the same sign.
        Then $\overline{F^v}+\overline{F^v_1}$ meets a vectorial facet $F^v_2$ with $F^v\subset\overline{F^v_2}$ and $F^v_2\cap\langle F^v,F^v_1\rangle$ open in the vector space $\langle F^v,F^v_1\rangle$ ($F^v_2$ is the projection of $F^v_1$ in $F^{v*}$).
        By \ref{3.9}.3 any apartment containing $\overline{F}^l(x,F^v)$ and $x+F^v_1$ (or $F^l(x,F^v_1)$) contains $\overline{F}^l(x,F^v_2)$.
        Suppose $F_2^v$ spherical (\eg if $\g r$ is solid) then, by using a few more times the same argument, we see that any apartment containing $\g r$ contains the $cl^\QF_\R-$enclosure $\QO$ of $\overline{F}^l(x,F^v)$ and $F^l(x,F^v_1)$ and also $\QO+\overline{F^v_1}$ which is the $cl^\QF_\R-$chimney $\g r^\QF_\R(F^l,F^v_1)$.
        So one could use this $cl^\QF_\R-$chimney. But unfortunately it is not clear that $\QO$ or $\QO+\overline{F^v_1}$ has a good fixer. Moreover the following proposition seems difficult to prove for $cl^\QF_\R$. So we shall concentrate on $cl^{\QD^{ti}}_\R$.
   \end{enonce*}

       \begin{enonce*}[plain]{\quad6) Proposition}  Let $\g R_1$ be the germ of a splayed $cl^{\QD^{ti}}_\R-$chimney $\g r_1=\overline{F^l_1}+\overline{F^v_3}$ and $\g R_2$ be either a closed local facet $\overline{F^l_2}$ or the germ of a solid $cl^{\QD^{ti}}_\R-$chimney $\g r_2=\overline{F^l_2}+\overline{F^v_4}$.
        Then $\QO=\g R_1\cup\g R_2$ has a half-good fixer and $Q(\QO)=Q(cl^{\QD^{ti}}_\R(\QO)).N(\QO)$.
         In particular any apartment of $\SHI$ containing $\QO$ also contains $cl^{\QD^{ti}}_\R(\QO)$ and is conjugated to $\A$ by $Q(cl^{\QD^{ti}}_\R(\QO))$.
 \end{enonce*}

 \begin{enonce*}[definition]{N.B}  Actually if $\g R_2$ is the germ of a splayed $cl^{\QD^{ti}}_\R-$chimney (\ie $F_4^v$ is spherical), then $\QO$ has a good fixer \cite[6.10]{Ru-10}.

  \end{enonce*}

 \begin{proof} We may replace $\QO$ by $\QO=\g r_1\cup\g r_2$ with $\g r_1$ and $\g r_2$ sufficiently small.
 Consider $\Theta=\bigcup_{C^v\in F_3^{v*}}\,\overline{C^v}$; by shortening $\g r_1$ we may assume $\g r_1\subset\g r_2+\Theta$. So, by \ref{3.5}.3e and \ref{3.6}.3 or 4) above, $\QO$ has a half good fixer.
 \label{N10} We use \ref{3.5}.4 with $\QO$ and $F^v_3$: as $\g r_1-\Theta=\A$, $\QO'$ is actually equal to
  $\QO'=(\g r_2+\Theta)\cap(\bigcap_{\qa\in\QF^m(F^v_3)}\, D(\qa,\QO))$ which is convex and closed.
  So $\QO\subset \QO''=cl^{\QD^{ti}}_\R(\QO)=\overline{conv}(\QO)\subset\QO'$ and $Q(\QO)=Q(\QO'').N(\QO)$.
 \end{proof}

 \subsection{(Generalized) affine hovels}\label{3.10}
 
 \par We consider an affine apartment $\A$, as defined abstractly in \ref{2.3}.2 a,b and an enclosure map $cl$ as in \ref{2.4}.1 (using \ref{2.3}.4).
 For the following definitions see below in 1), 2) the restrictions on $cl$, or the associated variants.
 See also \ref{4.15} for some variants.

   \begin{enonce*}[definition]{Definitions} An {\it affine hovel} of type $(\A,cl)$ is a set $\shi$ endowed with a covering $\sha$ by subsets called apartments such that:

   \par (MA1) Every $A\in\sha$ is an apartment of type $\A$ (\cf \ref{2.3}.2.f).

   \par (MA2) \label{N9} If $F$ is a point, a preordered open-segment-germ, a generic ray or a solid chimney in an apartment $A$ and if $A'$ is another apartment containing $F$, then $A\cap A'$ contains the enclosure $cl(F)$ of $F$ in $A$ and there exists a Weyl isomorphism from $A$ to $A'$ fixing (pointwise) this enclosure.

   \par (MA3) If $\g R$ is a splayed chimney-germ, if $F$ is a facet or a solid chimney-germ, then $\g R$ and $F$ are always contained in a same apartment.

   \par (MA4) If two apartments $A$, $A'$ contain $\g R$ and $F$ as in (MA3), then their intersection $A\cap A'$ contains the   enclosure $cl(\g R\cup F)$ of $\g R\cup F$ in $A$ and there exists a Weyl isomorphism from $A$ to $A'$ fixing (pointwise) this enclosure.
   \medskip
   \par This affine hovel is told {\it ordered} if it satisfies moreover:

   \par (MAO) Let $x$, $y$ be two points in $\shi$ and $A$, $A'$ be two apartments containing them; if $x\leq{}y$ in $A$, then the segments $[x,y]_{A}$ and $[x,y]_{A'}$ defined by $x$, $y$ in $A$ and $A'$ are equal.
   \medskip
   \par An automorphism of the hovel $\shi$ is a bijection $\qf:\shi\to\shi$ such that, for every apartment $A$, $\qf(A)$ is an apartment and $\qf\rest A$ an isomorphism.
   We say that $\qf$ is {\it positive}, {\it vectorial-type-preserving}, {\it vectorially Weyl} or a {\it Weyl automorphism}, if  $\qf\rest A$ is positive, vectorial-type-preserving, vectorially Weyl or a Weyl isomorphism, for some $A\in\sha$. 
   This is then true for any $A\in\sha$: for any two apartments $A_1$, $A_2$, there is a third apartment $A$ such that $A_i\cap A$ contains a non empty open convex subset (\eg a splayed chimney) and we may use \ref{2.3}.2.g.

   \par We say that a group $G$ acting on $\shi$ acts {\it strongly transitively} if it acts by automorphisms of $\shi$
   and moreover the Weyl isomorphisms between apartments involved in the axioms (MA2) or (MA4) may be chosen induced by elements of $G$.
(In the classical case of thick discrete affine buildings and  groups of Weyl automorphisms, this is equivalent to the known definition, \cf \ref{3.13}.1 below and \eg \cite[prop. 6.6]{AB-08}.)

\par So $G$ acts strongly transitively if, and only if, the subgroup $G^w$ of Weyl automorphisms acts strongly transitively.

   \end{enonce*}

 \par{\bf Variations due to the enclosure map.} Unfortunately the definition of affine hovel given in \cite{Ru-10} is still not general enough, it is too restrictive for $cl$. 
 We explain now our more general definition, for a more general enclosure map.
  \label{N8}

 \par{\bf1)} The enclosure map considered in \LC is $cl^\QD_{ma}$ or (after changing $\QD$ or $\QL=(\QL_\qa)_{\qa\in\QF}$) $cl^\QF$, $cl^\QF_\R$, $cl^\QD_\R$.
 As suggested in [\lc 1.6] the enclosure map $cl^\QD$ is often not so different from $cl^\QD_{ma}$. The results of \LC are true for $cl^\QD$ without changing anything.
  We may also enlarge as we want the family $\QL$ to a family $\QL'$.

  \par{\bf2)} In \LC (except in section 1) the root system $\QD$ is asked to be tamely imaginary.
  This excludes in particular the totally imaginary case $\QD^{ti}$.

  \par When $\QD$ is not tamely imaginary, the axioms of affine hovels of type $(\A,cl^\QD_{\QL'})$) have to be modified as follows:

  \par We must add to the list of the filters involved in (MA2) the local facets and the spherical sector faces. Moreover in (MA3) and (MA4) we must add the possibilities that $F$ is a point or a preordered segment germ and that $\g R$ or $F$ is a generic ray germ.

  \par Then all results of \LC are true up to section 4 (except the last sentence of [\lc 4.8.2]).
  In section 5 (specially 5.2 N.B.) we must add (MA2) for $F$ a segment germ and $cl^\QF_\R$ \ie :

 \par\qquad For $]x,y)\subset F^l(x,F^v)$, any apartment containing $[x,y)$ contains $\overline F^l(x,F^v)$
 \par \qquad\qquad (We can restrict to the case where $F^v$ is a chamber.)

 \par{\bf Generalizations.}

   \par{\bf3)} A generalization is necessary when we drop axiom (P10). 
   We shall say that $\SHI$ is an affine (ordered) {\it closed-restricted-hovel} of type $(\A,cl)$ if it satisfies the above axioms modified as follows:

   \par In the list of axiom (MA2) or in (MA3), (MA4), we replace preordered open-segment-germ by preordered segment-germ, facet by closed facet and (in the case of 2) above) local facet by closed local facet, spherical sector-face by spherical sector-face closure.
   Then all results in \LC are true if we make the same replacements.
   
   \par{\bf4)} 
   We shall say that $\SHI$ is an affine (ordered) {\it generic-restricted-hovel} of type $(\A,cl)$ if it satisfies the above axioms modified as follows:

   \par In axioms (MA2), (MA3), (MA4), (MAO) (eventually modified as in 2) above), we replace everywhere the words preordered, solid, full by generic, splayed, full and splayed (respectively), the preorder $\leq$ by $\stackrel{o}{<}$ and suppose all facets spherical.
   Then all results in \LC are true if we make the same replacements.


\begin{theo}\label{3.11} Let $\shq$ be a very good family of parahorics in $G$.

\par 1) Then $\SHI$ with its family of apartments is an ordered affine hovel of type $(\A,cl_\R^{\QD^{ti}})$. The group $G$ acts strongly transitively and by vectorially Weyl automorphisms on $\SHI$.

 \par 2) The twin buildings $\SHI^{\pm{}\infty}$ constructed at infinity of $\SHI$ in \cite[sec. 3]{Ru-10} are $G-$equivariantly isomorphic to the combinatorial twin buildings $\SHI^{vc}_\pm{}$ of \ref{1.3c}.3 (restricted to their spherical facets).
  This isomorphism associates to each spherical sector-face-direction $\g F^\infty$ a spherical vectorial facet $F^v\in\SHI^{vc}_\pm{}$.

  \par 3) If $F^v$ spherical corresponds to $\g F^\infty$, then there is a $P(F^v)-$equivariant isomorphism between the affine building $\shi(\g F^\infty)$ of [\lc 4.2] and the (essentialization of the) fa\c{c}ade $\SHI_{F^v}^e$ of $\overline\SHI$.
\end{theo}

\begin{enonce*}[definition]{N.B} a) Of course in this theorem affine hovel must be understood according to the choice of $cl_\R^{\QD^{ti}}$, see \ref{3.10}.2 above. 
 If we drop the hypothesis (P10), then we get only a closed-restricted-hovel.

\par b) We may replace $\QD^{ti}$ by the non essential system $\QD^{tine}$ with $\QD^{tine}_{im}=(\sum_{\qa\in\QF}\,\R\qa)\setminus(\cup_{\qa\in\QF}\,\R\qa)$ \cf \ref{3.9}.1.

\par c) As we chose $cl=cl_\R^{\QD^{ti}}$ (or $cl=cl_\R^{\QD^{tine}}$), the Bruhat-Tits building $\SHI_{F^v}^e$ in 3) above is endowed with its $\R-$structure.

\par d) The hovel $\SHI$ inherits all properties proved in \LC. In particular it is endowed with a preorder relation $\stackrel{o}{<}$ (resp. $\leq{}$ if we are not in the generic-restricted case) inducing on each apartment $A$ the known relation $\stackrel{o}{<}_A$ (resp. $\leq_A$) associated to the Tits cone \cf \ref{2.3}.2b.

\par e) If a wall $M(\qa,k)$ contains a panel of a chamber $C\subset D(\qa,k)\subset\A$, then the chambers adjacent to $C$ along this panel are in one to one correspondence with $U_{\qa,k}/U_{\qa,k+}$ (\cf \cite[2.9.1]{Ru-10} and \ref{3.6}.4). In particular $\SHI$ is thick (\ref{2.3}.2a).

\par f) As $G=G^\emptyset.N$ and $N^\emptyset=G^\emptyset\cap N$ is the group of Weyl automorphisms of $\A$, the following proof tells that $G^\emptyset$ is the subgroup $G^w$ of Weyl automorphisms in $G$.
\end{enonce*}

\begin{proof} 1) It is sufficient to use the family $\shq^\emptyset$ in $G^\emptyset$. Axiom (MA1) is  then clear by definition and all the properties asked for axioms (MA2), (MA4) and (MAO) are proved in \ref{3.5}, \ref{3.6} or \ref{3.9}. \label{N26}
 If $F$ and $\g R$ in $\A$ are as in (MA3), then the Bruhat-Birkhoff-Iwasawa decomposition \ref{3.3}.4 and corollary \ref{3.8} prove that $G=Q(F).N.Q(\g R)$; it is classical that this proves (MA3).
 
 \par As the elements in $\qn(N)$ are vectorially Weyl automorphisms of $\A$ (\ref{2.3}.3.a) and $N$ is the stabilizer of $\A$ (\ref{3.6}.4), the elements in $G$ act by vectorially Weyl automorphisms.

 \par 2) The fixer $Q(\g f)$ of a spherical sector-face $\g f=x+F^v$ in $\A$ is in $P(F^v)$ (\ref{3.7}).
  So the map $\g f\mapsto F^v$ is well defined and onto the spherical facets of $\SHI^{vc}_\pm{}$.
   Consider $\g f_1$ and $\g f_2$, after shortening they are in a same apartment and then, by definition, they are parallel if and only if they correspond to the same $F^v$.
   So we have got the desired bijection. Now this bijection is clearly compatible with domination and opposition \cf \cite[3.1]{Ru-10}: it is an isomorphism of the twin buildings.

   \par 3) $\shi(\g F^\infty)$ is the set of sector-face-germs with direction $\g F^\infty$.
   Now in $\A$ we saw (\ref{2.5}.3) that the map $\g F=germ_\infty(x+F^v)\mapsto[x+F^v]$ identifies the apartment $\A(\g F^\infty)$ in $\shi(\g F^\infty)$ with $\A_{F^v}^e$.
   By \ref{3.8} $Q([x+F^v])=Q(\g F).N([x+F^v])$; but in $\A$ it is clear that $N([x+F^v])=N(\g F)$, so $Q([x+F^v])=Q(\g F)$.
    The identification of $\shi(\g F^\infty)$ and $\SHI_{F^v}^e$ is now clear, through a construction as in \ref{3.1}.
\end{proof}

    \subsection{Compatibility with enclosure maps}\label{3.12}

    \par We have proved good properties with respect to $cl^{\QD^{ti}}_\R$. But the example of Kac-Moody groups (\cite{GR-08} or  \ref{s5} below) proves that we may hope the following strong compatibility property.

    \begin{enonce*}[definition]{\quad1) Definition} The family $\shq$ of parahorics is {\it compatible with the enclosure map} $cl$ if for all non empty filter $\QO$ in a façade $\A_{F^v}$ and all vectorial chamber $C^v\in F^{v*}$, we have:  $Q(\QO)\cap U(\pm{}C^v)\subset Q(cl(\QO))$.
\end{enonce*}

\begin{enonce*}[definition]{\quad2) Remarks} a) Combined with (P9) and \ref{3.4} this implies $Q(\QO)\cap P(C^v)\subset Q(cl(\QO+C^v))$.

\par b) \label{N10'} Even for $cl=cl^{\QD^{ti}}_\R$ this is stronger than (P9), \eg if $\overline\QO+\overline{C^v}$ is not closed in $\A$ or $\QO$ not convex. It implies always (P10).

\par c) The most important case is when $\QO$ has an (half) good fixer. Then $Q(\QO)=Q(cl(\QO)).N(\QO)$, more precisely we may generalize \cite[prop. 4.3]{GR-08} :
\end{enonce*}

\begin{enonce*}[plain]{\quad3) Lemma} Suppose $\shq$ very good, compatible with $cl$ and $\QO\subset\QO'\subset cl(\QO)\subset\A$.

\par If $\QO$ has a good (or half good) fixer, then this is also true for $\QO'$ and $Q(\QO)=Q(\QO').N(\QO)$, $Q(\QO).N=Q(\QO').N$. In particular any apartment containing $\QO$ contains its enclosure $cl(\QO)$ and is conjugated to $\A$ by $Q(cl(\QO))$.

\par Conversely, if $supp(\QO)=\A$ (or $supp(\QO')=supp(\QO)$, hence $N(\QO')=N(\QO)$), $\QO$ has an half good fixer and $\QO'$ has a good fixer, then $\QO$ has a good fixer.
\end{enonce*}

 \begin{enonce*}[definition]{\quad4) Consequences} All the results proved in  \cite[sec. 4]{GR-08} are then true. For example the results in \ref{3.9} above for $cl_\R$ or $cl^{\QD^{ti}}_\R$ are true for $cl$; hence:
\end{enonce*}

\begin{enonce*}[plain]{\quad5) Theorem} If $\shq$ is a very good family of parahorics compatible with $cl$, then theorem \ref{3.11} is true with type $(\A,cl)$ instead of $(\A,cl^{\QD^{ti}}_\R$).
 If $cl=cl^\shp_{\QL'}$ and $\shp\subset\QD$ is tamely imaginary, we get an ordered affine hovel exactly as in \cite{Ru-10}, see \ref{3.10} 1) and 2).
\end{enonce*}

    \begin{enonce*}[definition]{\quad6) Definition} A {\it parahoric hovel} of  type $(\A,cl)$ is an ordered affine hovel, obtained from a valued root datum endowed with a very good family of parabolics compatible with $cl$. We suppose moreover $cl=cl^\shp_{\QL'}$ with $\shp\subset\QD$ tamely imaginary.
\end{enonce*}

\par A parahoric hovel has all properties of hovels and some other ones: the associated group $G$ acts strongly transitively by vectorially Weyl automorphisms, moreover \ref{3.3}.3 tells that any sector germ and any segment germ are in a same apartment.

 \subsection{Backwards constructions}\label{3.13}

 \begin{enonce*}[plain]{\quad1) Lemma} Let $\shi$ be an affine hovel  of type $(\A,cl)$ with a group $G$ acting on it strongly transitively.
  Then $G$ acts transitively on the apartments and  the stabilizer $N$ of an apartment $A$ induces in $\A$ a group $\qn(N)$ containing the group $W^{ath}$ generated by the reflections along the thick (hence true) walls.
 \end{enonce*}

 \begin{enonce*}[definition]{N.B}  The subgroup $W^{ath}$ of $W^a$ is equal to it when $\shi$ is thick.
 \end{enonce*}

 \begin{proof} Let $\g S_1\subset A_1$, $\g S_2\subset A_2$ be sector germs in apartments. By (MA3) there exists an apartment $A_3$ containing $\g S_1$ and $\g S_2$. By (MA2) there exists $g_1,g_2\in G$ with $A_1=g_1A_3$ and $A_2=g_2A_3$, so $A_1$ and $A_2$ are conjugated by $G$.

\par  If now $M$ is a thick wall in $A$, we write $D_1, D_2$ the half-apartments in $A$ limited by $M$.
  By \cite[2.9]{Ru-10} there is a third half-apartment $D_3$ in $\shi$ limited by $M$ such that for $i\not=j$, $D_i\cap D_j=M$ and $D_i\cup D_j$ is an apartment $A_{ij}$.
  By (M4) applied to a sector-panel-germ $\g F$ in $M$ and a sector-germ in $D_i$ (dominating the opposite in $M$ of $\g F$) there exists $g_{ijk}\in G$ with $g_{ijk}.A_{ij}=A_{ik}$ (where $\{1,2,3\}=\{i,j,k\}$). \label{N15}
   Now $A=A_{12}$ and $g_{142}.g_{231}.g_{123}$ (where $D_4=g_{231}.D_1$) stabilizes $A$ and exchanges $D_1$ and $D_2$: it is the reflection with respect to $M$.
  \end{proof}

 \par {\bf2)} Let $\shi$ and $G$ be as in the lemma. Then $G$ acts "nicely" (in particular strongly transitively) on the twin buildings $\shi^{\pm{}\infty}$ and we saw in \cite[3.8]{Ru-10} following \cite{T-92a}, that $G$ is often endowed with a RGD system.

 \par {\bf3)}  Suppose $G$ endowed with a generating root group datum such that the corresponding twin buildings $\SHI^{vc}_\pm{}$ are identified $G-$equivariantly with $\shi^{\pm{}\infty}$, in particular $G$ acts via positive, vectorial-type-preserving automorphisms. \label{N11}
   Then the action of $G$ on the affine buildings $\shi(\g F^\infty)$ (for $\g F^\infty$ a panel in  $\shi^{\pm{}\infty}$) should endow the root group datum with a valuation as in the classical case \cite[4.12]{Ru-10}.

 \par {\bf4)}  Suppose now the existence of a valuation of the root group datum which gives the affine buildings $\shi(F^v)$ on which $P(F^v)$ acts through $P(F^v)/U(F^v)$ (for any spherical vectorial facet $F^v$).
   Then $\shi$ is constructed as in \ref{3.1} with a family $\shq=(Q(x))_{x\in\A}$ of parahorics.
   We define also $\shq$ on $\overline\A^e_{sph}$ by the action of $G$ on the buildings $\shi(F^v)$. Let's look to the properties satisfied by $\shq$:

   \par (P1), (P2), (P4), (P5) and (P6) are clear by definition and hypothesis.

   \par By \cite[4.7]{Ru-10} $x\in\shi$ and $F^v\in\shi^{\pm{}\infty}$ (hence spherical) determine a unique sector face $x+F^v$ so (P9) is satisfied: $Q(x)\cap P(F^v)$ stabilizes $x+F^v$
   and, up to elements fixing $x+\overline{F^v}$, it stabilizes $\A$ and is vectorially Weyl, hence fixes $x+\overline{F^v}$.
   As $\shq$ is well known on $\overline\A^e_{sph}$, $Q(x)\cap P(F^v)$ fixes $(x+\overline{F^v})\cap\overline\A^e_{sph}$.


\par Now let $u\in U_{\qa,\ql}$ and $F^v$ a panel in Ker$\qa=M^\infty$.
Then by [\lc sec.4] $u\A$ is an apartment of the building $\shi(M^\infty)\simeq\shi(F^v)$ (which is a tree) and its intersection with $\A$ is an half-apartment $D(\qa,\qm)$. \label{N12}
 But by definition of the valuation $u$ fixes $pr_{F^v}(D(\qa,\ql))\subset\A^e_{F^v}$; so $\A\cap u\A\supset D(\qa,\ql)$ hence $u$ fixes $D(\qa,\ql)$. So (P3) is satisfied.

    \par For $x\in\A$ and $g\in G$, suppose $g\in Q(x).N\cap P(F^v).N$ then $g\A\ni x$, $g\A^v\supset F^v$ and $g\A$ contains the sector face $x+F^v$ \cite[4.7]{Ru-10}.
    So by (MA2) $g\in Q(x+\overline{F^v}).N$. But $Q(x+\overline{F^v})\supset Q(x)$ and $Q(x+\overline{F^v})\supset Q(pr_{F^v}(x))$ as $pr_{F^v}(x)\in\A_{F^v}$ is the class of $x+\overline{F^v}$. Hence (P7) is satisfied.
    
   \par (P10) is satisfied as $\shi$ is an hovel (not a closed-restricted-hovel).

    \par When $\overline\A=\overline\A^e$, Charignon proved that (P8) is satisfied for every good family of parahorics (\ref{3.3}.6). \label{N27} We may also use a geometrical translation of (P8) for good families satisfying (P9):
     let $x\in\A$ and $\g s$ a sector of origin $x$ in $\A$, then any apartment $A'$ containing $x$ contains also a sector $\g s_1$ of origin $x$ opposite $\g s$ (in an apartment containing them both).

    \par So if $\overline\A=\overline\A^e$ is essential, we know that the family $\shq$ (defined on $\A\cup\overline\A_{sph}$) is very good.

     \par {\bf5)} These sketchy constructions reduce more or less the classification problem for affine hovels with a good group of automorphisms to the problem of existence (or uniqueness ?) of very good (excellent ?) families of parahorics associated to valuated RGD systems.

  \begin{prop}\label{3.14} We consider a group $G$ (resp. $G'$) acting strongly transitively on an ordered affine hovel $\shi$ (resp. $\shi'$) and a map $j:\shi\to\shi'$ which is $G-$equivariant with respect to an homomorphism $\qf:G\to G'$. We suppose that:

  \par 1) There exist apartments $A\subset\shi$ and $A'\subset\shi'$ such that $j\,\rule[-1.5mm]{.1mm}{3mm}_{\, A}$ is injective affine from $A$ to $A'$.

  \par 2) There exists a sector germ $\g S$ in $A$ such that the direction of the cone $j(\g S)$ meets the interior of the Tits cone $\sht'_\pm{}$ in $\vect{A'}$.

  \par Then $j$ is injective.
  \end{prop}

  \begin{enonce*}[definition]{N.B}  We exclude here the closed-restricted-hovels. For buildings the proof is easier, as two points are in a same apartment.
\end{enonce*}

  \begin{proof} Let $x_1,x_2\in\shi$ such that $j(x_1)=j(x_2)$. There is an apartment $A_i=g_iA$ containing $x_i$ and $\g S$, with $g_i$ fixing pointwise a sector $\g s$ in $\g S$.
  Then $A'_i=\qf(g_i)A'$ is an apartment containing $j(x_i)$ and $j(\g s)$ with $\qf(g_i)$ fixing pointwise $j(\g s)$.
  Let's consider $y\in\g s$ sufficiently far away; then $[y,x_i]$ and $j([y,x_i])=[j(y),j(x_i)]$ are preordered (even generic) segments in $A_i$ and $A'_i$.
 But $j(x_1)=j(x_2)$, so $[j(y),j(x_1)]=[j(y),j(x_2)]$ (axiom (MAO) ).
  As $g=g_2g_1^{-1}$ fixes pointwise the segment germs $[y,x_1)$ and $[y,x_2)$, $\qf(g)$ fixes pointwise $[j(y),j(x_1))=[j(y),j(x_2))$ and, as $\qf(g)$ is an affine isomorphism from $A'_1$ to $A'_2$, it fixes pointwise the whole segment $[j(y),j(x_1)]=[j(y),j(x_2)]$.
  Then $g[y,x_1]$ and $[y,x_2]$ are two segments in $A_2$ with the same image $[j(y),j(x_2)]$ in $A'_2$ by $j$ (injective on the apartments).
  So these segments are equal; in particular $[y,x_1)=[y,x_2)$.

  \par Now $j([y,x_1])=j([y,x_2])$ and $[y,x_1)=[y,x_2)$. Then $[y,x_1]\cap[y,x_2]$ is a segment $[y,z]$ (\cf (MA2) for open-segment-germs, as we avoid \ref{3.10}.3) with $z\not=y$.
  We are done if $z=x_1$ or $z=x_2$. Otherwise $[z,x_1)$ and $[z,x_2)$ are distinct segment germs in a same apartment \cite[5.1]{Ru-10} with the same image by $j$, contrary to the hypothesis.
 \end{proof}

 \subsection{Simplifications of the axioms}\label{4.15}
 
 \par For simplicity we suppose $\QD$ tamely imaginary.
 
  \par {\bf1)} Let $G$ be a group acting on an affine ordered hovel $\SHI$ by vectorially Weyl automorphisms. It is proved in \cite{CiR15} that strong transitivity is equivalent to any of the two following conditions:
  
  \par For any local chamber (resp. sector germ) $\QO$ in $\SHI$ and any two apartments $A,A'$ containing $\QO$, there is $g\in G$ fixing pointwise $\QO$ such that $A'=g.A$.
 
 \begin{enonce*}[plain]{\quad\;\,2) Proposition} In the definition of an affine ordered hovel (resp. an affine ordered generic-restricted-hovel), we may replace axiom (MA2) by the axiom we get when we allow $F$ to be only a preordered open-segment-germ or a solid chimney (resp. only a generic open-segment-germ)
\end{enonce*}

\begin{proof} Let $A,A'$ be two apartments containing a point $x$.
We consider in $A$ a generic segment germ $[x,y)$ and in $A'$ a sector germ $\g S'$.
By (MA3) there is an apartment $A''$ containing $[x,y)\cup\g S'$.
The intersection $A'\cap A''$ contains  a sector $\g s'$ (with germ $\g S'$).
We choose $z\in\g s'$ sufficiently far; so $[x,z]$ is a generic segment in $A'$ and by (MA0) this is also a line segment in $A''$.
By (MA2) for generic open-segment-germs, there is a Weyl isomorphism $\qf:A\to A''$ (resp. $\psi:A'\to A''$) fixing $cl(]x,y))\supset cl(x)$ (resp.  $cl(]x,z))\supset cl(x)$).
Then $\psi^{-1}\circ\qf$ is the expected Weyl isomorphism from $A$ to $A'$.

\par Let $\qd$ be a generic ray  with origin $x$ in an apartment $A$ and $A'$  an apartment containing $\qd$.
Then $\g R=cl_A(germ_\infty(\qd))$ is a splayed chimney germ and $cl_A(germ_x(\qd))$ a spherical closed-facet.
By (MA2) for generic open-segment-germs, $cl_A(germ_x(\qd))\subset A'$.
By (MA0) and (MA2) for generic open-segment-germs, $\qd$ is closed convex in a generic line of $A'$.
But $\qd$ has only one endpoint ($x$) in $A$, so $\qd$ is a generic ray in $A'$.
We consider the splayed chimney germ $\g R'=cl_{A'}(germ_\infty(\qd))$.
By (MA3) and (MA4), there is an apartment $A''$ containing $\g R\cup\g R'$ and, by (MA4), it is clear that $\g R=cl_{A''}(germ_\infty(\qd))=\g R'$ in $A''$.
So $A\cap A'\supset cl_A(germ_x(\qd)) \cup cl_A(germ_\infty(\qd))$ and, by (MA4), there is a Weyl isomorphism from $A$ to $A'$ fixing $cl_A(germ_x(\qd) \cup germ_\infty(\qd))=cl_A(\qd)$.

\par A splayed chimney is the enclosure of a spherical facet and its (splayed) chimney germ.
So (MA2) for splayed chimneys is a consequence of (MA4) (with $F$ a spherical facet).
\end{proof}

\section{Hovels and bordered hovels for split Kac-Moody groups}\label{s5}

  \par We consider now the situation of \ref{2.2} and \ref{2.3b} and shall build a very good family $\widehat\shp$ of parahorics following \cite{Ru-11}. We choose the enclosure map $cl=cl^\QD$.

   \subsection{The parahoric subgroup associated to $y\in\A$}\label{5.1}

        \par {\quad\bf1) The free case with $V=V^x$}: In \cite{Ru-11} the RGS $\shs$ is  supposed free and the affine apartment $\A$ is equal to $\A^x$ with associated vector space $V^x=Y\otimes_\Z\R$.
        Then for $y\in\A$, one defines the group $\widehat P(y)=U_y^{pm+}.U_y^{nm-}.N(y)=U_y^{nm-}.U_y^{pm+}.N(y)$ where $N(y)$ is the fixer of $y$ in $N$ and $U_y^{pm+}$ (resp. $U_y^{nm-}$) is the intersection with $G$ or $U^+$ (resp. $U^-$) of a group $U_y^{ma+}=\prod_{\qa\in\QD^+}\,U_\qa(y)$ (resp. $U_y^{ma-}=\prod_{\qa\in\QD^-}\,U_\qa(y)$) which exists in a suitable completion $G^{pma}$ (resp. $G^{nma}$) of the Kac-Moody group $G$ [\lc 4.5, 4.14];
        actually one has to define suitably $U_\qa(y)$ for $\qa\in\QD_{im}$: $U_\qa(y)=U_{\qa,-\qa(y)}:=U^{ma}_{\{y\}}(\{\qa\})$ in the notations of [\lc 4.5.2].

        \par The group $U^{\pm}_y=U^\pm{}\cap G(y)$ of \ref{3.2}.1 is clearly included in $\widehat P(y)$.
        As $U_y^{pm+}=U^+\cap \widehat P(y)$, we have $U^+_y\subset U_y^{pm+}$ and, similarly, $U^-_y\subset U_y^{nm-}$.
   Moreover we know that $\widehat P(y)=U_y^{pm+}.U^-_y.N(y)=U_y^{nm-}.U^+_y.N(y)$ [\lc 4.14];

   \par     The interesting point for us is that $U_y^{ma+}$, $U_y^{ma-}$, $U_y^{pm+}$, $U_y^{nm-}$, $U^+_y$ or $U^-_y$ depend only of the true half-apartments (imaginary or not) containing $y$.
   In particular they depend only of the class $\overline y$ of $y$ in the essentialization $\A^q=\A^x/V_0$.

   \par In the classical case where $\QF$ is finite (and $\QD_{im}$ empty) the group $U_y^{pm+}$ (resp. $U_y^{nm-}$) is the group $U^{++}_y$ (resp. $
   U^{--}_y$) generated by the groups $U_\qa(y)$ for $\qa\in\QF^+$ (resp. $\qa\in\QF^-$).

        \par {\bf2)}  Consider now any RGS $\shs$, any affine apartment $\A$ as in \ref{2.3b} for the root datum in $G=\g G_\shs(K)$ and any $y\in\A$.
        By \cite[1.3, 1.11]{Ru-11} there is an injective homomorphism $\qf:G\hookrightarrow G^{xl}=\g G_{\shs^l}(K)$ where $\shs^l$ is a free RGS.
        The affine apartment associated to it is $\A^{xl}$ and we know that the essentializations of $\A$ and $\A^{xl}$ are equal: $\A/V_0=\A^{xl}/V_0^{xl}=\A^q$.

        \par To $\overline y\in\A^q$ we associated above some subgroups of $G^{xl}$.
        By [\lc 1.9.2, 3.19.3] the groups $U^\pm{}$ in $G$ and $G^{xl}$ are isomorphic by $\qf$, so $U_{\overline y}^{pm+}$, $U_{\overline y}^{nm-}$,  $U_{\overline y}^{+}$ and $U_{\overline y}^{-}$ are actually in $G$ (and if $\shs$ is free, they are as defined in 1) above).
        We define the group $\widehat P^m(\overline y)$ as generated by $U_{\overline y}^{pm+}$, $U_{\overline y}^{nm-}$ and $\g T_\shs(\sho)\subset Z_0=$ Ker$\qn$.
        We define $N(y)$ the fixer of  $y$ in $N$ and $\widehat P(y)=\widehat P^m(\overline y).N(y)$ which is called the {\it fixer group} associated to $y$ in $\A$ (\cf \ref{5.2}b below).

  \begin{lemm}\label{5.2} a) We have
  $\widehat P^m(\overline y)=U_{\overline y}^{pm+}.U_{\overline y}^{nm-}.N^m({\overline y})=U_{\overline y}^{nm-}.U_{\overline y}^{pm+}.N^m({\overline y})=$ \goodbreak\noindent
  $U_{\overline y}^{pm+}.U_{\overline y}^{-}.N^m({\overline y})=U_{\overline y}^{nm-}.U_{\overline y}^{+}.N^m({\overline y})$
  where $N^m({\overline y})$ is a subgroup of $N(y)$, hence fixing pointwise $y+V_0\subset\A$.

  \par b) Moreover $\widehat P^m(\overline y)$ does not change when one changes $\QF^+$ by $W^v$, hence it is normalized by $N(y)$ and $\widehat P(y)$ is a group.
  \end{lemm}
    \begin{proof} a) We identify $G$ and $\qf(G)\subset G^{xl}$. We choose an origin in $\A$ (resp. $\A^{xl}$) fixed by $\qn(m(x_{\qa_i}(1)))$, $\forall i\in I$; hence $\A$ (resp. $\A^{xl}$)  is identified with $V$ (resp. $V^{xl}$) and $\qn(N)$ (resp. $\qn^{xl}(N^{xl})$ where $N^{xl}=N_{\shs^l}(K)$) with $\qn(T)\rtimes W^v$ (resp. $\qn^{xl}(T^{xl})\rtimes W^v$) where $W^v$ acts linearly via $\qn^v$.
    Actually $\qn:T\to V$ factorizes through $T/\g T(\sho)=Y\otimes\QL$: $\qn(t)=\overline\qn(\overline t)$ where $\overline t$ is the class of $t$ modulo $\g T(\sho)$; and the same thing for $\qn^{xl}$.
    We consider $z\in V^{xl}$ such that $\overline z=\overline y\in V^q$.

    \par By [\lc 4.6, 4.14] we have $U_{\overline y}^{pm+}.U_{\overline y}^{nm-}.N^{min}({\overline y})\subset \widehat P^m(\overline y)\subset \widehat P^{xl}(z)=U_{\overline y}^{pm+}.U_{\overline y}^{nm-}.N^{xl}(z)=U_{\overline y}^{pm+}.U_{\overline y}^{-}.N^{xl}(z)$
    where $N^{min}({\overline y})$ is a subgroup of $N$ and $N^{xl}(z)$ the fixer in $N^{xl}$ of $z$. Moreover $N^{xl}(z)=N^{xl}\cap \widehat P^{xl}(z)$.
    It is now clear that $\widehat P^m(\overline y)=U_{\overline y}^{pm+}.U_{\overline y}^{nm-}.N^m({\overline y})=U_{\overline y}^{pm+}.U_{\overline y}^{-}.N^m({\overline y})$
    with $N^m({\overline y})=\widehat P^m(\overline y)\cap N^{xl}(z)=\widehat P^m(\overline y)\cap N\subset N\cap N^{xl}(z)$.
     The same thing is clearly true when exchanging $U_{\overline y}^{pm+}$, $U_{\overline y}^{+}$ and $U_{\overline y}^{nm-}$, $U_{\overline y}^{-}$ .

    \par Let $n=tw\in N\cap N^{xl}(z)$ with $t\in T$ and $w\in W^v$ (fixing $0$).
    We have $z=nz=\qn^{xl}(t)+w(z)$. But, if $w=s_{i_1}.\cdots.s_{i_n}\in W^v$, $z-w(z)=\sum_{j=1}^n\,(s_{i_{j+1}}.\cdots.s_{i_{n}}(z)-s_{i_{j}}.\cdots.s_{i_{n}}(z))=\sum_{j=1}^n\,\qa_{i_{j}}(s_{i_{j+1}}.\cdots.s_{i_{n}}(z)).\qa_{i_j}^\vee=:\partial(\overline z,V^{xl})$ an element of $V^{xl}$ depending only of $\overline z=\overline y$.
    Hence $\overline\qn^{xl}(\overline t)=\partial(\overline z,V^{xl})$; but $\overline\qn^{xl}$ is one to one, so $\overline t\in(\sum_{i\in I}\,\R\qa_i^\vee\otimes 1)\cap Y\otimes\QL$.
    By \ref{2.3b}.3a, there exists $r\in\Z_{>0}$ with $r\overline t=-\sum_{i\in I}\,\qa_i^\vee\otimes\ql_i$ with $\ql_i\in\QL$ a suitable $\Z-$linear combination of the coefficients $r.\qa_{i_{j}}(s_{i_{j+1}}.\cdots.s_{i_{n}}(z))\in\R$ (as the relations between the $\qa_i^\vee$ in $Y\subset Y^{xl}$ have coefficients in $\Q$).
    Now  $\overline\qn(r\overline t)=\sum_{i\in I}\,\ql_i\qa_i^\vee\in V$ and, by the expression of the $\ql_i$, $r\overline\qn(\overline t)=r\partial(\overline z,V)$ (as the $\qa_i^\vee$ in $V$ satisfy the $\Z-$linear relations between the $\qa_i^\vee$ in $Y$).
    In $V$ we may divide by $r$, so  $\overline\qn(\overline t)=\partial(\overline z,V)$.
    By the same calculations as above $\qn(n)$ fixes any element $y$ with $\overline y=\overline z$.

    \par b) It is proved in \cite[4.6c]{Ru-11} that $U_{\overline y}^{pm+}.U_{\overline y}^{nm-}.N^{min}({\overline y})$ doesn't change when one changes $\QF^+$ by $W^v$. So it is the same for the subgroup $\widehat P^m(\overline y)$ it generates.
    \end{proof}

    \subsection{The fixer group associated to $y\in\overline\A\setminus\A$}\label{5.3}

    For $F^v$ a non minimal vectorial facet, the fa\c{c}ade $\A_{F^v}$ is an affine apartment for the group $P(F^v)/U(F^v)=G(F^v)\simeq M(F^v)$ endowed with the generating root datum $(G(F^v),(U_\qa)_{\qa\in\QF^m(F^v)},Z)$ \cf \ref{1.3c}.5 and \ref{2.3}.
    Moreover $G(F^v)$ is actually the group of $K-$points of a Kac-Moody group: if $F^v=F^v_\qe(J)$ then $G(F^v)=G(J)=\g G_{\shs(J)}(K)$.

    \par So for $y\in \A_{F^v}$ we may define $\widehat P(y)$ as the subgroup of $P(F^v)$  inverse image of the subgroup $\widehat P_{F^v}(y)$ constructed inside $G(F^v)\simeq M(F^v)$ as above.
    We have $\widehat P(y)=\widehat P_{F^v}(y)\ltimes U(F^v)=U_{F^vy}^{pm+}.U_{F^vy}^{nm-}.N(y).U(F^v)=U_{F^vy}^{nm-}.U_{F^vy}^{pm+}.N(y).U(F^v)=U_{F^vy}^{pm+}.U_{F^vy}^{-}.N(y).U(F^v)=U_{F^vy}^{nm-}.U_{F^vy}^{+}.N(y).U(F^v)$
    where $+$ and $-$ refer to the choice of a chamber $C^v\in F^{v*}$.

    \begin{enonce*}[definition]{Remark}  Even if $F^v=V_0$ is minimal, a point $\overline y\in\A^e_{F^v}=\A_{F^v}/ \vect{F^v}$ corresponds to a collection $y+\vect{F^v}$ of points $y\in \A^{ne}_{F^v}=\A$.
    So we have two parahoric subgroups $\widehat P(y)\subset\widehat P(\overline y)= \widehat P(y).N(F^v)(\overline y)$ and $N(F^v)(\overline y)$ acts by translations on $y+\vect{F^v}$.
    We say that $\widehat P(y)$ (resp. $\widehat P(\overline y)$) is the strong (resp. weak) fixer of $y$ or $\overline y$.
\end{enonce*}

\begin{defi}\label{5.4} We define $\widehat\shp$ as the family $(\widehat P(y))_{y\in\overline\A}$. By construction it is a family of parahorics.
 The corresponding hovel (resp. bordered hovel) will be written $\SHI=\SHI(\g G_\shs,K,\A)$ (resp. $\overline\SHI=\overline\SHI(\g G_\shs,K,\overline\A)$ ) and called the {\it affine hovel} (resp. {\it affine bordered hovel}) of $\g G_\shs$ over $K$ with model apartment $\A$ (resp. $\overline\A$).
 When we add the adjective essential we mean that $\A=\A^q$ (resp. $\overline\A=\overline\A^e$).

 \par It is perhaps possible that $\widehat\shp=\shp$ \cite[4.13.5]{Ru-11}, see also \ref{5.11}.4c.
 \end{defi}

 \begin{lemm}\label{5.5} Let $x\in\overline\A$, $F^v=F^v(x)$ and $C^v$ a chamber in $F^{v*}$; then:
 \par $\widehat P(x)\cap N=N(x)$ ; $\widehat P(x)\cap N.U(C^v)=N(x).U_{F^vx}^{pm+}.U(F^v)$ ;  $\widehat P(x)\cap U(C^v)=U_{F^vx}^{pm+}.U(F^v)$
 \par $\widehat P(x)\cap N.U(-C^v)=N(x).U_{F^vx}^{nm-}$ and $\widehat P(x)\cap U(-C^v)=U_{F^vx}^{nm-}$.
 \end{lemm}

 \begin{proof} Let $g\in \widehat P(x)\cap N.U(C^v)$. So $g=nu^+=n'v^-v^+u_{F^v}$ with $n\in N$, $u^+\in U(C^v)$, $n'\in N(x)$, $v^-\in U_{F^vx}^{nm-}$, $v^+\in U_{F^vx}^{pm+}$ and $u_{F^v}\in U(F^v)$.
 Hence $(n'^{-1}n)(u^+u_{F^v}^{-1}(v^+)^{-1})=v^-\in N.U(C^v)\cap U(-C^v)$.
 By the uniqueness in the Birkhoff decomposition (\ref{1.3c}.2) we have $v^-=1$, $n=n'$ and $u^+=v^+u_{F^v}$ so $g\in N(x).U_{F^vx}^{pm+}.U(F^v)$.
 If moreover $g\in N$ (resp. $g\in U(C^v)$) we have $u^+=1$ (resp. $n=1$) hence $g=n'\in N(x)$ (resp. $g=v^+u_{F^v}\in U_{F^vx}^{pm+}.U(F^v)$).

 \par Now let $g\in \widehat P(x)\cap N.U(-C^v)$. We write $g=nu^-=n'v^+v^-u_{F^v}=n'v^+u_{F^v}'v^-$ (with obvious notations).
 Hence $(n'^{-1}n)(u^-(v^-)^{-1})=v^+u_{F^v}'\in N.U(-C^v)\cap U(C^v)$.
 So $v^+u_{F^v}'=1$ (hence $v^+=u_{F^v}'=1$, as $P(F^v)=M(F^v)\ltimes U(F^v)$), $n=n'$ and $u^-=v^-$.
 We have $g=n'v^-\in N(x).U_{F^vx}^{nm-}$ and, if $g\in U(-C^v)$, $n=n'=1$ so $g=v^-\in U_{F^vx}^{nm-}$.
 \end{proof}

  \begin{prop}\label{5.6} The family $\widehat \shp$ is a very good family of parahorics. It is compatible with the enclosure map $cl^\QD$. Hence $\SHI(\g G_\shs,K,\A)$ is a thick parahoric hovel of type $(\A,cl^\QD)$ and $G=\g G(K)$ acts strongly transitively on it (by vectorially Weyl automorphisms).
   \end{prop}
 \begin{proof} We proved  (P5) in the lemma above. If $F^v(x)$ is spherical, $U_{F^vx}^{pm+}$ and $U_{F^vx}^{nm-}$ are generated by the groups $U_\qa(x)$ for $\qa\in\QF^{m\pm{}}(F^v)$, so (P6) holds.
 By definition $\widehat P(x)=U_{F^vx}^{pm+}.U_{F^vx}^{nm-}.N(x).U(F^v)=U_{F^vx}^{nm-}.U_{F^vx}^{pm+}.N(x).U(F^v)$, so (P8) is a consequence of lemma \ref{5.5}.
 We have also $\widehat P(x)=(\widehat P(x)\cap U(-C^v)).(\widehat P(x)\cap U(C^v)).N(x)$.

 \par Consider now the situation of (P7) or (P9). We have to prove $N.\widehat P(x)\cap P(F^v)\subset N.\widehat P(\{x,pr_{F^v}(x)\})$ and $\widehat P(x)\cap P(F^v)\subset \widehat P(\overline{x+F^v})$.
 These relations are in $P(F_1^v)$ and each side contains $U(F_1^v)$, so we may argue in $G(F_1^v)=P(F_1^v)/U(F_1^v)$. Actually we shall suppose $x\in\A$.
  Consider a chamber $C^v\in F^{v*}$, we have $P(F^v)=U(F^v)\rtimes M(F^v)$ and (by Iwasawa)  $M(F^v)=(U(C^v)\cap M(F^v)).N(F^v).G(\QF^m(F^v),x)$ with $G(\QF^m(F^v),x)\subset \widehat P(\overline{x+F^v})$.
  Let $g\in N.\widehat P(x)\cap P(F^v)$ (resp. $g\in \widehat P(x)\cap P(F^v)$). We write $g=n'q=u_{F^v}v^+nq'$ with $n'\in N$ (resp. $n'=1$), $q\in\widehat P(x)$, $u_{F^v}\in U(F^v)$, $v^+\in U(C^v)\cap M(F^v)$, $n\in N(F^v)$,  $q'\in  \widehat P(\overline{x+F^v})$ and we want to prove that $g\in N. \widehat P(\overline{x+F^v})$ (resp. $g\in  \widehat P(\overline{x+F^v})$).
  So one may suppose $q'=1$, then $g\in nU(n^{-1}C^v)$ and $q\in n'^{-1}nU(n^{-1}C^v)$. By the proof of \ref{5.5} $q\in n'^{-1}nU_x^{pm}(n^{-1}C^v)$ with $n'^{-1}n\in N(x)$ and, as $n\in N(F^v)$, $U_x^{pm}(n^{-1}C^v)\subset\widehat P(\overline{x+F^v})$.
  So $g=n'q\in N.\widehat P(\overline{x+F^v})$ (resp. $g=q\in \widehat P(\overline{x+F^v})$, as $n=n'^{-1}n\in N(x)\cap N(F^v)\subset N(\overline{x+F^v})$).

  \par By \ref{5.5} $\widehat P(x)\cap U(C^v)=U_{F^vx}^{pm+}\ltimes U(F^v)$ and $\widehat P(x)\cap U(-C^v)=U_{F^vx}^{nm-}$.
  So $\widehat P(\QO)\cap U(C^v)=U_{F^v\QO}^{pm+}\ltimes U(F^v)$ and $\widehat P(\QO)\cap U(-C^v)=U_{F^v\QO}^{nm-}$ and these groups depend only of $cl^\QD(\QO)$ \cite[4.5.4f]{Ru-11}.
  We have proved that $\widehat\shp$ is compatible with $cl^\QD$.
 \end{proof}

 \subsection{Remarks}\label{5.7}
  {\quad\bf1)} So we get for $\SHI$ and $\overline\SHI$ all the properties proved in section \ref{s3}.
   The map $g\g Tg^{-1}\mapsto A(g\g Tg^{-1})=g\A$ (resp. $g\g Tg^{-1}\mapsto \overline A(g\g Tg^{-1})=g\overline\A$) is a bijection between the split maximal tori in $\g G_\shs$ and the apartments  in $\SHI$ (resp. the bordered apartments  in $\overline\SHI$) \cf \ref{3.6}.4 and \ref{1.4b}.1.

 \par {\bf2)} Actually we proved (P7) and (P9) even when $F^v$ is non spherical. So one may define a projection $pr_{F^v}:\SHI_{F^v_1}\to\SHI_{F^v}$ even if $F^v\in F_1^{v*}$ is non spherical \cite[11.7.3]{Cn-10b}.
 This gives stronger links between the hovel $\SHI$ and its non spherical fa\c{c}ades.

 \par {\bf3)} For (P8) we proved also $\widehat P(x)=(\widehat P(x)\cap U(-C^v)).(\widehat P(x)\cap U(C^v)).N(x)$ which improves (P8) essentially when $F^v(x)$ has a well defined sign.

 \par{\bf4)} If we choose $\A$ as in \cite[4.2]{Ru-11} (which implies $\shs$ free) then $\SHI(\g G_\shs,K,\A)$ is the affine hovel $\SHI(\g G_\shs,K)$ defined in [\lc 5.1], with the same action of $G=\g G_\shs(K)$, the same apartments, the same enclosure map, the same facets, ...).
  By lemma \ref{5.5} the notions of (half) good fixers for filters in $\A$ are the same.
  Note however that, when (and only when) $\QO$ has not an (half) good fixer, $\widehat P(\QO)$ may be different from $\widehat P_\QO$ as defined in \cite{Ru-11}.

  \par The group $G^\emptyset$ of Weyl automorphisms in $G$ (\ref{2.3}.3c, \ref{3.11}f) is equal to $\psi(G^A).Z_0$ defined in \cite[5.13.2 or]{Ru-11} (as $\psi^{-1}(N)=N^A$ and $\qn(\psi(N^A))=W^a$).

 \par{\bf5)} A point $\overline x\in \SHI_{F^v}$ determines a sector-face-germ $\g F=germ_\infty(x+F^v)$ of direction $F^v$ in $\SHI$ and the correspondence is one to one if $\overline\A=\overline\A^e$ (or $\overline\A=\overline\A^i$ and $F^v$ non trivial) \cf \ref{3.11}.
 The strong (resp. weak) fixer of $\overline x$ (\cf \ref{5.3}) is the set of the $g$ in $G$ which fix pointwise an element (resp. which induce a bijection between the sets which are elements) of the filter $\g F$.

 \subsection{Functoriality}\label{5.8}

 \par{\quad\bf 1) Changing the group, commutative extensions:} Let's consider a commutative extension of RGS $\qf:\shs\to\shs'$ \cite[1.1]{Ru-11}. We then get an homomorphism $\g G_\qf:\g G_\shs\to\g G_{\shs'}$ inducing homomorphisms $\g T_\qf:\g T_\shs\to\g T_{\shs'}$, $\g N_\qf:\g N_\shs\to\g N_{\shs'}$ and isomorphisms $\g U^{\pm}_{\shs}\to\g U^{\pm}_{\shs'}$.
  If $\A$ is a suitable apartment for $(\g G_{\shs'},\g T_{\shs'})$ (\ref{2.3b}.3a) it is clearly suitable for  $(\g G_{\shs},\g T_{\shs})$ and, for $x\in\A_{F^v}$, $U^{pm+}_{F^vx}$ or $U^{nm-}_{F^vx}$ is the same for $\g G_{\shs}$ or $\g G_{\shs'}$.
   Hence $\widehat P_{\shs'}(x)=G_\qf(\widehat P_{\shs}(x)).N_{\shs'}(x)$.
    But $G_\qf^{-1}(N_{\shs'})=N_\shs$ [\lc 1.10] hence Ker$G_\qf\subset T_\shs$, so the lemma \ref{5.5} tells that $G_\qf^{-1}(\widehat P_{\shs'}(x))=\widehat P_{\shs}(x)$.
     It is now clear that $G_\qf\times Id_\A$ induces a $G_\qf-$equivariant embedding $\SHI(\g G_\qf,K,\A):\SHI(\g G_\shs,K,\A)\hookrightarrow\SHI(\g G_{\shs'},K,\A)$ which is an isomorphism (bijection between the sets of apartments, isomorphism of the apartments).
      Hence the affine Weyl groups $W^a$ are the same, but $\qn(N_\shs)\subset\qn(N_{\shs'})$ are in general different.

      \par The same things are true for the bordered hovels and the embeddings are functorial (note however that $\A$ or $\overline\A$ depends on $\g G'$).

  \par{\quad\bf 2) Changing the group, Levi factors:}
  For a vectorial facet $F^v$, we may  consider the homomorphism $M(F^v)\into G$.
       More precisely if $F^v=F^v_\qe(J)$ then $\g G_{\shs(J)}$ embeds into $\g G_{\shs}$ (\ref{1.3c}.5 and \ref{1.4b}.1).
        If $\A$ is suitable for $\g G_{\shs}$ then it is also suitable for $\g G_{\shs(J)}$, but we have only to consider the walls of direction Ker$\qa$ with $\qa\in Q(J)$.
         By construction $\SHI(\g G_{\shs(J)},K,\A)$ is $G_{\shs(J)}-$equivariantly isomorphic to the fa\c{c}ade  $\SHI(\g G_\shs,K,\overline{\overline\A})_{F^v}$ for $F^v= F^v_\qe(J)$ or $F^v_{-\qe}(J)$ or any other maximal vectorial facet in $\cap_{i\in J}\,$Ker$\qa_j$.
       Clearly for $x\in\A$, $\widehat P_{\shs(J)}(x)\subset\widehat P(x)$ and $N_{\shs(J)}\subset N$, so $\SHI(\g G_{\shs(J)},K,\A)$ maps onto $\SHI(\g G_{\shs(J)},\g G_\shs,K,\A):=G_{\shs(J)}.\A\subset\SHI(\g G_\shs,K,\A)$ and the projection $pr_{F^v}$ maps $\SHI(\g G_{\shs(J)},\g G_\shs,K,\A)$ onto $\SHI(\g G_\shs,K,\overline{\overline\A})_{F^v}$.
        So the three sets $\SHI(\g G_{\shs(J)},K,\A)$, $\SHI(\g G_{\shs(J)},\g G_\shs,K,\A)$  and $\SHI(\g G_\shs,K,\overline{\overline\A})_{F^v}$ are $G_{\shs(J)}-$equivariantly isomorphic.

        \par For $\g G_{\shs(J)}$, the bordered apartment associated to $\A$ is a union of fa\c{c}ades with direction facets for $\QF(J)$.
        These facets are in one to one correspondence with the facets in $F^{v*}$, for $F^v$ as above.
                Let $\overline\A_J^i$,  $\overline\A_J^e$ and $\overline{\overline\A}_J$ be the three possible apartments as in \ref{2.5}.
         Then $\overline\SHI(\g G_{\shs(J)},K,\overline\A_J^e)$ (resp. $\overline\SHI(\g G_{\shs(J)},K,\overline{\overline\A}_J)$ ) is isomorphic to $\overline\SHI(\g G_\shs,K,\overline\A^e)_{F^v}$ (resp. $\overline\SHI(\g G_\shs,K,\overline{\overline\A})_{F^v}$) as defined in \ref{3.2}.3c.
          And $\overline\SHI(\g G_{\shs(J)},K,\overline\A_J^i)$ is isomorphic to $\overline\SHI(\g G_\shs,K,\overline\A^i)_{F^v}$ where we remove $\SHI(\g G_\shs,K,\overline\A^i)_{F^v}$ and add $\SHI(\g G_{\shs(J)},\g G_\shs,K,\A)$.


  \par{\bf 3) Changing the field:} Let's consider a field extension $i:K\into L$ and suppose that the valuation $\qo$ may be extended to $L$.
  Then $\g G_\shs(K)$ embeds via $\g G_\shs(i)$ into $\g G_\shs(L)$.
  If $\A$ is suitable for $\g G_\shs(L)$, it is also suitable for $\g G_\shs(K)$; the three examples of \ref{2.3b}.1 on $K$ and $L$ are corresponding this way each to the other.
  There are also embeddings $\g G_\shs^{pma}(K)\into\g G_\shs^{pma}(L)$, $\g G_\shs^{nma}(K)\into\g G_\shs^{nma}(L)$ and it is clear that, for $x\in\A$, $U^{pm+}_{Kx}=U^{pm+}_{Lx}\cap\g G_\shs(K)$, $U^{nm-}_{Kx}=U^{nm-}_{Lx}\cap\g G_\shs(K)$ and $\g N(K)(x)=\g N(L)(x)\cap\g G_\shs(K)$.
  So, using Iwasawa decomposition for $\g G_\shs(K)$, \ref{5.5} and uniqueness in Birkhoff decomposition for $\g G_\shs(L)$, we have:

\par  $\widehat P_L(x)\cap\g G_\shs(K)=\widehat P_K(x).(\widehat P_L(x)\cap(\g N(K).\g U^+(K)))=$
\par\noindent$\widehat P_K(x).(\widehat P_L(x)\cap(\g N(L).\g U^+(L))\cap(\g N(K).\g U^+(K)))=\widehat P_K(x).((\g N(L)(x).U^{pm+}_{Lx})\cap(\g N(K).\g U^+(K)))$
\par$=\widehat P_K(x).(\g N(L)(x)\cap \g N(K)).(U^{pm+}_{Lx}\cap\g U^+(K))=\widehat P_K(x).\g N(K)(x).U^{pm+}_{Kx}=\widehat P_K(x)$.
\par\noindent The same calculus gives $\g N(L).\widehat P_L(x)\cap\g G_\shs(K)=\g N(K).\widehat P_K(x)$.

\par  Hence there is a $\g G_\shs(K)-$equivariant embedding $\SHI(\g G_\shs,i,\A):\SHI(\g G_\shs,K,\A)\into\SHI(\g G_{\shs},L,\A)$; it sends each apartment onto an apartment.
But this embedding is not onto and the bijection between an apartment $A_K$ and its image $A_L$ is in general not an isomorphism:
if the extension $i$ is ramified, $\QL_L=\qo(L^*)$ is greater than $\QL=\qo(K^*)$, so there are more walls in $A_L$ than in $A_K$ and the enclosures or facets are smaller in $A_L$ than in $A_K$.

\par This embedding extends clearly to the bordered hovels.
 Hence $\SHI(\g G_\shs,K,\A)$ and  $\overline\SHI(\g G_\shs,K,\overline\A)$ are functorial in $(K,\qo)$.
  In particular a group $\QG$ of automorphisms of $K$ fixing $\qo$ acts on $\SHI(\g G_\shs,K,\A)$ and  $\overline\SHI(\g G_\shs,K,\overline\A)$.

\par Actually this possibility of embedding $\SHI(\g G_\shs,K,\A)$ or $\overline\SHI(\g G_\shs,K,\overline\A)$ in a (bordered) hovel where there are more walls or even where all points are special (if $\QL_L=\R$) is technically very interesting.
 It was axiomatized for abstract (bordered) hovels and used by Cyril Charignon: \cite{Cn-10}, \cite{Cn-10b}.

   \par{\bf 4) Changing the model apartment:} Let's consider an affine map $\psi:\A\to\A'$ between two affine apartments suitable for $G=\g G_\shs(K)$.
    We ask that $\psi$ is $N-$equivariant and $(^t\vect\psi)(\QD)=\QD$, this makes sense as $\QD\subset Q$ is in $(\vect\A)^*$ and $(\vect{\A'})^*$.
    So $\psi^{-1}(V_0')=V_0$ and the quotients $\A/V_0$, $\A'/V_0'$ are naturally equal to $\A^q$ (with the same walls).
     In particular there is a one to one correspondence between the enclosed filters in $\A$, $\A'$ or $\A^q$.

     \par For $y\in\A$, $N(y)\subset N(\psi(y))$, $\widehat P(y)=U_{\overline y}^{pm+}.U_{\overline y}^{nm-}.N(y)\subset\widehat P'(\psi(y))=U_{\overline y}^{pm+}.U_{\overline y}^{nm-}.N(\psi(y))$.
     We get a $G-$equivariant map $\SHI(\g G_\shs,K,\psi):\SHI(\g G_\shs,K,\A)\to\SHI(\g G_{\shs},K,\A')$.
      It induces a one to one correspondence between the apartments or facets, chimneys,... of both hovels but it is in general neither into nor onto. The most interesting example is the essentialization map $\SHI(\g G_\shs,K,\A)\to\SHI(\g G_{\shs},K,\A^q)$.

      \par Clearly these maps extend to the bordered hovels.

  \subsection{Uniqueness of the very good family of parahorics}\label{5.9}

  \par{\quad\bf1)} Actually, by \ref{5.3} the family $\widehat\shp$ satisfies the following strengthening of axiom (P8):

  \par (P8+) For all $x\in\overline{\A}$, for all chamber $C^v\in F^v(x)^*$,
 \par\qquad\qquad\qquad    $Q(x)=(Q(x)\cap U(C^v)).(P(x)\cap U(-C^v)).N(x)$


 \par\noindent By the following lemma, we know that $\widehat\shp$ is the only very good family of parahorics over $\overline\A$.

  \par{\bf 2)} At least for $\A=\A^q$, Charignon defines a maximal good family of parahorics $\overline\shp$:

  \par\qquad for $x\in\A_{F^v}$, $\overline P(x)=\{g\in P(F^v)\mid g.pr_{F^v_1}(x)=pr_{g.F^v_1}(x)\;\forall F_1^v\in \SHI^v_{sph},F^v\subset\overline{F^v_1}\}$

  \par\noindent where $pr_{F^v_1}$ is the projection associated to the minimal family $\shp$ (supposed good) or to any good family $\shq$ \eg $\widehat \shp$.

\par  We have $\shp\subset\widehat\shp\subset\overline\shp$ in the sense that $\forall x\in\overline\A$, $P(x)\subset\widehat P(x)\subset\overline P(x)$ \cite[sec. 11.8]{Cn-10b}.
 It is likely that $\widehat\shp=\overline\shp$, but it seems not to be a clear consequence of the preceding results.

\begin{lemm}\label{5.10} Let $\shq$ and $\shq'$ be two very good families of parahoric subgroups of $G$ (in the general setting of section \ref{s3}).
 Suppose that $\shq\leq{}\shq'$ (\ie $Q(x)\subset Q'(x)$ $\forall x\in\overline\A$) or that $\shq'$ satisfies (P8+), then $\shq=\shq'$.
\end{lemm}

\begin{proof} If $\shq\leq{}\shq'$ there is clearly a $G-$equivariant map $j:\overline\shi\to\overline\shi'$ between the bordered hovels associated to $G$, $\overline\A$ and $\shq$ or $\shq'$.
 This map sends each bordered apartment isomorphically to its image. Let $F^v$ be a vectorial facet in $\A^v$, then \ref{3.14} applies to the map $j$ between the ordered affine hovels $\shi_{F^v}$ and $\shi_{F^v}'$.
  So $j$ is one to one. Let $x\in\A_{F^v}$ and $g\in Q'(x)$, then $j(gx)=gj(x)=j(x)$, so $gx=x$ and $g\in Q(x)$.

  \par If $\shq'$ satisfies (P8+) we may apply the first case to $\shq$ or $\shq'$ and $\shq''=\shq\cap\shq'$ (\ie $Q''(x)=Q(x)\cap Q'(x)$ $\forall x$), as this family $\shq''$ is  very good.
   Actually for $\shq''$ (P1) to (P6) and (P9), (P10) are clear.
   For (P7) we have to prove $Q''(x)\cap NP(F^v)\subset NQ''(pr_{F^v}(x))=NP(pr_{F^v}(x))$ (as $F^v$ is spherical); it is clear.
   For (P8) $Q''(x)=Q(x)\cap[(Q'(x)\cap U(C^v)).(P(x)\cap U(-C^v)).N(x)]=(Q(x)\cap Q'(x)\cap U(C^v)).(P(x)\cap U(-C^v)).N(x)$ as $P(x).N(x)\subset Q(x)$.
\end{proof}

  \subsection{Residue buildings}\label{5.11}

  \par{\quad\bf1)} Let $x$ be a point in the apartment $\A$. We defined in \cite[4.5]{GR-08} or \cite[{\S{}} 5]{Ru-10} the twinned buildings $\SHI^+_x$ and $\SHI^-_x$, where $\SHI^+_x$ (resp. $\SHI^-_x$) is the set of segment germs $[x,y)$ for $y\in\SHI$, $y\not=x$ and $x\leq{}y$ (resp. $y\leq{}x$).
  Any apartment $A$ containing $x$ induces a twin apartment $A_x=A_x^+\cup A_x^-$ where $A_x^\pm{}=\{[x,y)\mid y\in A\}\cap\SHI_x^\pm{}$.
  As we want to consider thick buildings, we endow the apartments of $\SHI_x^\pm{}$ with their restricted structure of Coxeter complexes; on $\A_x$ it is associated with the subroot system $\QF_x=\{\qa\in\QF\mid-\qa(x)\in\QL_\qa\}$ of $\QF$ (\cf \cite[5.1]{By-96}) and the Coxeter subgroup $W_x^{min}\simeq W^v_x$ of $W^v$.
  One should note that $\QF_x$ is reduced but could perhaps have an infinite non free basis, corresponding to an infinite generating set of $W^v_x$.

  \par The group $G_x=\widehat P(x)$ contains three interesting subgroups: $P(x)=N(x).G(x)\supset P_x^{min}=Z_0.G(x)$ (see \cite[{\S{}} 3.2]{GR-08}, they are equal when $x$ is special); the group $G_{\SHI_x}$ is the pointwise fixer of all $[x,y)\in\SHI^{\pm}_x$ (\ie $g\in G_{\SHI_x}\iff\forall[x,y)\in\SHI^{\pm}_x, \exists z\in]x,y]$ such that $g$  fixes pointwise $[x,z]$), it is clearly normal in $G_x$.

  \par We write $\overline G_x=G_x/G_{\SHI_x}$ and $\overline U_\qa$ or $\overline R$ the images in $\overline G_x$ of $U_{\qa,-\qa(x)}$ ($\qa\in\QF$) or $R$ any subgroup of $G_x$.

  \begin{enonce*}[plain]{\quad2) Lemma} A $g\in G_x$ fixing an element in $\SHI^-_x$ and fixing pointwise $\SHI^+_x$ (\eg  $g\in G_{\SHI_x}$) fixes pointwise each $[x,y)$ for $y\not=x$ in a same apartment as $x$.
  \end{enonce*}
    \begin{proof} So $g$ fixes $[x,z]$ for some $z<x$. By \cite[5.12.4]{Ru-11}, $[x,z)$ and $[x,y)$ are in a same apartment $A$.
    By hypothesis $g$ fixes points $z_1,\cdots,z_n$ in $A$ such that each $z_i-x$ is in the open Tits cone $\sht^\circ\subset\vect A$, these vectors generate the vector space $\vect A$ and the interior of the convex hull of $\{x,z_1,\cdots,z_n\}$ contains an opposite of $[x,z)$.
    By moving each $z_i$ in $]x,z_i]$ one may suppose $x\leq{}z_1\leq{}\cdots\leq{} z_n$. Now as $z<x$, $g$ fixes (pointwise) the convex hull of $\{z,z_1,\cdots,z_n\}$ which is a neighbourhood of $x$ in $A$, hence contains $[x,y)$.
    \end{proof}

  \begin{enonce*}[plain]{\quad3) Lemma} Any $u\in U^+$ fixing (pointwise) a neighbourhood of $x$ in $\A$, fixes pointwise $\SHI_x$.
  This applies in particular to a $u\in U_{\qa,-\qa(x)+}$ for $\qa\in\QF_x$ or a $u\in U_{\qa,-\qa(x)}$ for $\qa\in\QF\setminus\QF_x$.
  \end{enonce*}
      \begin{proof} By \ref{5.8}.3 we may suppose $x$ special, hence $x=0$. By the above lemma it is sufficient to prove that $u$ fixes $\SHI^+_0$.
      An element $[0,y)$ of $\SHI^+_0$ is in an apartment $A$ containing the chamber  $F=F(0,C^v_+)$ and even the sector $\g q=0+C^v_+$ \cite[5.12.4]{Ru-11}; this apartment may  be written $A=g^{-1}\A$ with $g\in\widehat P(\g q)$ and even $g\in U_0^{pm+}$.
      Now $g([0,y))$ is in a sector $wC^v_+$ for some $w\in W^v$ and we have to prove that $gug^{-1}$ fixes a neighbourhood of $0$ in this sector.

      \par We argue in $U_0^{ma+}$ (as defined in \ref{5.1} or \cite[4.5.2]{Ru-11}).
      This group may be written as a direct product: $U_0^{ma+}=U_0^{ma}(\QD^+)=(\prod_{\qb\in\QD'}\,U_{\qb,0})\times U_0^{ma}(\QD^+\setminus \QD')$ where $\QD'$ is the finite set of positive roots of height $\leq{}N$ (with $N$ such that $\QD^+\cap w\QD^-\subset\QD'\cap\QF$) and $U_0^{ma}(\QD^+\setminus \QD')$ is a normal subgroup.
      Moreover each $U_{\qb,0}$ is a finite product of sets in bijection with $\sho$, the neutral element corresponding to $(0,\cdots,0)$ (actually for $\qb$ real, $U_{\qb,0}$ is isomorphic to the additive group of $\sho$).
      For $g_1\in U_0^{ma+}$ the map sending $v\in U_0^{ma+}$ to the component of $g_1vg_1^{-1}$ in $\prod_{\qb\in\QD'}\,U_{\qb,0}$, factors through $U_0^{ma+}/U_0^{ma}(\QD^+\setminus \QD')=\prod_{\qb\in\QD'}\,U_{\qb,0}$ and induces a polynomial map with coefficients in $\sho$ and without any constant term.

      \par Now $u\in U^+\cap G_0=U_0^{pm+}$ and $u$ fixes (pointwise) a neighbourhood of $x$ in $\A$, hence some $x'\in-C^v_+$. So $u\in U_{x'}^{ma+}$ and the component of $u$ in  $U_{\qb,0}$ is in the maximal ideal $\g m$ of $\sho$ if $\qb$ is real or in $\g m\times\cdots\times\g m$ if $\qb$ is imaginary.
      By the above property this is also true for $gug^{-1}$ and $gug^{-1}$ fixes a neighbourhood of $0$ in $wC^v_+$ (as $wC^v_+$ is fixed by $U_0^{ma}(\QD^+\setminus \QD')$).
    \end{proof}

  \begin{enonce*}[plain]{\quad4) Proposition} $(\overline P^{min}_x=\overline Z_0.\overline{G(x)},(\overline U_\qa)_{\qa\in\QF_x},\overline Z_0)$ is a generating root datum whose associated twin buildings have the same chamber sets or twin-apartment sets as $\SHI^{\pm}_x$.

 \par  Moreover $G_x=G(x).N(x).G_{\SHI_x}$ and $U_x^{pm+}\subset U_x^{++}.G_{\SHI_x}$, $U_x^{nm-}\subset U_x^{--}.G_{\SHI_x}$.
  \end{enonce*}

  \begin{remas*} a) As the basis of $\QF_x$ could be infinite the above generating root datum must be understood in a more general sense than in \ref{1.3}: we should consider the free covering $\widetilde\QF_x$ of $\QF_x$ (whose basis is free) which is in one to one correspondence with $\QF_x$ (\cf \cite[4.2.8]{By-96}) and a root datum as in \cite[6.2.5]{Ry-02a}.
  Another (less precise) possibility is to index the $\overline U_\qa$ by subsets of the Weyl group $W^v_x$, see \cite[1.5.1]{Ry-02a} or \cite[8.6.1]{AB-08}.
  Actually there is no trouble in defining the combinatorial twin buildings associated to this generalized root datum; but, except for chambers, their facets may not be in one to one correspondence with those of $\SHI^{\pm}_x$, \cf \cite[5.3.2]{Ru-10}.

  \par b) We may define the subgroup $P_x=P_x^{min}.G_{\SHI_x}$; this generalizes the definition given in \cite[5.14.2]{Ru-11}, as clearly $P_x\cap N=N_x^{min}$.
  It is the subgroup of $G_x$ which preserves the "restricted types in $x$" of the facets $F(x,F^v)$ (\ie their types as defined in the twin buildings $\SHI^{\pm}_x$ endowed with their restricted structures).
  The greater group $\widehat P_x^{sc}$ of \cite[5.14.1]{Ru-11} preserves the "unrestricted types" of the local facets $F^l(x,F^v)$ (\ie the (vectorial) type of $F^v$).

  \par c) These results and \cite[4.13.5]{Ru-11} suggest that $G_x=\widehat P(x)$ could perhaps be always equal to $P(x)=G(x).N(x)=P_x^{min}.N(x)$ for any $x\in\A$ (\ie $U_x^{pm+}=U_x^{+}$ and $U_x^{nm-}=U_x^{-}$). On the contrary we already said in \ref{3.2}.1 that $U_x^{pm+}$ or $U_x^{+}$ is in general different from $U_x^{++}$.
  \end{remas*}
      \begin{proof}
By definition for $\qa\in\QF_x$, $-\qa(x)\in\QL$ hence there is a $r\in K$ with $\qo(r)+\qa(x)=0$, so $\qf_\qa(x_\qa(r))+\qa(x)=0$ and the fixed point set of $u=x_\qa(r)\in U_{\qa,-\qa(x)}$ is $D(\qa,-\qa(x))$.
    Therefore the image of $u$ in $\overline U_\qa\subset\overline G_x$ is non trivial; (RD1) follows.

    \par (RD2) is a consequence of (RD2) and (V3) in $G$ as $\overline U_\qa$ is trivial when $\qa\not\in\QF_x$ by lemma 3. (RD3) is useless as $\QF$ is reduced.
    For (RD4) $\overline u\in \overline U_\qa\setminus\{1\}$ is the class of an element $u\in U_\qa$ with $\qf_\qa(u)=-\qa(x)$ (by lemma 3); hence the result follows from \ref{2.1}.2.

\par An element $\overline g\in \overline Z_0.\overline U^+\cap\overline U^-$ is the class of an element $g\in U^-$ and, up to $G_{\SHI_x}$, $g$ fixes $x+C^v_+$ and $x-C^v_+$, hence $g$ fixes a neighbourhood of $x$ in $\A$ (by convexity) and, by lemma 3, $g\in G_{\SHI_x}$. So $\overline g=1$ and (RD5) is proved.

\par The group $P_x^{min}$ is generated by $Z_0$ and the $U_{\qa,k}$ for $\qa\in\QF$ and $\qa(x)+k\geq0$. So the lemma 3 tells that its image $\overline P_x^{min}$ is generated by $\overline Z_0$ and the $\overline U_{\qa}$ for $\qa\in\QF_x$. Hence $(\overline P^{min}_x,(\overline U_\qa)_{\qa\in\QF_x},\overline Z_0)$ is a generating root datum.
We define $\overline U^\pm{}$ as the group generated by the $\overline U_{\pm{}\qa}$ for $\qa\in\QF_x^+$ (and not the image $\overline U^\pm_x$ of $U^\pm_x$).

\par Let $\shi^\pm_c$ be its associated (combinatorial) twin buildings and $C^\pm_c$ its fundamental chambers \cite[8.81]{AB-08}.
The twin apartments of $\shi^\pm_c$ or $\SHI^\pm_x$ are both the twin Coxeter complexes associated to $W^v_x$ and $N$ acts transitively on their four chamber sets.
Moreover the chambers in $\shi^+_c$ (resp. $\SHI^+_x$) sharing with $C^+_c$ (resp. $C=F(x,C^v_+)$) a panel of type $r_\qa\in W^v_x$ (for $\qa$ simple in $\QF_x$) are in one to one correspondence with $\overline U_\qa$ by \cite[8.56]{AB-08} (resp. $U_{\qa,-\qa(x)}/U_{\qa,-\qa(x)+}$ by \ref{3.11}e).
 By lemma 3 these two groups are isomorphic. So the chamber sets of $\shi^+_c$ and  $\SHI^+_x$ are in one to one correspondence. The same thing is true for the negative buildings.

 \par The twin apartments in $\shi^\pm_c$ are permuted transitively by $P^{min}_x$ and the stabilizer of the fundamental one is $N_x^{min}=P_x^{min}\cap N$. So the twin apartments in $\shi^\pm_c$ correspond bijectively to some apartments of $\SHI^\pm_x$.
 But two chambers in $\SHI^\pm_x$ correspond to chambers in $\shi^\pm_c$, hence are in a twin apartment of $\shi^\pm_c$ and their distance or codistance is the same in $\shi^\pm_c$ or $\SHI^\pm_x$.
 As a twin apartment is uniquely determined by a pair of opposite chambers, every twin apartment of $\SHI^\pm_x$ comes from a twin apartment in $\shi^\pm_c$.

 \par The chambers in $\shi^-_c$ opposite $C^+_c$ are transitively permuted by $\overline Z_0.\overline U^+$ \cite[6.87]{AB-08} hence in one to one correspondence with $\overline U^+$, as $\overline Z_0.\overline U^-\cap\overline U^+=\{1\}$ [\lc 8.76].
 In $\SHI^-_x$ the chambers opposite $C$ are in a same apartment as the sector $x+C^v_+$ \cite[5.12.4]{Ru-11} hence transitively permuted by $U_x^{pm+}$ (\ref{3.6}.4).
 Now the fixer in $U_x^{pm+}$ of the chamber $F(x,-C^v_+)$ is actually in $G_{\SHI_x}$ by lemma 3.
 So the chambers opposite $C$ in $\SHI^-_x$ are in one to one correspondence with the image $\overline U_x^{pm+}$ of $U_x^{pm+}$ in $\overline G_x$. Thus $\overline U_x^{pm+}=\overline U^+$.
 But $\overline U^+$ is the image of $U_x^{++}$, hence $U_x^{pm+}\subset U_x^{++}.G_{\SHI_x}$.
 As $G_x=\widehat P(x)=U_x^{pm+}.U_x^-.N(x)=U_x^{pm+}.G(x).N(x)$ we have $\overline G_x=\overline G(x).\overline N(x)$, hence $G_x=G(x).N(x).G_{\SHI_x}$.
    \end{proof}
    
\subsection{Iwahori-Matsumoto decompositions}\label{5.12}

\par Let $x$ be a special point in $\A$ and $\shv$ the filter of neighbourhoods of $x$ in $\A$.
 Then $P_x^{min}=Z_0.G(x)=N(x).G(x)=P(x)$ (as $W_x^{min}\simeq W_x^{v}=W^v$) and $G_x=\widehat P(x)=P(x).U_x^{nm-}=P(x).U_x^{pm+}$, equal moreover to $P_x^{pm}=P_x^{nm}=\widehat P_x$ with the notations of \cite[4.6, 4.13]{Ru-11}.
 
 \par We get the following Iwahori-Matsumoto decompositions (compare \ref{5.11}.4).
 
 \begin{prop*} \qquad $P(x)=U^+_\shv.U^-_x.N(x).U^-_x=\bigsqcup_{w\in W^v}\,U^+_\shv.U^-_x.wZ_0.U^-_x$ 
 \par and \qquad $G_x=\widehat P(x)=U^+_\shv.U^-_x.N(x).U^{nm-}_x=\bigsqcup_{w\in W^v}\,U^+_\shv.U^-_x.wZ_0.U^{nm-}_x$.
 \end{prop*}
 
 \begin{rema*} The equality of $P(x)$ and $U^{++}_\shv.U^{--}_x.N(x).U^{--}_x$, stated in a preliminary draft of \cite{BrKP12}, fails in general.
  With the notations of \cite[4.12.3]{Ru-11} with $x=0$, let's consider the following element $h$ of $SL_2(K[t,t^{-1}])$: 
  
   \par\noindent$\left(\begin{matrix}1&0  \cr 1&1\cr \end{matrix}\right)\left(\begin{matrix}1&\varpi t\cr 0&1\cr \end{matrix}\right)\left(\begin{matrix}1&0  \cr -1&1\cr \end{matrix}\right)=\left(\begin{matrix}1-\varpi t&\varpi t  \cr -\varpi t&1+\varpi t\cr \end{matrix}\right)=\left(\begin{matrix}1&1  \cr 0&1\cr \end{matrix}\right)\left(\begin{matrix}1&0  \cr -\varpi t&1\cr \end{matrix}\right)\left(\begin{matrix}1&-1  \cr 0&1\cr \end{matrix}\right)$
  
  \par As in \LC a) we see that $h$ is in $U^{++}_x\setminus U^{++}_\shv$ (right hand side).
  Hence $U^{++}_\shv$ (which contains $\left(\begin{matrix}1&0  \cr -\varpi t&1\cr \end{matrix}\right)$) is not normalized by $\left(\begin{matrix}1&1  \cr 0&1\cr \end{matrix}\right)\in U_{\qa_1,x}\subset U^{++}_x$.
  Moreover if $h\in U^{++}_\shv.U^{--}_x.wZ_0.U^{--}_x$ with $w\in W^v$, one has $w=1$ (as the image of $h$ in $SL_2(\qk[t,t^{-1}])$ is trivial).
  So $h\in U^{++}_\shv.U^{--}_x.Z_0\cap U^{++}_x$. But the decomposition of an element in $U^+U^-Z$ is unique (\ref{1.3c}.2), hence $h\in U^{++}_\shv$, nonsense.
 \end{rema*}
 \begin{proof} As suggested in \cite{BrKP12}, we get the decomposition of $P(x)$ as in the article by  Iwahori and Matsumoto \cite[prop. 2.4]{IM-65}. \label{N32}
 The main ingredients are the decompositions $U^+_\shv=U_{\qa_i,\shv}\ltimes U_\shv(\QD^+\setminus\{\qa_i\})$ and $U^-_x=U_{-\qa_i,x}\ltimes U_x(\QD^-\setminus\{-\qa_i\})$ where $\qa_i$ is a simple root and $U_\shv(\QD^+\setminus\{\qa_i\})$, $U_x(\QD^-\setminus\{-\qa_i\})$ are defined as in \cite[proof of 3.4d]{GR-08} and normalized by  $U_{\pm\qa_i,x}$.
 The decomposition of $\widehat P(x)$ is a simple consequence.
 We have now to prove that two sets $U^+_\shv.U^-_x.wZ_0.U^{nm-}_x$ for two different $w\in W^v$ have no intersection.
  But $U^+_\shv.U^-_x$ and $Z_0.U^{nm-}_x$ fix the local chamber $F^l(x,-C^v_+)$; so this is a consequence of the uniqueness in the Bruhat decomposition of $\overline G_x=G_x/G_{\SHI_x}$, \cf \ref{5.11}.4.
 \end{proof}
    
\section{Hovels and bordered hovels for almost split Kac-Moody groups}\label{s6}

\subsection{Situation and goal}\label{6.1}

We consider an almost split Kac-Moody group $\g G$ over a field $K$, as in section \ref{s1b}.
 We suppose the field $K$ endowed with a non trivial real valuation $\qo=\qo_K$ which may be extended functorially and uniquely to all extensions in $\shs ep(K)$.
 This condition is satisfied 
  if $K$ is complete for $\qo_K$ (or more generally if $(K,\qo_K)$ is henselian).

 \par We built in \ref{5.4} a bordered hovel $\overline\SHI(\g G,L,\overline\A)$ for any $L\in\shs ep(K)$ splitting $\g G$.
  We want a bordered hovel $\overline\SHI(\g G,K,{_K\overline\A})$ on which $G=\g G(K)$ would act strongly transitively and $G-$equivariant embeddings $\overline\SHI(\g G,K,{_K\overline\A})\into\overline\SHI(\g G,L,\overline\A)$ for $L\in\shs ep(K)$ splitting $\g G$.

  \par An idea (already used in the classical case \cite{BtT-84a}) is to suppose $L/K$ finite Galois, to build an action of the Galois group $\QG=Gal(L/K)$ over $\overline\SHI(\g G,L,\overline\A)$ and to find $\overline\SHI(\g G,K,{^K\overline\A})$ in the fixed point set $\overline\SHI(\g G,L,\overline\A)^\QG$.
   As already known in the classical case \cite{Ru-77}, the equality of these last two objects is in general impossible.

\subsection{Action of the Galois group on the bordered hovel}\label{6.2}

\par{\quad\bf1)} We consider a finite Galois extension $L$ of $K$ which splits $\g G$.
 The Galois group $\QG=Gal(L/K)$ acts on $\SHI^v(L)=\SHI^v(\g G,L,\A^v)$ and the action of $\g G(L)$ on $\SHI^v(L)$ is $\QG-$equivariant.
 More precisely we suppose $L$ such that there exists a maximal $K-$split torus $\g S$ in $\g G$ contained in a maximal torus $\g T$ defined over $K$ and split over $L$ (\cf \ref{1.8}.4 and \ref{1.9}).
 We described in section \ref{s1b} the fixed point set $\SHI^v(L)^\QG$, its $K-$apartments and $K-$facets.
 In particular the apartment $A^v=A^v(\g T)$ corresponding to $\g T$ is stable under $\QG$ and $(A^v)^\QG$ is the $K-$apartment $_KA^v(\g S)$ corresponding to $\g S$.

 \par We want an action of $\QG$ on the bordered hovel $\overline\SHI(L)=\overline\SHI(\g G,L,\overline\A)$ compatible with its action on $\SHI^v(L)$ and the action of $\g G(L)$.
  Hence $\QG$ must permute the apartments and fa\c{c}ades of $\overline\SHI(L)$ as the apartments and facets of $\SHI^v(L)$. In particular $\QG$ has to stabilize the bordered apartment $\overline A=\overline A(\g T)$ corresponding to $\A^v$ \ie to $\g T$.

  \par{\bf 2) Action on $\overline A$}: $\qg\in\QG$ must act affinely on $\overline A$ with associated linear action the action of $\qg$ on $V=\vect{A^v}=\vect{A}$.
  Moreover this action has to be compatible with the action on the root groups ($\forall\qa\in\QF$ $\qg(U_{\qa,\ql})=U_{\qg\qa,\ql'}\Rightarrow\qg(D(\qa,\ql))=D(\qg\qa,\ql')$ at least when $\ql\in\QL_\qa$) and we know that the action of $\QG$ on $\g G(L)$ is compatible with its action on its Lie algebra ($\qg(exp(ke_\qa))=exp(\qg(k)\qg(e_\qa))$ ).
   Using these results and conditions, C. Charignon succeeds in finding a (unique) good action of $\QG$ on the essentialization $A^e=A/V_0$ of $A$; in particular the action of $N$ is $\QG-$equivariant \cite[13.2]{Cn-10b}.
   As $\QG$ is finite and acts affinely, it has a fixed point $x_0+V_0$ in $A^e$.

   \par Now $\QG$ has to fix a point in $x_0+V_0$.
    But all points in $x_0+V_0$ play the same r\^ole with respect to the conditions; so we may choose a point in $x_0+V_0$, \eg $x_0$, and say that $\QG$ fixes $x_0$ \ie that $\QG$ acts on $A$ as on $\vect{A^v}$ (after choosing $x_0$ as origin).
     This action is compatible with the above action on $A^e$. It permutes the walls, facets, ... and extends clearly to $\overline{A}$ ( $=\overline{A}^e$,  $\overline{A}^i$ or  $\overline{\overline A}$).

     \par{\bf N.B.} We had to make a choice to define this action. This is not a surprise: as in the classical case, $V_0$ is a group of $G-$equivariant automorphisms of $\overline\SHI(L)$.

     \begin{enonce*}[plain]{\quad3) Lemma} This action of $\QG$ on $\overline A$ stabilizes $\widehat\shp$: \quad$\qg(\widehat P(x))=\widehat\shp(\qg x)$, $\forall x\in\overline A$, $\forall\qg\in\QG$.
\end{enonce*}

\begin{proof} By Charignon's work ( 2) above) we know that $\QG$ stabilizes $\shp$. \label{N18}
 Hence $\qg\in\QG$ sends $\widehat\shp$ to a family $\shq$ which is still a very good family of parahorics. So \ref{5.9}.1 tells that $\shq=\widehat\shp$.
  Note that a longer proof may also be given using the star actions instead of \ref{5.9}.1.
\end{proof}

\par{\bf4)} We have got compatible actions of $\QG$ on $G=\g G(L)$ and $\overline A$ satisfying the above lemma.
 As $\overline\SHI(L)=\overline\SHI(\g G,L,\overline\A)$ is defined by the formula in \ref{3.1}, we obtain an action of $\QG$ on this bordered hovel, for which the $G-$action is $\QG-$equivariant.
  Each $\qg\in\QG$ acts as an automorphism: it induces a permutation of the apartments, facets, walls, fa\c{c}ades, chimneys, ... and the bijection between an apartment and its image is an affine automorphism.

  \par This action of  $\QG$ on $\overline\SHI(L)$ is compatible with its action on $\SHI^v(L)$ ($\qg(\SHI_{F^v})=\SHI_{\qg(F^v)}$ ) and on the sector faces ($\qg(x+F^v)=\qg(x)+\qg(F^v)$ ) or the chimneys.
   Moreover the projections on the fa\c{c}ades are $\QG-$equivariant ($\qg\circ pr_{F^v}=pr_{\qg(F^v)}\circ\qg$).
   These results are first proved (easily) for the actions on $A$ and $A^v$, and then extended (easily) to $\overline\SHI$.

   \par As $\QG$ has fixed points in $\SHI$, any $\QG-$fixed point in a fa\c{c}ade $\SHI_{F^v}\subset\overline\SHI$ is associated to a $\QG-$stable sector face $x+F^v$ in $\SHI$.

   \subsection{The descent problem}\label{6.3}

   \par In $\overline\SHI$ we have got an apartment $\overline A$ stable under $\QG$. But $\QG$ is finite and acts affinely, so it has a fixed point in $A$ and $A^\QG$ is an affine space directed by $(\vect{A^v})^\QG$.
   It seems interesting to choose $\overline A^\QG$ as affine bordered $K-$apartment and define $_K{\overline\SHI}=\g G(K).\overline A^\QG$.
   Unfortunately we are not sure then that $\overline A^\QG$ is stable under $_KN$ or fixed by $_KZ$; so this $_K{\overline\SHI}$ is not a good candidate for a bordered hovel associated to the root datum $(\g G(K),(V_{_K\qa})_{{_K\qa}\in{_K\QF}},{_KZ})$.

   \par It is possible to find in $\overline\SHI^\QG$ a subspace of apartment $_KA_1$ directed by  $(\vect{A^v})^\QG$ and stable under $_KN$.
    But then it is not clear that there exists an apartment $\overline A_2$ in $\overline\SHI$ containing $_K\overline A_1$ and stable under $\QG$, or even such that $\overline A_2\cap \overline\SHI^\QG={_K\overline A_1}$ \cite[13.3]{Cn-10b}.

    \par This problem is the same as in the classical case of reductive groups: \cite{BtT-72}, \cite{BtT-84a}, \cite{Ru-77}. Charignon solves it the same way: under some hypothesis on $\g G$ or $K$ and by a two steps descent.

   \subsection{The  descent theorem of Charignon}\label{6.4}

   This abstract result is largely inspired by the descent theorem of Bruhat and Tits in the classical case \cite[9.2.10]{BtT-72}. We explain here the hypotheses and conclusions of \cite[{\S} 12]{Cn-10b}, but, to simplify, we consider a more concrete framework.
   We keep the notations of \LC or we indicate them in brackets when they are too far from ours.  We keep our idea to replace many Charignon's overrightarrows by an exponent $^v$ and to use often an overrightarrow to indicate the generated vector spaces.

   \par{\bf1) Vectorial data}: We consider a finite Galois extension $L/K$ which splits $\g G$ as in \ref{6.2}.1.
   So there exists in $\g G$ a maximal $K-$split torus $_K\g S$ and a maximal torus $\g T$ split over $L$ containing $_K\g S$ (we don't ask $\g T$ to be defined over $K$).

   \par We consider the fixed point set $\vect\shi_\natural={_K\SHI}^{vq}=(\SHI^{vq})^\QG$ of $\QG=Gal(L/K)$ in $\vect\shi=\SHI^{vq}$.
   The group $\g G(K)$ ($=G^\natural$) acts on $\vect\shi_\natural$. By \ref{1.11}, \ref{1.12} $\vect\shi_\natural$ is a good geometrical representation of the combinatorial twin building $^K\!\!\SHI^{vc}=\SHI^{vc}(\g G,K)$.

   To $_K\g S$ and $\g T$ correspond apartments $A^{v}_\natural={_KA^{vq}}({_K\g S})\subset \vect\shi_\natural$ included in $A^{vq}=A^{vq}(\g T)\subset\vect\shi$; 
   they are cones in the vector spaces $\vect{A^{v}_\natural}$ ($=\vect V^\natural$) included in $\vect{A^{vq}}$ ($=\vect V$).
   The real root system $\QF$ (resp. the real relative root system $_K\QF=\QF^\natural$) is included in the dual $(\vect{A^{vq}})^*$ (resp. $(\vect{A^{v}_\natural})^*$) and has a free basis.
   Its associated vectorial Weyl group is $W^v=N/T$ ($=W(\QF)$) (resp. $_KW^v={_KN}/{_KZ}=W(\QF^\natural)$).
   Here $_KZ=T^\natural$ or $_KN=N^\natural$ is the generic centralizer or normalizer in $G^\natural$ of $_K\g S$. We write $\vect{A^{v}_{\natural0}}=\cap_{a\in\QF^\natural}\,Ker(a)$.

   \par We consider also the Weyl$-K-$apartment $A^{v\natural}={^K\!A^{vq}}({_K\g S})$ with $A^v_\natural\subset A^{v\natural}\subset \vect{A^v_\natural}$ and the corresponding building $\vect\shi^\natural=G^\natural.A^{v\natural}$ (\cf \ref{1.12}).
    As in \cite{Cn-10b} we define the facets in $A^v_\natural$ or $\vect\shi_\natural$ as the traces $F^v_\natural=F^{v\natural}\cap\vect\shi_\natural$ of the Weyl$-K-$facets $F^{v\natural}$ of $A^{v\natural}$ or $\vect\shi^\natural$.
     The same set $A^v_\natural$, endowed as facets with the non empty traces $F^v_\#=F^{vq}\cap A^v_\natural$ for $F^{vq}$ a facet in $A^{vq}$, will be written $A^v_\#$.
     There is a one to one correspondence between facets of $A^{v\natural}$ and $A^v_\natural$.
     But a facet $F^{v\natural}$ or $F^v_\natural=F^{v\natural}\cap A^v_\natural$ contains several facets in $A^v_\#$; among them one $F^{v\natural}_\#$ is maximal, open in $F^{v\natural}$, generates the same vector space  and $F^{v\natural}_\#+\vect{A^v_{\natural0}}=F^v_\natural+\vect{A^v_{\natural0}}=F^{v\natural}$  (\cf \ref{1.12}).

     \par The combinatorial twin building $^K\!\SHI_\pm^{vc}$ is associated to the root datum $(\g G(K),(V_{_K\qa})_{{_K\qa}\in{_K\QF}},$\goodbreak\noindent${_KZ})$ ($=(G^\natural,(U^\natural_{a})_{{a}\in{\QF^\natural}},{_KZ})$).
     Everything associated as in {\S} \ref{s1} to this root datum will be written with an exponent $^\natural$ or a subscript $_K$.
      The reader will check easily the conditions (DSR), (DDR1),...,(DDR3.2) and (DIV) of [\lc 12.1], \cf [\lc 13.4.1].
      In particular for $a={_K\qa}\in\QF^\natural$, $U^\natural_{a}=V_{_K\qa}$ is included in the group $U_a$ generated by the groups $U_\qb$ for the roots $\qb$ in the finite set $\QF_a=\{\qb\in\QF\mid\qb\,\rule[-1.5mm]{.1mm}{3mm}_{\,\vect V^\natural}\in\R^{+*}a\}=\{\qb\in\QF\mid{_K\qb}={_K\qa}$ or $(\frac{1}{2}).{_K\qa}$ or $2.{_K\qa}\}$.
      Actually $U_a=\prod_{\qb\in\QF_a}\,U_\qb$ for any order \cite[6.2.5]{Ry-02a}.

   \par{\bf2) Affine data}: We consider the essential bordered hovel $\overline\SHI=\overline\SHI(\g G,L,\overline\A^e)$ ($=\shi$) and $\overline A$ ($=A(T)$) the bordered apartment associated to $A^{vq}$ whose main fa\c{c}ade $A$ ($=A^\circ(T)$) is an affine space under $\vect{A^{vq}}=\vect V$.
   The fa\c{c}ades of $\overline\SHI$ are indexed by the facets $F^v\in\SHI^{vc}$.

   \par We consider moreover a subset $\overline\shi_\#$ in $\overline\SHI$, we write $A^\#=A\cap\overline\shi_\#$ and suppose:
 \medskip
   \par (DM1) $\overline\shi_\#$ is $G^\natural-$stable and, $\forall F^v\in\SHI^{vc}_{sph}$, $\overline\shi_\#\cap\SHI_{F^v}$ is convex in $\SHI_{F^v}$.

   \par(DM2) $A^\#$ is affine in $A$, directed by $\vect V^\natural$ and $\overline A\cap\overline\shi_\#$ is the closure $\overline A^\#$ of $A^\#$ in $\overline A$.

   \par(DM3) $\forall F^v\in\SHI^{vc}_{sph}$, if $F^v\cap A^{v\natural}\not=\emptyset$, there exists a facet $F$ in the (classical) apartment $A_{F^v}$ with $F\cap \overline\shi_\#\not=\emptyset$ and $F$ is equal to any facet $F'$ in $\SHI_{F^v}$ with $F'\cap \overline\shi_\#\not=\emptyset$ and $F\subset\overline{F'}$.

   \par(DM4) $\overline A^\#$ is stable under $_KN=N^\natural$.

\par\noindent Axiom (DM3) means essentially that, in appropriate spherical fa\c{c}ades, $\overline A\cap\overline\shi_\#$ cannot be enlarged by modifying the apartment $\overline A$.

\medskip
\par For $a\in\QF^\natural$ and $u\in U_a^\natural$, \label{N28} one defines $\qf_a^\natural(u)$ as the supremum in $\R\cup\{+\infty\}$ of the $k$ such that $u$ is in the group $U_{a,k}$ generated by the $U_{\qa,r_\qa k}=\qf_\qa^{-1}([r_\qa k,+\infty])$ for $\qa\in\QF_a$, $r_\qa\in\R^{+*}$ and $\qa\,\rule[-1.5mm]{.1mm}{3mm}_{\,\vect V^\natural}=r_\qa a$.
Actually $U_{a,k}=\prod_{\qa\in\QF_a}\,U_{\qa,r_\qa k}$ and $U_{a,k}^\natural:=(\qf_a^\natural)^{-1}([k,+\infty])=U_a^\natural\cap U_{a,k}$.

\par There are two more axioms, one normalizing $\qf$ (among equipollent valuations, in such a way that the associated origin $0_\qf$ of $A$ is in $A_\#$) and one avoiding triviality for each $\qf_a^\natural$.
 They are easily verified in our situation [\lc 13.4.1].

 \par As we have three types of vectorial facets in $\vect V^\natural=\vect{A^v_\natural}$, we may define three bordered apartments with $A^\#$:
 $\overline A_\natural$ (resp. $\overline A^\natural$, $\overline A^\#$) is the disjoint union of the fa\c{c}ades $A^\#_{F^{v}_\natural}=A^\#/\vect{F^{v}_\natural}$ (resp. $A^\#_{F^{v\natural}}=A^\#/\vect{F^{v\natural}}$, $A^\#_{F^{v}_\#}=A^\#/\vect{F^{v}_\#}$), for ${F^{v}_\natural}$ (resp. ${F^{v\natural}}$, ${F^{v}_\#}$) a facet in ${A^{v}_\natural}$ (resp. ${A^{v\natural}}$, ${A^{v}_\#}$).
  Actually $\overline A^\#$ is the closure of $A^\#$ in $\overline A$ as in (DM2) above.
  Moreover the sets $\overline A_\natural$ and $\overline A^\natural$ are equal (as $\vect{F^{v}_\natural}=\vect{F^{v\natural}}$ when $F^{v}_\natural=F^{v\natural}\cap A^v_\natural$) but they differ by their facets, sectors, ...

\par  When $F^{v\natural}\supset F^{v\natural}_\#=F^v\cap A^v_\natural$ for $F^v$ a (maximal) facet in $A^v$, we have $\vect{F^{v\natural}}=\vect{F^v}$, so $A^\#_{F^{v\natural}}\subset A_{F^v}\subset\overline A$.
 Hence for $x\in A^\#_{F^{v\natural}}\subset\overline A^\natural$, we may define:\qquad $Q^\natural(x)=\widehat P(x)\cap G^\natural$ .

 \par This is the same definition as in [\lc 12.4] as, for us, $F^{v\natural}_\#$ is uniquely determined by $F^{v\natural}$.

\begin{enonce*}[plain]{\quad3) Theorem} We suppose satisfied all conditions or axioms in 1) or 2) above, then:

\par a) $_KN=N^\natural\subset N.\widehat P(\overline A^\#)$.

\par b) $\forall a\in\QF^\natural$, $\forall u\in U_a^\natural\setminus\{1\}$ the fixed point set of $u$ in $A^\#$ is $D_\#(a,\qf_a^\natural(u)):=\{x\in A^\#\mid a(x-0_\qf)+\qf_a^\natural(u)\geq0\}$ and $m^\natural(u)\in N^\natural$ induces on $A^\#$ the reflection with respect to the wall $M_\#(a,\qf_a^\natural(u))=\partial D_\#(a,\qf_a^\natural(u))$.

\par c) The family $\qf^\natural$ is a valuation for the root datum $(G^\natural,(U^\natural_{a})_{{a}\in{\QF^\natural}},{_KZ})$.

\par d) The family $\shq^\natural=(Q^\natural(x))_{x\in \overline A^\natural }$ is a very good family of parahorics.

\par e) There is an injection of the essential bordered hovel $\overline\shi^\natural$  associated to $\shq^\natural$ into $\overline\SHI$ which may be described on the fa\c{c}ades as follows:

\par For $F^{v\natural}_\#=F^v\cap A^v_\natural$ open in $F^{v\natural}$ as above in 2), the  $_KN-$equivariant embedding $A^\#_{F^{v\natural}}\into A_{F^v}$ between apartment-fa\c{c}ades may be extended uniquely in a  $P_K(F^{v\natural})-$equivariant embedding $\shi^\natural_{F^{v\natural}}\into \SHI_{F^v}$, where $\shi^\natural_{F^{v\natural}}$ is the fa\c{c}ade of $\overline\shi^\natural$  associated to $F^{v\natural}$.
\end{enonce*}

\begin{rema*} The definition of $\qf_a^\natural$ tells us that a wall $M_\natural(a,\qf_a^\natural(u))$ is the trace on $A^\#$ of a wall $M(\qa,k)$ for some $\qa\in\QF_a$.
\end{rema*}

\begin{proof} a), b) c) and a great part of d), e) are among the main results of Charignon [\lc 12.3, 12.4].
 For $\shq^\natural$ he proves (P1) to (P7), but then (P8) is got for free in this framework (\cf \ref{3.3}.6) and (P10) is clearly satisfied.

 \par He proves (P9) actually for $\overline A_\natural$ \ie for (spherical) vectorial facets in $A^v_\natural$:
   if $F^{v\natural}$ is spherical, $F^{v\natural}_1\subset\overline{F^{v\natural}}$ and $x\in A^\#_{F^{v\natural}_1}=A^\#_{F^{v}_{1\natural}}$ (with $F^{v}_{1\natural}=F^{v\natural}_1\cap A^v_\natural$ and $F^{v}_{\natural}=F^{v\natural}\cap A^v_\natural$) he proves only $Q^\natural(x)\cap{P_K}(F^{v}_{\natural})=Q^\natural(\overline{x+F^{v}_{\natural}})$.
   But ${P_K}(F^{v}_{\natural})={P_K}(F^{v\natural}_{})$, $F^{v}_{\natural}+\vect{A^v_{\natural0}}=F^{v\natural}$ (\ref{1.12}.3) and the "torus" $S_Z$ in the center of $G^\natural$ (\ref{1.13}.2) acts on $A^\#$ as a group (of translations) $\vect T$ generating $\vect{A^v_{\natural0}}$.
    So $Q^\natural(x)\cap{P_K}(F^{v\natural})=Q^\natural(\cup_{\qt\in\vect T}\,\overline{x+\qt+F^{v}_{\natural}})=Q^\natural(\overline{x+F^{v\natural}})$.

    \par The maps  in e) between fa\c{c}ades are described in [\lc 12.5] and proved to be injective in the spherical case; but \ref{3.14} gives the general injectivity.
\end{proof}

\subsection{Tamely quasi-splittable descent}\label{6.5}

\par{\quad\bf1)} Let $\g S$ be a maximal $K-$split torus in the almost split Kac-Moody group $\g G$ over $K$. The generic centralizer $\g Z_g(\g S)$ of $\g S$ in $\g G$ (\ref{1.4d} and \ref{1.9}.3) is actually a reductive group defined over $K$ \cite[12.5.2]{Ry-02a}.
 We suppose satisfied the following condition (independent of the choice of $\g S$, as different choices are conjugated by \ref{1.9}.1).
 \medskip
 \par (TRQS) $\g Z_g(\g S)$ becomes quasi-split over a finite tamely ramified Galois extension $M$ of $K$.
 \smallskip
 \par\noindent(Actually $\g Z_g(\g S)$ is quasi-split over $M$ if and only if $\g G$ itself is quasi-split. It is an easy consequence of \ref{1.11} NB 2) applied to $M$ and a maximal $M-$split torus containing $\g S$.)
\smallskip
\par There are two important cases where this condition is satisfied for any $\g G$: when the field $K$ is complete (or henselian) for a discrete valuation with perfect residue field (we then may replace tamely ramified by unramified, \cf \cite[5.1.1]{BtT-84a} or \cite[5.1.3]{Ru-77}) or when the residue field of $K$ has characteristic $0$ (we then may replace quasi-split by split).

\par A consequence of this hypothesis is that there exists a finite Galois extension $L$ of $K$ containing $M$, a maximal $K-$split torus $_K\g S$, a maximal $M-$split torus $_M\g S$ and a maximal torus $\g T$ with $\g T$ $L-$split, $M-$defined and $_K\g S\subset{ _M\g S}\subset\g T$ \cite[13.4.2]{Cn-10b}.
 We shall now apply the abstract descent theorem successively to $L/M$ and $L/K$ to build a bordered hovel for $\g G$ over $K$.

 \par{\bf2) Quasi-split descent}: We consider the extension $L/M$, so we apply \ref{6.4} with $K=M$: $\g G$ is quasi-split over $K$ and split over $L$. Then $\g T=\g Z_g({_K\g S})$ is the only maximal torus containing ${_K\g S}$.

 \par We choose the essential bordered hovel $\overline\SHI=\overline\SHI(\g G,L,\overline\A^e)$ and set $\overline\shi_\#=\overline\SHI^\QG$.
  Then the bordered apartment $\overline A=\overline A^e(\g T)$ is $\QG-$stable. The Galois group $\QG$ has a fixed point in its main fa\c{c}ade $A=A^q(\g T)$ and $A^\#=A^\QG=A\cap\overline\shi_\#$ is an affine subspace directed by $\vect V^\QG=\vect V^\natural$.
  It is easy to verify (DM1), (DM2) and also (DM4) (as $_KN$ is the normalizer in $\g G(K)$ of $_K\g S$).
  For (DM3) there exists a chamber $F$ in $A_{F^v}$ meeting $\overline\shi_\#$, so the condition is clearly satisfied.

  \par Therefore theorem \ref{6.4}.3 applies. Actually in the classical case ($\QF$ finite) $\g G(K).A^\#$ is the extended Bruhat-Tits building of $\g G$ over $M$ \cf \cite{BtT-84a} or \cite{Ru-77}.

 \par{\bf3) General descent}: We come back to the situation and notations in 1) above. We still choose the essential bordered hovel $\overline\SHI=\overline\SHI(\g G,L,\overline\A^e)$ with $\A=\A^q$.

 \par The generic centralizer $\g Z_g(\g S)$ of $\g S={_K\g S}$ is a $K-$defined reductive group generated over $L$ by $T=\g T(L)$ and the groups $U_\qa$ for $\qa\in\QF$, $\qa\,\rule[-1.5mm]{.1mm}{3mm}_{\,\g S}$ trivial.
 In particular over $L$, $\g Z_g(\g S)$ is isomorphic to some $\g G_{\shs(I_0)}$ and by \ref{5.8}.2 $\SHI(\g Z_g(\g S),L,\A^q)$ may be embedded in $\SHI$.
  The image is the union $\SHI_{\g S}=\SHI(\g Z_g(\g S),\g G,L,\A^q)$ of the apartments of $\SHI$ corresponding to $L-$split maximal tori of $\g G$ containing $\g S$.
  This set is stable by $\QG$ and $Z_L(\g S)=\g Z_g(\g S)(L)$ or the normalizer $N_L(\g S)$ of $\g S$ in $\g G(L)$.
  If we choose a vectorial $K-$chamber $_KC_0^{vq}\subset A^v_\natural$ and let $F_0^{vq}\in\SHI^{vc}(L)$ be the spherical vectorial facet containing $_KC_0^{vq}$, the projection map $\qp$ from $\SHI_{\g S}$ to $\SHI_{F^{vq}_0}$ is onto and $\QG\ltimes N_L(\g S)-$equivariant; it identifies the essentialization of $\SHI_{\g S}$ with $\SHI_{F^{vq}_0}$.

  \par In $\SHI_{\g S}$ we consider the union $\SHI_{\g S}^{ord}$ of the apartments corresponding to a torus containing a maximal $M-$split torus $_M\g S$ (containing $\g S$).
   It  is stable by $\QG\ltimes N_L(\g S)$ and we saw in 2) above that $\SHJ_{\g S}=Z_M(\g S).A^\#_M=(\SHI_{\g S}^{ord})^{Gal(L/M)}$ is a good candidate for the hovel of $\g Z_g(\g S)$ over $M$.
   More precisely its image $\SHJ_{F^{vq}_0}=\qp(\SHJ_{\g S})$ in $\SHI_{F^{vq}_0}$ is the Bruhat-Tits building of $\g Z_g(\g S)$ over $M$: it is the set of ordinary $Gal(L/M)-$invariant points in the Bruhat-Tits building over $L$ \cite[2.5.8c]{Ru-77}.

\par We consider now $A^\#=(\SHJ_{\g S})^{Gal(M/K)}=(\SHI_{\g S}^{ord})^\QG$; its image by $\qp$ is in $(\SHJ_{F^{vq}_0})^{Gal(M/K)}$.
 But the semi-simple quotient of $\g Z_g(\g S)$ is $K-$anisotropic and $M/K$ is tamely ramified, so we know that $(\SHJ_{F^{vq}_0})^{Gal(M/K)}$ contains at most one point \cite[5.2.1]{Ru-77}.
  Moreover Koen Struyve \cite{Se-11} proved what was missing in \cite{Ru-77} (condition (DE) of \cite[5.1.5]{BtT-84a}):  this set is non empty (even if the valuation is not discrete).
  So $(\SHJ_{F^{vq}_0})^{Gal(M/K)}$ is reduced to one point $x_0$ and $A^\#=\qp^{-1}(x_0)^\QG$.
  But $\qp^{-1}(x_0)$ is an affine space directed by $\vect{F_0^{vq}}$, $\QG$ is finite and acts affinely, so $A^\#$ is a (non empty) affine space directed by $(\vect{F_0^{vq}})^\QG=\vect{_KC_0^{vq}}=\vect{A^v_\natural}=\vect V^\natural$.
   We shall apply \ref{6.4} with $A^\#$, $A$ any apartment of $\SHI_{\g S}^{ord}$ containing $A^\#$ and $\overline A$ its closure in $\overline\SHI=\overline\SHI(\g G,L,\overline\A^e)$.

   \par We define $\SHJ=\g G(M).\SHJ_{\g S}=\g G(M).A^\#_M$ (resp. its closure $\overline\SHJ=\g G(M).\overline{A^\#_M}$); it is the set of $Gal(L/M)-$fixed points in the union of the apartments in $\SHI$ (resp. $\overline\SHI$) corresponding to a maximal torus containing a maximal $M-$split torus, itself containing a maximal $K-$split torus.
   We take $\overline\shi_\#=\overline\SHJ^{Gal(M/K)}=\overline\SHJ^\QG$.
   The verification of axioms (DM1) to (DM4) is made in \cite[13.4.4]{Cn-10b}. Actually (DM4) is clear, (DM2) not too difficult and (DM1), (DM3) have to be verified in spherical fa\c{c}ades, hence are corollaries of the classical Bruhat-Tits theory.

   \par{\bf4) Conclusion}: We keep the notations as in 1); let $_KA^{vq}$ be the $K-$apartment in $_K\SHI^{vq}(\g G)$ and $_K\QF$ the real root system  associated to $\g S$.
   Then theorem \ref{6.4}.3 gives us a valuation $_K\qf=\qf^\natural=(_K\qf_{_K\qa})_{_K\qa\in{_K\QF}}$ of the root datum $(\g G(K),(V_{_K\qa})_{{_K\qa}\in{_K\QF}},{_KZ})$ (\cf \ref{1.11}) and a very good family of parahorics $({\widehat P}_K(x))_{x\in\overline{A^\natural}}$.
    The corresponding bordered hovel is written $\overline\SHI(\g G,K,\overline{A^\natural})$.

 \par   For $^K\!F^{vq}=F^{v\natural}$ a vectorial Weyl$-K-$facet and $F^{vq}$ a vectorial facet with $F^{vq}\cap{_KA^{vq}}$ open in $^K\!F^{vq}$, we have a $P_K({^K\!F^{vq}})-$equivariant embedding $\SHI(\g G,K,\overline{A}^\natural)_{^K\!F^{vq}}\into\SHI(\g G,L,\overline\A^e)_{F^{vq}}$ between the fa\c{c}ades. The image $\g G(K).A^\#_{F^{v\natural}}$ is pointwise fixed by $\QG$.

 \par  Actually the set $\overline A^\natural$ is the essential bordered apartment associated to $A^\#$ and $_K\QF$, its fa\c{c}ades are the $A^\#_{F^{v\natural}}$ for $F^{v\natural}$ as above.
 Such a fa\c{c}ade $A^\#_{F^{v\natural}}$ may be identified with the closure of $A^\#$ in $A^q(\g T)_{F^{vq}}$. Moreover  $A^\#$ is the set of $Gal(L/K)-$fixed points
 in the union of the apartments $A^q(\g T)\subset\SHI(\g G,L,\A^q)$ for $\g T$ a $L-$split maximal torus containing a maximal $M-$split torus, itself containing the maximal $K-$split torus $_K\g S$.
 More precisely for each such apartment $A^q(\g T)$, $A^q(\g T)\cap\SHI^\QG$ is empty or equal to $\A^\#$ (an affine subspace directed by $\vect{_KA^v({_K\g S})}\subset\vect{A^q(\g T)}$) and, for each $F^{vq}$ as above,
 the intersection $\overline  A^q(\g T)\cap\SHI^\QG_{F^{vq}}$ is empty or equal to $\overline A^\#\cap A^q(\g T)_{F^{vq}}$ (as the arguments in 3) give analogous results in the $Gal(L/K)-$stable fa\c{c}ades).

    \par So the image of $\SHI(\g G,K,\overline{A}^\natural)_{^K\!F^{vq}}$ in $\SHI(\g G,L,\overline\A^e)_{F^{vq}}$ is the set of $Gal(L/K)-$fixed points
 in the union of the apartments $A^q(\g T)_{F^{vq}}\subset\SHI(\g G,L,\A^q)_{F^{vq}}$ for $\g T$ a $L-$split maximal torus containing a maximal $M-$split torus, itself containing a maximal $K-$split torus.

 \subsection{More general relative apartments}\label{6.6} Most of the preceding arguments apply with a more general choice of apartments.
 We keep the hypotheses as in \ref{6.5}.1, but we choose for $\A$ one of the model apartments associated to $\g G$ and $\g T$ as in \ref{2.3b}.1 (\ie via a commutative extension $\qf:\shs\to\shs'$ of RGS) or eventually a quotient by a subspace $V_{00}$ of $V_0\subset V=Y'\otimes\R$.
 We suppose moreover $\shs'$ endowed with a star action of $\QG$ for which $\qf$ is $\QG^*-$equivariant and $V_{00}$ $\QG^*-$stable; \cf remark \ref{1.6} and the choice made in \ref{1.8}.1.
 We write $\A^v$ the corresponding vectorial apartment in $\vect\A=V/V_{00}$ and $\overline\A$ one of the three associated bordered apartments.

 \par The Galois group $\QG$ acts on $\overline\SHI(\g G,L,\overline\A)$ and $\SHI^v(\g G,L,\A^v)$, \cf  \ref{6.2}.4.
  These actions are compatible with each other, with the $\g G(L)-$actions and the essentialization maps $\eta:\overline\SHI(\g G,L,\overline\A)\to\overline\SHI(\g G,L,\overline\A^e)=\overline\SHI$, $\eta^v:\SHI^v(\g G,L,\A^v)\to\SHI^v(\g G,L,\A^{vq})=\SHI^{vq}$.
   We define  $_K\A=\eta^{-1}(A^\#)^\QG$; it is an affine space directed by $(\vect{F_0^v})^\QG=\vect{_K\A^v}$ (where $F_0^v=(\eta^v)^{-1}(F_0^{vq})$ and $A^\#$, $F_0^{vq}$ are as in \ref{6.5}.3).
   The group $_KN$ acts on $_K\A$,  we write $\qn_K$ this action.

   \par We choose $_K\A$ as model relative apartment. We may suppose $_K\!\A\subset\A$, but then $\A$, as apartment in $\SHI(\g G,L,\A)$, is non necessarily $\QG-$stable.
    We choose in $_K\A$ a special origin $x_0$ \ie its image by $\eta$ is the special point in $A^\#$ chosen as origin in \ref{6.4}.2 to define the valuation $_K\qf=\qf^\natural$ of $(\g G(K),(V_{_K\qa})_{_K\qa\in{_K\QF^{re}}},{_KZ})$.
    For $x\in{_K\A}$ we define $\widehat P_K(x)=\widehat P(x)\cap\g G(K)$.

    \par The (real) walls in $_K\A$ are the inverse images by $\eta$ of the walls in $A^\#$ defined in \ref{6.4}.3b, \ie they are described as
    $M_K({_K\qa},{_K\qf_\qa}(u))=\{x\in{_K\A}\mid{_K\qa}(x-x_0)+{_K\qf_\qa}(u)=0\}$
    for ${_K\qa}\in{_K\QF}$ and $u\in V_{_K\qa}\setminus\{1\}$; their set is written $\shm_K$.
    Note that, even if $\A=\A^q$ is essential, $_K\A$ may be inessential (as essentiality does not involve the imaginary walls defined below).

    We consider the set  $\shm_{L/K}\cup\shm_{L/K}^i$  of the non trivial traces on $_K\A$ of the real or imaginary walls of $\A$.
    More precisely if $M(\qa,\ql)\in \shm_{L}\cup\shm_{L}^i$ is such a wall and $\qa\in\QF$, $_K\qa=\qa\rest{\g S}\in{_K\QF}$ (resp. $\qa\in\QD$, $_K\qa=\qa\rest{\g S}\in{_K\QD}\setminus{_K\QF}$) then $M_K({_K\qa},\ql)=M(\qa,\ql)\cap{_K\A}$ is a real (ghost) wall (resp. an imaginary wall) and we write $M_K({_K\qa},\ql)\in \shm_{L/K}$ (resp. $M_K({_K\qa},\ql)\in \shm_{L/K}^i$).
    By remark \ref{6.4}.3 $\shm_K\subset \shm_{L/K}$.

    \par We define the enclosure map $cl^{_K\QD^r}_{L/K}$ as in \ref{2.4}.1:
    it is associated to $\shm_{L/K}$ and the subset $\shm_{L/K}^{ir}$ of $\shm_{L/K}^i$ containing the imaginary walls which are almost real \ie of direction $Ker({_K\qa})$ with ${_K\qa}\in{_K\QD^r}={_K\QD}\cap(\sum_{\qg\in{_K\QF}}\,\R\qg)\subset {_K\QD}$, \cf \ref{1.13}.3b.
    By \ref{2.4}.1 $cl^{_K\QD^r}_{L/K}$ does not change if we replace $\shm_{L/K}^i$ by $\shm_{L/K}^{i\R}=\{M_K({_K\qa},\ql)\mid {_K\qa}\in{_K\QD^{im}},\ql\in\R\}$ (and $\SHM^{ir}_{L/K}$ by the set $\SHM^{ir\R}_{L/K}$ of all walls parallel to a wall in $\SHM^{ir}_{L/K}$).
    A more precise enclosure map  $cl^{_K\QD^r}_{K}$ associated to $\shm_{K}$ and a subset of  $\shm_{L/K}^{i\R}$ will be introduced in \ref{6.10a}.
\begin{prop}\label{6.7} In the above situation, we have:

\par a) The action $\qn_K$ is affine and $_KN\subset N.\widehat P(_K\A)\subset N.{_KZ}$. In particular for $n\in{_KN}$, the linear map associated to $\qn_K(n)$ is $\qn_K^v(n)\in{_KW^v}={_KN}/{_KZ}$.

\par b) The group $_KZ$ acts on $_K\A$ by translations.
 More precisely for $z\in{_KZ}$, the vector $\qn_K(z)$ of this translation is the class modulo $V_{00}$ of a vector $\tilde\qn_K(z)\in V$ which satisfies  the formula:
 $\chi_1(\tilde\qn_K(z))=-\qo_K(\chi_2(z))$, for any $\chi_1\in X'\subset V^*$ and $\chi_2$ in the group $X(\g Z)$ of characters of the reductive group $\g Z_g(\g S)$ with the condition that $\chi_2$ and $\qf^*(\chi_1)\in X(\g T)$ coincide on $\g S$.

\par As $X(\g Z)$ is identified by restriction to a finite index subgroup of $X(\g S)$, this formula determines completely $\tilde\qn_K(z)$ and $\qn_K(z)$.

\par c) For any real relative root $_K\qa$ and $u\in V_{_K\qa}\setminus\{1\}$ $u$ fixes the half apartment $D_K({_K\qa},{_K\qf_\qa}(u))$ $=\{x\in{_K\A}\mid{_K\qa}(x-x_0)+{_K\qf_\qa}(u)\geq{}0\}$ and $\qn_K(m_K(u))$ is the reflection $r_{{_K\qa},{_K\qf_\qa}(u)}$ with respect to the wall $M_K({_K\qa},{_K\qf_\qa}(u))=\partial D_K({_K\qa},{_K\qf_\qa}(u))$.

\par d) If moreover $_K\qa$ is non multipliable, $m_K(u)^2={_K\qa^\vee}(-1)$ and $m_K(u)^4=1$.

\par e) $\qn_K({_KN})$ is a semi-direct product of $\qn_K({_KZ})$ by a subgroup fixing $x_0$ and isomorphic, via $\qn^v_K$, to ${_KW^v}={_KN}/{_KZ}$.

\par f) The action of ${_KN}$ on the closure $\eta^{-1}(\overline{A^\#})^\QG$ of ${_K\A}$ in $\overline\A$ is deduced from its actions on ${_K\A}$ and ${_K\A}^v$ : $\qn_K(n).pr_{_KF^v}(x)=pr_{\qn_K^v(n)({_KF^v})}(\qn_K(n).x)$, for $n\in{_KN}$, $x\in{_K\A}$ and ${_KF^v}$ a K-facet in ${_K\A}^v$.
\end{prop}

 \begin{NB} The equations defining $\vect{_K\A^v}$ in $\vect{\A^v}$ are in $Q$ and correspond (via $bar$) to the equations defining $_KY=Y(\g S)$ in $Y=Y(\g T)$ \ie $\g S$ in $\g T$, \cf \ref{1.8}.4 an \ref{1.9}.2.
 So the formula in b) above defines a vector $\qn_K(z)\in \vect{_K\A^v}=\vect{_K\A}$.
  Moreover $\qn_K(z)$ is in the image of the map $Y(\g S)\otimes\R\into Y(\g T)\otimes\R\stackrel{\qf}{\to} V\to V/V_{00}$ (analogous to the map in \ref{1.4d}).
 \end{NB}

 \begin{proof} With the notations as in \ref{1.12}.1, let $j\in{_KI}_{re}$; then $J={_K\underline{\{j\}}}=I_0\cup\{i\in I\mid\QG^*i=j\}$ is spherical.
  So $\g G_{\shs(J)}$ is a reductive group, containing $\g Z_g(\g S)=\g G_{\shs(I_0)}$ and defined over $K$; we write $\g G_J$ the corresponding $K-$subgroup-scheme of $\g G$.
  By \ref{5.8}.2 the extended Bruhat-Tits building $\SHI(\g G_J,L,\A)$ embeds in the hovel $\SHI(\g G,L,\A)$: the way we have chosen $\A$ ensures us that $\A$ is really endowed with the same action of the normalizer of $\g T$ in $\g G_J(L)$ as in the case of an extended Bruhat-Tits building \cite[{\S{}}  2.1]{Ru-77}. Moreover the actions of $\QG$ are compatible.

  \par As the classical construction of $\SHI(\g G_J,K,\A)$ uses the same methods as in \ref{6.5} above, we know that a), b) and c) are satisfied for $_KN\cap\g G_J(K)$ and ${_K\qa}=\pm{}{_K\qa}_j,\pm{}2{_K\qa}_j$ \cite[5.1.2]{Ru-77}.
  So b) is completely proved. Now ${_KW^v}={_KN}/{_KZ}$ is generated by simple reflections in $({_KN}\cap\g G_J(K))/{_KZ}$ for $j\in{_KI}_{re}$ (as $_KZ\subset\g G_J(K)$, $\forall j$).
  So a) is satisfied and also c) as any ${_K\qa}\in{_K\QF}_{re}$ is conjugated by $_KW^v$ to some $\pm{}{_K\qa}_j$ or $\pm{}2{_K\qa}_j$.

  \par Let ${_K\qa}$ and $u$ be as in d); we choose $s\in\g S(K_s)$ such that $_K\qa(s)=-1$.
  By \cite[7.2 (2)]{BlT-65}, $m_K(u)^2=m_K(u).s.m_K(u)^{-1}.s^{-1}=r_{_K\qa}(s).s^{-1}={_K\qa^\vee}(-1)$; so d) follows.
  As $x_0$ was chosen special, $\forall i\in{_KI}_{re}$, $\exists u_i\in V_{_K\qa_i}\setminus\{1\}$ with ${_K\qf_{\qa_i}}(u_i)=0$ hence $m_K(u_i)$ fixes $x_0$.
  So the subgroup fixing $x_0$ in e) is the image by $_K\qn$ of the subgroup of $_KN$ generated by the $m_K(u_i)$ and e) follows from a) and b).

  \par We know that, for  the action $\qn$ of $N$ on $\overline\A$, $\qn(n).pr_{F^v}(x)=pr_{\qn^v(n)({F^v})}(\qn(n).x)$; so f) follows from a).
 \end{proof}

\subsection{Embeddings of bordered apartments}\label{6.8}

\par{\bf\quad1)} To define the bordered apartment $_K\overline\A$, we always choose the vectorial Weyl$-K-$facets in $^K\A^v$ (as for $\overline A^\natural$ in \ref{6.4}.2 but differently from  \ref{6.7}f).
 We still have three choices  for $_K\overline\A$ (as in the general definition \ref{2.5}.2): $_K\overline{\overline\A}$ (resp. $_K\overline\A^e$) is the disjoint union of the inessential fa\c{c}ades $_K\A^{ne}_{^KF^v}={_K\A}$
 (resp. the essential fa\c{c}ades $_K\A^{e}_{^KF^v}={_K\A}/\vect{^KF^v}$) for ${^KF^v}$ a Weyl$-K-$facet in $^K\A^v$, and $_K\overline\A^i$ differs from $_K\overline\A^e$ only by its main fa\c{c}ade which is the inessential one.
 A Weyl$-K-$facet $^KF^v$ contains a unique maximal $K-$facet $_KF^v_{max}$ which is open in  $^KF^v$, hence $\vect{_KF^v_{max}}=\vect{^KF^v}$.
  So $_K\A^{e}_{^KF^v}$ is equal to $_K\A^{e}_{_KF^v_{max}}$. Now the proposition \ref{6.7}f tells us that the action $\qn_K$ of $_KN$ on $_K\A$ extends naturally to $_K\overline\A$ ($={_K\overline\A^e}$, or $_K\overline\A^i$ or $_K\overline{\overline\A}$).

  \par{\bf2)} For any choice of $\A$ (suitable for $\g G$ and $L$), we chose a unique $_K\A$ (inside $\A$ for some embedding).
  So it is interesting to define a good choice for $_K\overline\A$ for each choice of $\overline\A$.
  And it is natural to choose $_K\overline\A^i$ (resp. $_K\overline\A^e$, $_K\overline{\overline\A}$) when $\overline\A=\overline\A^i$ (resp. $_K\overline\A^e$, $_K\overline{\overline\A}$).
  Then we have a $_KN-$equivariant embedding $_K\overline\A\into\overline\A$ defined as follows on each fa\c{c}ade:
  for $^KF^v$ a vectorial Weyl$-K-$facet, let $F^v$ be the facet in $\A^v$ containing ${_KF^v_{max}}$, then $_K\A^{ne}_{^KF^v}={_K\A}\into\A=\A^{ne}_{F^v}$ and $_K\A^{e}_{^KF^v}={_K\A}/\vect{_KF^v_{max}}\into\A^{e}_{F^v}=\A/\vect{F^v}$.

  \par Note that the main fa\c{c}ade does not embed in general in the main fa\c{c}ade when we choose $_K\overline\A^e$ (as was the case for Charignon, \cf \ref{6.4}.3e).
  Moreover, if $_KF^v$ is positive and negative, the definition of $_KF^v_{max}$ may  include a choice of sign. For example the main fa\c{c}ade $_K\A^e$ of $_K\overline\A^e$ may embed in $\A^{e}_{F^v}$ or $\A^{e}_{-F^v}$ (they are equal but included separately in $\overline\A$).

    \par{\bf3)} For $x\in{_K\overline\A}$, more precisely $x\in{_K\A^{(n)e}_{^KF^v}}$, we define $\widehat P_K(x)=\widehat P(x)\cap\g G(K)$ where $x$ is considered in $\A^{(n)e}_{F^v}$ as above.
    This coincides with the above definition for $x\in{_K\A}$ and it is compatible with the projections: $\widehat P_K(x)\subset \widehat P_K(pr_{^KF^v_1}(x))$.

\begin{theo}\label{6.9} We suppose that the Kac-Moody group $\g G$ satisfies  the condition (TRQS) of \ref{6.5} and we keep the notations as in  \ref{6.5} to  \ref{6.8}.
See in particular \ref{6.6} for $_K\A$, $_K\qf$ and $cl^{_K\QD^r}_{L/K}$.

\par a) The family $_K\qf$ is a valuation for the root datum $(\g G(K),(V_{_K\qa})_{{_K\qa}\in{_K\QF}},{_KZ})$.

\par b) The family $\widehat\shp_K=(\widehat P_K(x))_{x\in{_K\overline\A}}$ is a very good family of parahorics.

\par\noindent We write $\SHI(\g G,K,{_K\A})$ (resp. $\overline\SHI(\g G,K,{_K\overline\A})$)  the corresponding  hovel (resp. bordered hovel).

\par c) The family $\widehat\shp_K$ is compatible with the enclosure map $cl^{_K\QD^r}_{L/K}$:
 $\SHI(\g G,K,{_K\A})$ is a parahoric hovel of type $({_K\A},cl^{_K\QD^r}_{L/K})$, in particular  $\g G(K)$ acts on it strongly transitively by vectorially Weyl automorphisms.

 \par d) The $_KN-$equivariant embedding $_K\overline\A\into{\overline\A}$ may be extended uniquely in a $\g G(K)-$equivariant embedding $\overline\SHI(\g G,K,{_K\overline\A})\into\overline\SHI(\g G,L,{\overline\A})$. Its image is in $\overline\SHI(\g G,L,{\overline\A})^\QG$.

 \par e) If the valuation $\qo_K$ of $K$ is discrete, then $_K\A$ (or $\SHI(\g G,K,{_K\A})$) is semi-discrete: in $\SHM_{L/K}$ or $\SHM_{K}$ the set of walls of given direction is locally finite.

 \par f) The hovel $\SHI(\g G,K,{_K\A})$ is thick: for any wall $M\in\SHM_K$, there are three half-apartments $D_1,D_2,D_3$ in $\SHI$ with boundary $M$ and  such that $D_i\cap D_j=M$ for $i≠j$. 
 Moreover the set of chambers adjacent to a chamber $C$ along a panel in a wall $M_K({_K\qa},k)$ with ${_K\qa}\in{_K\QF}$ non divisible, is in one to one correspondence with a finite dimensional vector space over the residue field $\qk$ of $K$.
\end{theo}

\begin{defi*} $\SHI(\g G,K,{_K\A})$ (resp. $\overline\SHI(\g G,K,{_K\overline\A})$) is the {\it affine hovel} (resp. {\it affine bordered hovel}) {\it of $\g G$ over $K$ with model apartment $\A$} (resp. $\overline\A$).
\end{defi*}

\begin{rema*} $\SHI(\g G,K,{_K\A})$ is the main fa\c{c}ade of $\overline\SHI(\g G,K,{_K\overline\A})$ for ${_K\overline\A}={_K\overline\A}^i$ or ${_K\overline{\overline\A}}$.
 By the definition of $_K\A$ in \ref{6.6} and of $A^\#$ in \ref{6.5}.4, the image of $\SHI(\g G,K,{_K\A})$ in $\SHI(\g G,L,{\A})$ is the set of $\QG-$fixed points in the union of the apartments $A(\g T)\subset\SHI(\g G,L,{\A})$ for $\g T$ a $L-$split maximal torus containing a maximal $M-$split torus, itself containing a maximal $K-$split torus.
\end{rema*}

\begin{proof} a) The family $_K\qf$ is actually defined by the essentialization of $_K\A$. So it is a valuation by \ref{6.4}.3c.

\par b) The family $\widehat\shp_K$ satisfies clearly to axioms (P1), (P2), (P4), (P5) and (P10). (P3) is proved in \ref{6.7}c for the main fa\c{c}ade; the result is analogous in the other fa\c{c}ades and the link is made by \ref{6.7}f.

\par If $^K\!F^v(x)$ is spherical, then the corresponding facet $F^v$ (as in \ref{6.8}.2) is spherical and $_K\A_{{^KF^v}(x)}$ embeds in $\A_{F^v}$ which is an apartment in the Bruhat-Tits building $\SHI(\g G,L,\overline{\A})_{F^v}$ for the reductive group $P(F^v)/U(F^v)\simeq M(F^v)$.
 Moreover $_K\A_{{^KF^v}(x)}$ is chosen in $\SHI(\g G,L,\overline{\A})_{F^v}$ as in \ref{6.5}.3 \ie as in the descent theorems of (extended) Bruhat-Tits buildings.
 So $\widehat P_K(x)$ is generated by $_KN(x)$ and the $V_{_K\qa}\cap\widehat P_K(x)$ for $_K\qa\in{_K\QF}$ and (P6) is satisfied.

 \par For $x\in{_K\A_{^KF_1^v}}$ and ${^KF_1^v}\subset\overline{^KF^v}$ we write $F^v_{1max}$ and $F^v_{max}$ the corresponding maximal facets in $\A^v$.
  Then $\widehat P_K(x)\cap P_K(^KF^v)\subset\g G(K)\cap\widehat P(x)\cap P(F^v_{max})\subset\g G(K)\cap\widehat P(\overline{x+F^v_{max}})\subset\g G(K)\cap\widehat P(\overline{x+F^v_{max}\cap{_K\A^v}})=\g G(K)\cap\widehat P(\overline{x+{_KF^v}})$
  with $_KF^v=F^v_{max}\cap{_K\A^v}\subset{_KF^v}$.
  But $_KF^v+\vect{_K\A^v}_0={^KF^v}$ and the "torus" $S_Z$ in the center of $\g G(K)$ (\ref{1.13}.2) acts on $_K\A$ as a group (of translations) $\vect T$ generating $\vect{_K\A^v}_0$;
  so $\widehat P_K(x)\cap P_K(^KF^v)=\widehat P_K(\cup_{\qt\in\vect T}\,\overline{x+\qt+{_KF^v}})=\widehat P_K(\overline{x+{^KF^v}})$ and (P9) is satisfied.

 \par For $x\in\overline\A$ and $^KC^v$ a chamber in $^KF^v(x)^*$, we have by \ref{6.4}.3d $\widehat P_K(x)\subset\widehat P_K(\eta x)=(\widehat P_K(\eta x)\cap U({^KC^v})).(\widehat P_K(\eta x)\cap U(-{^KC^v})).{_KN}(\eta x)$.
  We know by construction that $\widehat P(\eta x)\cap U(\pm{}{C^v})=\widehat P( x)\cap U(\pm{}{C^v})$ (\cf \ref{5.1} to \ref{5.3}) for any chamber $C^v$ in $F^{v*}$ ($F^v$ as above) \eg $C^v$ containing $_KC^v={^KC^v}\cap{_K\A^v}$.
  So $\widehat P(\eta x)\cap U(\pm{}{^KC^v})=\widehat P( x)\cap U(\pm{}{^KC^v})$ and $\widehat P_K(\eta x)\cap U(\pm{}{^KC^v})\subset\widehat P_K(x)$.
  Hence $\widehat P_K(x)=(\widehat P_K(x)\cap U({^KC^v})).(\widehat P_K(x)\cap U(-{^KC^v})).(\widehat P_K(x)\cap{_KN}(\eta x))$ and $\widehat P_K(x)\cap{_KN}={_KN}(x)$. So (P8) is satisfied.

  \par In the situation of (P7), let $B=\{x,pr_{F^v}(x)\}$.
  We saw above that $\widehat P_K(x).{_KN}=\widehat P_K(\eta x).{_KN}$.
  So, by \ref{6.4}.3d, $\widehat P_K(x).{_KN}\cap P_K({^KF^v}).{_KN}\subset \widehat P_K(\eta B).{_KN}$.
  Let $^KC^v$ be such that $^KF^v\subset\overline{^KC^v}$, then arguing as in \cite[4.3.4]{GR-08} we see that \label{N29}
  $\widehat P_K(\eta B)=
  [\widehat P_K(\eta B)\cap U({^KC^v})].[\widehat P_K(\eta B)\cap U(-{^KC^v})].{_KN(\eta B)}
  \subset  [\widehat P_K(B)\cap U({^KC^v})].[\widehat P_K(B)\cap U(-{^KC^v})].{_KN}$.
  So (P7) follows.

  \par c) Let $\QO$ be a non empty filter in $_K\A_{^K\!F^v}$ and  $^KC^v$ a chamber in $ ({^KF^v})^*$.
  We consider  the facet  $F_1^v$ in $\A^v$ such that $_KF_1^v=F_1^v\cap{_K\A^v}$ is open in $^KC^v$ and a chamber $C^v\in(F_1^v)^*$.
  Then $\widehat P_K(\QO)\cap U(\pm{}{^KC^v})\subset\g G(K)\cap\widehat P(\QO)\cap U(\pm{}{C^v})\subset\g G(K)\cap\widehat P(cl^\QD(\QO))\subset\g G(K)\cap\widehat P(cl^{_K\QD}_{L/K}(\QO))$,
  where $cl^\QD(\QO)$ (resp. $cl^{_K\QD}_{L/K}(\QO)$) is the enclosure of $\QO$ in $\A$ (resp. $_K\A$) for the root system $\QD$ (resp. for $_K\QD$ and $\shm_{L/K}$, $\shm_{L/K}^i$).
  We use once more the torus $S_Z$ in the center of $\g G(K)$: we have $cl^{_K\QD}_{L/K}(\cup_{\qt\in\vect T}\,\QO+\qt)=cl^{_K\QD^r}_{L/K}(\QO)$.
  So $\widehat P_K(\QO)\cap U_K(\pm{}{_KC^v})\subset\widehat P_K(cl^{_K\QD^r}_{L/K}(\QO))$.
  Hence $\widehat\shp_K$ is compatible with $cl^{_K\QD^r}_{L/K}$ and c) is a consequence of \ref{3.12}.5.

  \par d) The existence of a unique $P_K({^KF^v})-$equivariant map $\SHI(\g G,K,{_K\overline\A})_{^KF^v}\to \SHI(\g G,L,{\overline\A})_{F^v}$ extending $\A_{^KF^v}\into \A_{F^v}$ is an easy consequence of the definitions of $\qn_K$ and $\widehat P_K(x)$:
  $_KN\subset N.\widehat P({_K\A})$, "$\qn_K=\qn\,\rule[-1.5mm]{.1mm}{3mm}_{\,{_KN}}$" and $\widehat P_K(x)=\g G(K)\cap\widehat P(x)$.
  As for \ref{6.4}.3e we conclude with \ref{3.14}.

  \par e) As $\qo_K$ is discrete and $L/K$ finite, $\qo_L$ is discrete. Suppose $_K\qa\in{_K\QF}$ non divisible, then the walls in $\shm_{L/K}$ of direction Ker$({_K\qa})$ are the traces \label{N30} of walls in $\A$ of direction Ker$\qb$ for $\qb\in\QF$ with ${_K\qb}={_K\qa}$ or ${_K\qb}=2.{_K\qa}$.
  There is only a finite number of such $\qb$ and, for each $\qb$, $\QL_\qb=\qo_L(L^*)$ is discrete.
  So the set of these walls of direction Ker$({_K\qa})$ is locally finite.

  \par f) The first assertion (about thickness) is a simple consequence of \ref{6.7}.c.
  We write $a={_K\qa}$. By \ref{3.11}e this set of chambers is in one to one correspondence with $V_{a,k}/V_{a,k+}$.
  Suppose $2a\not\in{_K\QF}$, \label{N31} then by \ref{6.4}.2, $V_{a,k}=V_a\cap(\prod_{{_K\qb}=a}\,U_{\qb,k})$ is an $\sho_K-$module and $V_{a,k+}=V_a\cap(\prod_{{_K\qb}=a}\,U_{\qb,k+})$  an $\sho_K-$submodule such that $\g m_K.V_{a,k}\subset V_{a,k+}$.
  So $V_{a,k}/V_{a,k+}$ is a $\qk-$vector space of dimension $\leq{}dim_K(V_a)=\vert(a)\vert$.
  When $2a\in{_K\QF}$ we prove, the same way, that $V_{a,k}/V_{a,k+}$ is a group extension of two $\qk-$vector spaces of dimensions at most $\vert(2a)\vert$ and $\vert(a)\vert-\vert(2a)\vert$ \cf \ref{1.10}.
  To see that $V_{a,k}/V_{2a,2k}$ is an $\sho_K-$submodule of $\prod_{{_K\qb}=a}\,U_{\qb,k}\subset V_{a,L}/V_{2a,L}$, we may use the coroot $(2a)^\vee$ in $Y(\g S)$; as $a((2a)^\vee)=1$ the exterior multiplication by $K\setminus\{0\}$ in $V_{a,L}/V_{2a,L}$ is given by the action of the torus $\g S(K)$.
\end{proof}

\subsection{Remarks}\label{6.10}

\par{\quad\bf1)} The condition (TRQS) is certainly non necessary for the existence of an hovel $\SHI(\g G,K,{_K\A})$; the existence of this hovel for any almost split Kac-Moody group $\g G$ (over a complete field) seems a reasonable conjecture, as in the classical (= reductive) case.
 On the contrary the existence of a $\g G(K)-$equivariant embedding of $\SHI(\g G,K,{_K\A})$ in $\SHI(\g G,L,{_L\A})$ for any extension $L/K$ seems to necessitate (TRQS) or $\qo_K$ discrete.
  And the functoriality of these embeddings seems to necessitate (TRQS).
   There are counter-examples even in the classical case \cite[3.5.9 and 3.4.12a]{Ru-77}.


 \par{\bf2)} We chose to define the fa\c{c}ades of $_K\overline\A$ and $\overline\SHI(\g G,K,{_K\overline\A})$ using the Weyl$-K-$facets as indexing set.
 This is more natural for the bordered hovel of the root datum $(\g G(K),(V_{_K\qa})_{{_K\qa}\in{_K\QF}},$ ${_KZ})$ but a definition using $K-$facets seems richer. This is largely an illusion:

 \par Let $^KF^v\subset{^K\A^v}$ be a Weyl$-K-$facet and ${_KF^v_{}}$ (resp. ${_KF^v_{min}}$, ${_KF^v_{max}}$) be any (resp. the minimal, maximal) $K-$facet in $_K\A^v$ corresponding to $^KF^v$ and ${F^v_{}}$ (resp. ${F^v_{min}}$, ${F^v_{max}}$) the corresponding facet in $\A^v$:
 hence $\overline{_KF^v_{min}}\subset\overline{_KF^v_{}}\subset\overline{_KF^v_{max}}$ and $\overline{F^v_{min}}\subset\overline{F^v_{}}\subset\overline{F^v_{max}}$.
 The $K-$fa\c{c}ade $_K\A^{ne}_{_KF^v}={_K\A}$ or $_K\A^{e}_{_KF^v}={_K\A}/\vect{_KF^v}$ is endowed with a system of relative real roots $_K\QF^m({_KF^v})$ (and even a system of relative almost real roots $_K\QD^{rm}({_KF^v})$) which is independent of the choice of $_KF^v$.
 So $_K\A^{ne}_{_KF^v}$ and the essentialization of $_K\A^{e}_{_KF^v}$ do not depend of the choice: we have projections maps $_K\A={_K\A^{ne}_{_KF^v}}\to{_K\A^{e}_{_KF^v_{min}}}\to{_K\A^{e}_{_KF^v_{}}}\to{_K\A^{e}_{_KF^v_{max}}}={_K\A^{e}_{^KF^v_{}}}$ where the last term is the essentialization of the first three
 (actually $\vect{_KF^v}$ is in general non enclosed, as ${_KF^v}$ is defined by inequalities involving $_K\QD$, and not only $_K\QF$ or $_K\QD^r$).

\par We saw in \ref{1.12}.3c that the fixer $P_K({_KF^v})$ of ${_KF^v}$ in $\g G(K)$ is independent of the choice of ${_KF^v}$.
 In particular the above maps are $_KN\cap P_K({_KF^v})-$equivariant.
 Moreover the fixer of a point $x$ in an apartment is included in the fixer of the image of $x$ in another apartment.
 So we have corresponding projection maps between the fa\c{c}ades corresponding to these $K-$facets: $\SHI(\g G,K,{^K\overline{\overline\A}})_{^KF^v}\to\SHI(\g G,K,{_K\A})_{_KF^v_{min}}\to\SHI(\g G,K,{_K\A})_{_KF^v_{}}\to\SHI(\g G,K,{_K\A})_{_KF^v_{max}}=\SHI(\g G,K,{_K\overline\A}^e)_{^KF^v_{}}$ and the last hovel is the essentialization of the first three .

 \par Hence these hovels are more or less the same and it is not really interesting to include all of them in a bordered hovel.
  Perhaps the only interesting thing to do could be to define a fourth bordered apartment $_K\overline\A^{min}$ associated to $_K\A$ with ${_K\A^{min}_{^KF^v_{}}}={_K\A_{_KF^v_{min}}}$ and $\SHI(\g G,K,{_K\A^{min}})_{^KF^v_{}}=\SHI(\g G,K,{_K\A^{}})_{_KF^v_{min}}$.
  Then $\overline\SHI(\g G,K,{_K\overline\A^{min}})$ coincides with $\overline\SHI(\g G,K,{_K\overline\A^{e}})$ when $\g G$ is split over $K$ (or if $_KI_{re}={_KI}$ \ie $_K\QD={_K\QD^r}$).

 \par{\bf3)} The microaffine building of a split Kac-Moody group $\g G_\shs$ over a "local" field is defined in \cite{Ru-06}.
  In its Satake realization [\lc 4.2.3] it is the union $\overline\SHI_{\!sph}(\g G_\shs,K,\overline\A^e)$ of the spherical fa\c{c}ades in the essential bordered hovel $\overline\SHI_{}(\g G_\shs,K,\overline\A^e)$.
  Hence, as explained in this section, Charignon proved the existence of such a microaffine building for any almost split Kac-Moody group satisfying (TRQS).
  This building satisfies clearly the functorial properties proved below for bordered hovels.

\subsection{The enclosure map $cl^{_K\QD^r}_{K}$}\label{6.10a}

\par{\bf\quad1) Imaginary walls:} We defined in \ref{1.13}.1 a subgroup scheme $\g U_{({_K\qa})K_s}^{ma}$ of $\g U_{K_s}^{ma±}$ associated to a root ${_K\qa}\in{_K\QD}_±^{im}$.
 It is clear that $U_{({_K\qa})}^{ma}=\g U_{({_K\qa})K_s}^{ma}(K_s)$ is stable under the action of the Galois group $Gal(K_s/K)$ on $\g G^{pma}$ or $\g G^{nma}$ explained in \ref{1.14}.
 \label{N33} We define $\widehat V_{_K\qa}=(U_{({_K\qa})}^{ma})^{Gal(K_s/K)}$ hence $V_{_K\qa}=\widehat V_{_K\qa}\cap G(K_s)=\widehat V_{_K\qa}\cap G(K)=U_{(_K\qa)}^{ma}\cap G(K)$.
 
 \par An element $u$ of $U_{({_K\qa})}^{ma}$ may be written as an infinite product $u=\prod_{\qb\in({_K\qa})}\,u_\qb$ with $u_\qb=\prod_{j=1}^{j=n_\qb}\,[exp](\ql_{\qb,j}e_{\qb,j})$.
 Then, for a set $\QO\subset\A$, we have $u\in U^{ma}_\QO(({_K\qa}))$ if and only if, $\forall\qb\in({_K\qa})$, $\QO\subset D(\qb,inf\{\qo(\ql_{\qb,j})\mid j=1,\cdots,n_\qb\}-\qb(x'_0))$ where $x'_0$ is the (old) origin in $\A$, see \cite[4.5]{Ru-11}.
 So, for $\QO$ in $_K\A$, $u\in U^{ma}_\QO(({_K\qa}))$ if and only if $\QO\subset D_K({_K\qa},\qf_{_K\qa}(u))$ with $\qf_{_K\qa}(u)=inf\{\frac{1}{m_\qb}(\qo(\ql_{\qb,j})-\qb(x'_0))\mid \qb\in({_K\qa}), {_K\qb}=m_\qb.{_K\qa},  j=1,\cdots,n_\qb\}$.
 As $_K\qa\in{_K\QD^{im}}$, $({_K\qa})$ is infinite hence $\qf_{_K\qa}(U_{({_K\qa})}^{ma})=\R\cup\{±\infty\}$.
 
 \par We define the set $\shm_K^{i\R}$ of imaginary walls in $_K\A$ as the set of hyperplanes $M_K({_K\qa},\qf_{_K\qa}(u))$ for $_K\qa\in{_K\QD^{im}}$, $u\in \widehat V_{_K\qa}$ and $\qf_{_K\qa}(u)≠±\infty$.
 For $L$ as in \ref{6.5}.1, we have $\widehat V_{_K\qa}\subset \g U^{ma}_{({_K\qa})}(L)$ and $\shm_K^{i\R}\subset \shm_{L/K}^{i\R}$.
 We have $\shm_K^{i\R}= \shm_{L/K}^{i\R}$ in many cases \eg, when $\g G$ is split over $K$, $\shm_K^{i\R}$ ($\supset\shm_K^{i}$) is the set of true or ghost imaginary walls.
  We do not define in general the analogue of $\shm_K^{i}$.

\par{\bf2) Enclosure map:} The enclosure map $cl^{_K\QD^r}_{K}$ in $_K\A$ is associated to $\shm_K$ and the subset $\shm_K^{ir\R}$ of $\shm_K^{i\R}$ containing the imaginary walls in $\shm_K^{i\R}$ of direction $Ker({_K\qa})$ with $_K\qa\in{_K\QD^r}$.
 More precisely for $_K\qa\in{_K\QD^r}$, we set $\QL'_{_K\qa}=\qf_{_K\qa}(\widehat V_{_K\qa})\setminus \{±\infty\}$ and then $cl^{_K\QD^r}_{K}=cl^{_K\QD^r}_{\QL'}$ with the notations of \ref{2.4}.1.
 
 \par For any $\QO$ in $_K\A$, we have $cl^{_K\QD^r}_{K}(\QO)\supset cl^{_K\QD^r}_{L/K}(\QO)$.
 When $\g G$ is split over $K$, ${_K\QD^r}=\QD$ and $cl^{_K\QD^r}_{K}=cl^\QD$.
 
 \par This enclosure map $cl^{_K\QD^r}_{K}$ is more natural than $cl^{_K\QD^r}_{L/K}$ as it involves the set of true real relative walls $\SHM_K$ instead of $\SHM_{L/K}$ (and nothing else in the classical case).

\begin{enonce*}[plain]{\quad3) Proposition} The family $\widehat\shp_K$ of \ref{6.8}.3 is compatible with the enclosure map $cl^{_K\QD^r}_{K}$. In particular $\SHI(\g G,K,{_K\A})$ is a parahoric hovel of type $({_K\A},cl^{_K\QD^r}_{K})$.
\end{enonce*}

\begin{proof} Let $\QO$ be a non empty filter in a façade $_K\A_{^KF^v}$ of $_K\A$.
 We choose a Weyl$-K-$chamber $^KC^v$ containing $^KF^v$ and then $F_1^v$, $C^v$ as in \ref{1.14}.
 We have to prove that $\widehat P_K(\QO)\cap U(±{^K\!C^v})\subset\widehat P_K(cl^{_K\QD^r}_{K}(\QO))$.
 By \ref{5.3} and \ref{6.8} we may replace $\QO$ by its inverse image in $_K\A$, hence suppose $\QO\subset\A$.
 But $\widehat P_K(\QO)\cap U(±{^K\!C^v})=\widehat P_K(\QO)\cap U(±{F_1^v})=\g G(K)\cap\widehat P(\QO)\cap U(±{F_1^v})=\g G(K)\cap U_\QO^{ma}(\QF^u(±{F_1^v}))$, \cf \cite[4.5]{Ru-11}.
 Let $_K\QD^±_{nd}$ be the set of non divisible real or imaginary relative roots in $_K\QD^±$.
 Then by construction, $\g U^{ma}(\QF^u(±{F_1^v}))(L)$ (resp. $U_\QO^{ma}(\QF^u(±{F_1^v}))$) may be written uniquely as a product  $\prod_{{_K\qa}\in{_K\QD}^±_{nd}}\,\g U_{({_K\qa})}^{ma}(L)$ (resp. $\prod_{{_K\qa}\in{_K\QD}^±_{nd}}\,U_{\QO}^{ma}({_K\qa})$) where actually $\g U_{({_K\qa})}^{ma}=\g U_{({_K\qa})}$ when $_K\qa\in{_K\QF}$.
 Each subgroup $\g U_{({_K\qa})}^{ma}(L)$ is stable under the Galois group $\QG$, hence 
 $\g G(K)\cap U_\QO^{ma}(\QF^u(±{F_1^v}))\subset U_\QO^{ma}(\QF^u(±{F_1^v}))^\QG=\prod_{{_K\qa}\in{_K\QD}^±_{nd}}\,U_{\QO}^{ma}({_K\qa})\cap\widehat V_{_K\qa}$ (with $\widehat V_{_K\qa}= V_{_K\qa}$ for $_K\qa\in{_K\QF}$).
 Now the definition in 1) above of $\qf_{_K\qa}$ and $\shm_K^{i\R}$, together with \ref{6.7} and the definition of $\shm_K$, prove that $U_\QO^{ma}(\QF^u(±{F_1^v}))^\QG=U_{cl^{_K\QD}_{K}(\QO)}^{ma}(\QF^u(±{F_1^v}))^\QG$ with an obvious definition of $cl^{_K\QD}_{K}$.
 So $\widehat\shp_K$  is compatible with  $cl^{_K\QD}_{K}$ and the same arguments as in \ref{6.9}c prove the compatibility with $cl^{_K\QD^r}_{K}$.
\end{proof}

\subsection{Functoriality}\label{6.11}

\par{\quad\bf1) Changing the group, commutative extensions:} We consider a morphism $\psi:\g G\to\g G'$ between two almost split Kac-Moody groups and we suppose that, over $K_s$, $\psi=\g G_\qf:\g G_\shs\to\g G_{\shs'}$ for a commutative extension of RGS $\qf:\shs\to\shs'$.
 This extension is then automatically $Gal(K_s/K)^*-$equivariant.

 \par The conditions (TRQS) for $\g G$ and $\g G'$ are equivalent: $\psi$ induces a bijection between  the combinatorial vectorial buildings of $\g G$ and $\g G'$ over $K_s$ \cite[1.10]{Ru-11} which is clearly $Gal(K_s/K)-$equivariant; so $\g G$ has a Borel subgroup defined over a field $M\in\shs ep(K)$ if and only if $\g G'$ has one.

 \par Suppose (TRQS), then $\g G'$ and $\g G$ are quasi-split over a tamely ramified finite Galois extension $M/K$ and split over a  finite Galois extension $L/K$ with $L\supset M$.
 We choose an apartment $\A$ for $\g G'$ as in \ref{6.8}, hence associated to a morphism $\qf':\shs'\to\shs''$ of RGS and some $V_{00}$ in $V=Y''\otimes\R$ compatible with the star action of $\QG=Gal(L/K)$ associated to $\g G'$.
  Then the same thing is true for $\qf'\circ\qf$ and $\g G$.
  Now the constructions of $\overline\SHI(\g G,K,{_K\overline\A})$ inside $\overline\SHI(\g G,L,{\overline\A})$ or of $\overline\SHI(\g G',K,{_K\overline\A})$ inside $\overline\SHI(\g G',L,{\overline\A})$ are completely parallel.
  So the $\psi_L-$equivariant morphism $\overline\SHI(\psi_L,L,{\overline\A}):\overline\SHI(\g G,L,{\overline\A})\to \overline\SHI(\g G',L,{\overline\A})$ of \ref{5.8}.1 induces a $\psi_K-$equivariant morphism $\overline\SHI(\psi_K,K,{_K\overline\A}):\overline\SHI(\g G,K,{_K\overline\A})\to \overline\SHI(\g G',K,{_K\overline\A})$.

  \par This is functorial (up to the problem that $\A$ or $\overline\A$ has sometimes to change with $\g G'$).

 \par{\bf2) Changing the group, Levi factors:} Suppose that $\g G$ satisfies (TRQS) and let $M,L,\QG,\cdots$ be as in \ref{6.5}.

 \par Let $F^v_+$ and $F^v_-$ be opposite $\QG-$stable vectorial facets in $\SHI^v(\g G,L,\A^v)$.
  They determine completely a subgroup $M(F^v)$ in $\g G(L)$ which is $\QG-$stable. We write $\g G'=\g G_{F^v_\pm{}}$ the corresponding subgroup-scheme of $\g G$.
  We know that, over $L$, $\g G'$ is isomorphic to some $\g G_{\shs(J)}$.

  \par The parabolic subgroup-scheme $\g P(F^v)$ of $\g G_L$ associated to $F^v$ is defined over $K$, hence over $M$, and contains a minimal $M-$parabolic \ie a Borel subgroup defined over $M$.
   The parabolics in $\g P(F^v)$ correspond bijectively to the parabolics of its Levi factor $\g G'$ and this correspondence is $\QG-$equivariant as $\g G'$ is $\QG-$stable.
   So $\g G'$ is quasi-split over $M$: it satisfies (TRQS).

   \par If $\A$ is chosen as in \ref{6.6} for $\g G$, then it satisfies the same conditions for $\g G'\subset\g G$.
   Here also the constructions of the bordered hovels over $K$ inside the bordered hovels over $L$ for $\g G$ and $\g G'$ are completely parallel.
   We deduce from \ref{5.8}.2 a $\g G'(K)-$equivariant isomorphism of $\SHI(\g G',K,{_K\A})$ with the fa\c{c}ade $\SHI(\g G,K,{_K\overline{\overline\A}})_{^KF^v}$ (where $^KF^v$ is the Weyl$-K-$facet corresponding to $F^v$) or with $\SHI(\g G',\g G,K,{_K\A})=\g G'(K).({_K\A})\subset \SHI(\g G,K,{_K\A})$.

   \par The reader will write the results for bordered hovels analogous to those in \ref{5.8}.2.

   \par{\bf3) Changing the field:} We asked  in \ref{6.1} that the valuation $\qo=\qo_K$ of $K$ may be extended functorially to all extensions in $\shs ep(K)$.
   We ask also that the almost split Kac-Moody group $\g G$ satisfies (TRQS), hence is quasi-split over a tamely ramified finite Galois extension $M/K$ and split over a  finite Galois extension $L/K$ with $L\supset M$.

   \par Let's consider now a  field extension $i:K\into K'$ in $\shs ep(K)$. We define in $K_s$, $L'=K'L$ and $M'=K'M$; we write $i_L:L\into L'$.
   The extensions $L'/K'$ and $L'/M'$ are Galois, $Gal(L'/K')\subset Gal(L/K)$, $Gal(M'/K')\subset Gal(M/K)$ and $M'/K'$ is tamely ramified.
   Moreover $\g G$ is split on $L'$ and quasi-split on $M'$, so $\g G$ satisfies (TRQS) on $K'$.

   \par We saw in \ref{5.8}.3 that $\A$ (with some added walls) is still a suitable apartment for $(\g G,\g T)$ over $L'$ and there is a $\g G(L)-$equivariant embedding $\SHI(\g G,i_L,\A):\SHI(\g G,L,\A)\into \SHI(\g G,L',\A)^{}$.
    This embedding is also $Gal(L'/K')-$equivariant.
    Now $ \SHI(\g G,K,{_K\A})^{}\subset\SHI(\g G,L,{\A})^{Gal(L/K)}$ and $ \SHI(\g G,K',{_{K'}\A})^{}\subset\SHI(\g G,L',{\A})^{Gal(L'/K')}$.
  Moreover $\SHI(\g G,K,{{_K}\A})$ is the union of the apartments $\A_\g S=\SHI(\g Z_g(\g S),\g G,K,{{_K}\A})$ for $\g S$ a maximal $K-$split torus in $\g G$ and $\g Z_g(\g S)$ its generic centralizer, which is a reductive group. By 2) above and \cite[5.12]{Ru-77},
  $\SHI(\g G,i_L,{{}\A})(\A_\g S)=\SHI(\g Z_g(\g S),i_L,{{}\A})(\SHI(\g Z_g(\g S),\g G,K,{{_K}\A})
  \subset \SHI(\g Z_g(\g S),\g G,K',{{_{K'}}\A})\subset \SHI(\g G,K',{{_{K'}}\A})$, where $_{K'}\A$ is associated to a maximal $K'-$split torus $\g S'$ containing $\g S$.
     We have thus defined a $\g G(K)-$equivariant embedding $\SHI(\g G,i,\A):\SHI(\g G,K,{{_K}\A})\into \SHI(\g G,K',{_{K'}\A})^{}$.

     \par This is clearly functorial.
  We leave to the reader the "pleasure" to formulate a result for bordered hovels; there is the problem of the choice of the fa\c{c}ade of $_{K'}\A$ in which embeds a fa\c{c}ade of $_K\A$.
  This is easier for the essential spherical fa\c{c}ades \ie for the microaffine buildings.
  
  \par Note that the (real) walls in $_K\A$ are some of the traces on $_K\A$ of the (real) walls in $_{K'}\A$. In general any such trace is not necessarily a wall in $_K\A$, see nevertheless \ref{6.12} below.

 \par{\bf4) Changing the model apartment:} Suppose that $\g G$ satisfies (TRQS) and let $M,L,\QG,\cdots$ be as in \ref{6.5}.

 \par The apartment $\A$ is associated to a commutative extension $\qf:\shs\to\shs'$ of RGS and a subspace $V_{00}'$ of $V_0'\subset V'=Y'\otimes\R$ with the condition that $\shs'$ is endowed with a star action of $\QG$ for which $\qf$ is $\QG^*-$equivariant and $V_{00}'$ $\QG^*-$stable.

 \par Now let $\psi:\shs'\to\shs''$ be a commutative extension of RGS and $V_{00}''$ a subspace  of $V_0''\subset V''=Y''\otimes\R$ containing $\psi(V'_{00})$, with the condition that $\shs''$ is endowed with a star action of $\QG$ for which $\psi$ is $\QG^*-$equivariant and $V_{00}''$ $\QG^*-$stable.
 Then $\qf'=\psi\circ\qf:\shs\to\shs''$ satisfies the above condition and can be used to define a new apartment $\A'=V''/V_{00}''$.
  We have an affine map $\psi:\A\to\A'$ which is $N-$equivariant.

  \par By \ref{5.8}.4 we get a $\g G(L)-$equivariant map $\SHI(\g G,L,\psi):\SHI(\g G,L,\A)\to\SHI(\g G,L,\A')$.
  It is $\QG-$equivariant and induces a $\g G(K)-$equivariant map $\SHI(\g G,K,\psi):\SHI(\g G,K,{_K\A})\to\SHI(\g G,K,{_K\A}')$ (by the characterization given in remark \ref{6.9}).

  \par This construction is functorial and extends clearly to the bordered hovels.
  
\begin{prop}\label{6.12} In the situation of \ref{6.11}.3 above, suppose the extension $K'/K$ Galois and unramified (more precisely for a non discrete valuation, $K'/K$ is supposed etale \cite[1.6]{BtT-84a}).
 Then the intersection with $_K\A$ of any real wall of $_{K'}\A$ is a real wall of $_K\A$ if (and only if) it is a hyperplane of direction given by a root in $_K\QF$.
\end{prop}

\begin{proof} Let $\QG=Gal(K'/K)$. Then with obvious notations, $\sho_{K'}$ is a free $\sho_K-$module with basis a family $x_1,\cdots,x_n$ whose image in $\qk'=\sho_{K'}/\g m_{K'}$ is a basis over $\qk=\sho_{K}/\g m_{K}$; moreover $\qk'/\qk$ is Galois and $\QG=Gal(\qk'/\qk)$ \cite[1.6.1d]{BtT-84a}.
 If $\QG=\{\qg_1,\cdots,\qg_n\}$, then a well known theorem tells us that $det(\qg_i(x_j))$ is non trivial in $\qk'$, hence is in $\sho_{K'}^*$. \label{N34}
  An easy consequence is that any $\sho_{K'}-$module $M$ with a conjugate-linear action of $\QG$ is the $\sho_{K'}-$module generated by  $M^\QG$.
  
  \par Let $a\in{_K\QF}$ and $x\in{_K\A}$ in the wall $M(a',k)$ of $_{K'}\A$ for $a'\in{_{K'}\QF}$, $a'\rest\g S=a$ and $k=-a(x)$.
  With notations as in the proof of \ref{6.9}f, we have $V_{a',k}/(V_{2a',2k}.V_{a',k+})$ (or $V_{2a',2k}/V_{2a',2k+}$)  non trivial (where $V_{a',k},\cdots$ is relative to $K'$ and $\g S'$) and we want to prove that $V_{a,k}/V_{2a,2k}.V_{a,k+}$ (or $V_{2a,2k}/V_{2a,2k+}$) is non trivial.
  We concentrate on the first case, the second is easier.
  
  \par We set $V'_{2a,2k}=\prod_{b'\rest\g S=2a}\,V_{b',k}$, $V'_{a,k}=V'_{2a,2k}.\prod_{b'\rest\g S=a}\,V_{b',k}$ and analogous formulae for $V'_{2a,2k+}$, $V'_{a,k+}$.
  By hypothesis $V'_{a,k}/(V'_{2a,2k}.V'_{a,k+})$ is non trivial.
  But $V'_{a,k}/V'_{2a,2k}$ (resp. $V'_{a,k+}/V'_{2a,2k+}$) is an $\sho_{K'}-$module stable under $\QG$ and $(V'_{a,k}/V'_{2a,2k})^\QG=V_{a,k}/V_{2a,2k}$ (resp. $(V'_{a,k+}/V'_{2a,2k+})^\QG=V_{a,k+}/V_{2a,2k+}$). \label{N35}
  If $V_{a,k}/(V_{2a,2k}.V_{a,k+})$ were trivial, then we would have $V_{a,k}/V_{2a,2k}=V_{a,k+}/V_{2a,2k+}$ and, by the above result, $V'_{a,k}/V'_{2a,2k}=V'_{a,k+}/V'_{2a,2k+}$, nonsense.
\end{proof}



\medskip

\medskip

Guy.Rousseau@univ-lorraine.fr

\medskip
\parni 1. Universit\'e de Lorraine, Institut \'Elie Cartan de Lorraine, UMR 7502,  Vand\oe uvre-l\`es-Nancy, F-54506, France.
\parni 2. CNRS, Institut \'Elie Cartan de Lorraine, UMR 7502,  Vand\oe uvre-l\`es-Nancy, F-54506, France.

\end{document}

\bigskip
\par NOTES (2015)
\bigskip 

\par\noindent Note 1 : \ref{N1}; Note 2 : \ref{N2}; Note 3 : \ref{N3}; Note 4 : \ref{N4}; Note 5 : \ref{N5}; Note 6 : \ref{N6}; Note 7 : \ref{N7}; 
\par\noindent Note 8 : \ref{N8}; Note 9 : \ref{N9}; Note 10 : \ref{N10};  Note 11 : \ref{N11}; Note 12 : \ref{N12}; Note 13 : \ref{N13}; 
\par\noindent Note 14 : \ref{N14}; Note 15 : \ref{N15}; Note 16 : \ref{N16}; Note 17 : \ref{N17}; Note 18 : \ref{N18}; Note 19 : \ref{N19}; 
\par\noindent Note 20 : \ref{N20}; Note 21 : \ref{N21};  Note 22 : \ref{N22}; Note 23 : \ref{N23}; Note 24 : \ref{N24}; Note 25 : \ref{N25};
\par\noindent Note 26 : \ref{N26}; Note 27 : \ref{N27}; Note 28 : \ref{N28};   Note 29 : \ref{N29}; Note 30 : \ref{N30}; Note 31 : \ref{N31};
\par\noindent Note 32 : \ref{N32};   Note 33 : \ref{N33}; Note 34 : \ref{N34}; Note 35 : \ref{N35}; Note 36 : \ref{N36}; Note 37 : \ref{N37};

\end{document}